\documentclass[10pt]{article}
\usepackage[utf8]{inputenc}
\usepackage[margin = 1.75 cm]{geometry}

\usepackage{amsfonts,amsthm,enumitem,amsmath}
\usepackage{graphicx, caption}
\usepackage{amssymb}
\usepackage{amsthm}
\usepackage{slantsc}
\usepackage{soul}
\usepackage{changepage}
\usepackage{comment}
\usepackage{float}
\usepackage[pdftex,colorlinks,backref=page,citecolor=blue,bookmarks=false]{hyperref}
\usepackage{lscape}
\usepackage{mathtools}

\usepackage{appendix}

\usepackage{bbm}

\usepackage{setspace}
\usepackage{aliascnt}
\usepackage{cleveref}
\usepackage{graphicx}
\usepackage{mathrsfs}
\usepackage{algorithm}
\usepackage{algpseudocode}

\theoremstyle{plain}

\newtheorem*{thm*}{Theorem}
\newtheorem{theorem}{Theorem}[section]
\Crefname{theorem}{Theorem}{Theorems}

\newtheorem*{lem*}{Lemma}
\newaliascnt{lemma}{theorem}
\newtheorem{lemma}[lemma]{Lemma}
\aliascntresetthe{lemma}
\crefname{lemma}{Lemma}{Lemmas}
\Crefname{lemma}{Lemma}{Lemmas}

\newtheorem*{claim*}{Claim}
\newaliascnt{claim}{theorem}
\newtheorem{claim}[claim]{Claim}
\aliascntresetthe{claim}
\crefname{claim}{Claim}{Claims}
\Crefname{claim}{Claim}{Claims}

\newaliascnt{prop}{theorem}
\newtheorem{prop}[prop]{Proposition}
\aliascntresetthe{prop}
\crefname{prop}{Proposition}{Propositions}
\Crefname{prop}{Proposition}{Propositions}

\newaliascnt{fact}{theorem}
\newtheorem{fact}[fact]{Fact}
\aliascntresetthe{fact}
\crefname{fact}{Fact}{Facts}
\Crefname{fact}{Fact}{Facts}

\newaliascnt{corollary}{theorem}
\newtheorem{corollary}[corollary]{Corollary}
\aliascntresetthe{corollary}
\crefname{corollary}{Corollary}{Corollaries}
\Crefname{corollary}{Corollary}{Corollaries}

\newaliascnt{conj}{theorem}
\newtheorem{conj}[conj]{Conjecture}
\aliascntresetthe{conj}
\crefname{conj}{Conjecture}{Conjectures}
\Crefname{conj}{Conjecture}{Conjectures}

\newtheorem*{conj*}{Conjecture}

\newaliascnt{qn}{theorem}
\newtheorem{qn}[qn]{Question}
\aliascntresetthe{qn}
\crefname{qn}{Question}{Questions}
\Crefname{qn}{Question}{Questions}

\newaliascnt{obs}{theorem}
\newtheorem{obs}[obs]{Observation}
\aliascntresetthe{obs}
\crefname{obs}{Observation}{Observations}
\Crefname{obs}{Observation}{Observations}

\newaliascnt{ex}{theorem}

\aliascntresetthe{ex}
\crefname{ex}{Example}{Examples}
\Crefname{ex}{Example}{Examples}

\theoremstyle{definition}
\newaliascnt{prob}{theorem}
\newtheorem{prob}[prob]{Problem}
\aliascntresetthe{prob}
\crefname{prob}{Problem}{Problems}
\Crefname{prob}{Problem}{Problems}

\newaliascnt{defn}{theorem}

\aliascntresetthe{defn}
\crefname{defn}{Definition}{Definitions}
\Crefname{defn}{Definition}{Definitions}

\newtheorem*{defn*}{Definition}

\theoremstyle{remark}

\renewenvironment{proof}[1][]{\begin{trivlist}
\item[\hspace{\labelsep}{\bf\noindent Proof#1.\/}] }{\qed\end{trivlist}}

\newcommand{\remove}[1]{}

\newcommand{\cM}{\mathcal{M}}

\newcommand{\cK}{\mathcal{K}}

\renewcommand{\setminus}{-}

\newcommand{\bR}{\mathbb{R}}

\newcommand{\bN}{\mathbb{N}}

\newcommand{\ceil}[1]{\left\lceil #1 \right\rceil}
\newcommand{\floor}[1]{\left\lfloor #1 \right\rfloor}

\newcommand{\comp}[1]{\overline{#1}}
\newcommand{\one}{\mathbbm{1}}
\newcommand{\eps}{\varepsilon}

\makeatletter
\def\moverlay{\mathpalette\mov@rlay}
\def\mov@rlay#1#2{\leavevmode\vtop{%
   \baselineskip\z@skip \lineskiplimit-\maxdimen
   \ialign{\hfil$\m@th#1##$\hfil\cr#2\crcr}}}
\newcommand{\charfusion}[3][\mathord]{
    #1{\ifx#1\mathop\vphantom{#2}\fi
        \mathpalette\mov@rlay{#2\cr#3}
      }
    \ifx#1\mathop\expandafter\displaylimits\fi}
\makeatother

\newcommand{\cupdot}{\charfusion[\mathbin]{\cup}{\cdot}} 

\newcommand{\N}{\mathbb{N}}

\DeclareMathOperator{\TF}{TF}
\DeclareMathOperator{\TK}{TK}

\usepackage[square,sort,comma,numbers]{natbib}
\makeatletter
\newcommand{\optionaldesc}[2]{%
  \phantomsection
  #1\protected@edef\@currentlabel{#1}\label{#2}%
}
\makeatother

\DeclareMathOperator{\supp}{supp}

\title{Packing subgraphs in regular graphs}

\author{Shoham Letzter\thanks{Department of Mathematics, University College London, Gower Street, London, WC1E 6BT, U.K. Research supported by the Royal Society. Email:~\textbf{s.letzter@ucl.ac.uk}.} \and Abhishek Methuku\thanks{Department of Mathematics, University of Illinois at Urbana–Champaign, Urbana, IL, USA. Research supported by the UIUC Campus Research Board Award RB25050 and a Simons award SFI-MPS-TSM-00025583. Email:~\textbf{abhishekmethuku@gmail.com}} \and Benny Sudakov\thanks{Department of Mathematics, ETH, Z\"urich, Switzerland. Research supported in part by SNSF grant 200021-228014. Email:~\textbf{benjamin.sudakov@math.ethz.ch}}}
\date{}

\begin{document}

\maketitle
\begin{abstract}
    \noindent An \emph{$H$-packing} in a graph $G$ is a collection of pairwise vertex-disjoint copies of $H$ in $G$. We prove that for every $0<c<1$ and every bipartite graph $H$, every $d$-regular graph $G$ of order $n$ with $d\ge cn$ admits an $H$-packing that covers all but a constant number of vertices. This resolves a problem posed by Kühn and Osthus in 2005. Moreover, our result is essentially tight: the conclusion fails if $G$ is not both regular and sufficiently dense, it is in general not possible to guarantee covering all vertices of $G$ by an $H$-packing, and if $H$ is non-bipartite then $G$ need not contain any copies of $H$.

	We also prove that for all $0<c<1$, integers $t \geq 2$, and sufficiently large $n$, all the vertices of every $d$-regular graph of order $n$ with $d\ge cn$ can be covered by vertex-disjoint subdivisions of $K_t$. This resolves another problem of K\"uhn and Osthus from 2005, which goes back to a conjecture of Verstra\"ete from 2002.

	Our proofs combine novel methods for balancing expanders and super-regular subgraphs with a number of powerful techniques, including properties of robust expanders, the regularity lemma, and the blow-up lemma.
\end{abstract}

\section{Introduction}

	Given two graphs $F$ and $G$, an $F$-\emph{packing} (or an $F$-\emph{tiling}) in $G$ is a collection of pairwise vertex-disjoint copies of $F$ in $G$. An $F$-packing in $G$ is called \emph{perfect} if it covers all the vertices of $G$. Note that when $F$ consists of a single edge, an $F$-packing is a graph matching. Tutte’s theorem characterizes those graphs which have a perfect $F$-packing if $F$ is an edge, but no such characterisation is known for other connected graphs $F$. In fact, it is known~\cite{kirkpatrick1983complexity} that the decision problem of whether a graph $G$ has a perfect $F$-packing is NP-complete if and only if $F$ has a component which contains at least three vertices. So it is natural to seek simple sufficient conditions that guarantee the existence of a perfect $F$-packing in a given graph.

	A fundamental result in this area is the Hajnal--Szemer\'edi theorem~\cite{hajnal1970proof} from 1970 which states that every graph whose order $n$ is divisible by $t$ and whose minimum degree is at least $(1 - \frac{1}{t})n$ contains a perfect $K_t$-packing. For arbitrary graphs, Koml\'os, S\'ark\"ozy and Szemer\'edi~\cite{komlos2001proof} showed that for every graph $F$, there exists a constant $C = C(F)$ such that every graph $G$ whose order $n$ is divisible by $|V(F)|$ and whose minimum degree is at least $(1- \frac{1}{\chi(F)})n + C$ contains a perfect $F$-packing. It turns out that there are graphs $F$ for which the bound on the minimum degree can be improved significantly, i.e., one can often replace $\chi(F)$ by a smaller parameter called the \emph{critical chromatic number} $\chi_{cr}(F)$ of $F$, which is defined as $\chi_{cr}(F) \coloneqq (\chi(F)-1)\frac{|V(F)|}{|V(F)|-\sigma(F)}$. Here $\sigma(F)$ denotes the smallest size of a colour class in a proper $\chi(F)$-colouring of $F$. It is easy to see that $\chi(F) - 1 < \chi_{cr}(F) \le \chi(F)$. In 2000, Koml\'os~\cite{komlos2000tiling} showed that the critical chromatic number is the parameter that governs the existence of almost perfect packings in graphs of large minimum degree. Finally, in 2009, K\"uhn and Osthus~\cite{kuhn2009minimum} determined, up to an additive constant, the minimum degree of a graph $G$ that ensures the existence of a perfect $F$-packing in $G$, for every graph $F$. More precisely, they proved that for every graph $F$, either its critical chromatic number or its chromatic number is the parameter which governs the existence of perfect $F$-packings in graphs of large minimum degree (where the exact classification depends on a parameter called the \emph{highest common factor} of $F$).

	\subsection{$F$-packings in dense regular graphs}

		In view of the above results, rather surprisingly, K\"uhn and Osthus~\cite{kuhn2005packings} showed that if we restrict our attention to packings in regular graphs, then \emph{any} linear bound on the degree guarantees an almost perfect $F$-packing whenever $F$ is bipartite.\footnote{An $F$-packing in an $n$-vertex graph is called \emph{almost perfect} if it covers all but $o(n)$ vertices.}

		\begin{theorem}[K\"uhn and Osthus~\cite{kuhn2005packings}]
			\label{thm:almostregpackingKO}
			For every bipartite graph $F$, every $0<c<1$, and every $0<\alpha\le1$, there exists $n_0 = n_0(F, c, \alpha)$ such that every $d$-regular graph $G$ of order $n$, where $d \ge cn$ and $n \ge n_0$, has an $F$-packing that covers all but at most $\alpha n$ vertices of $G$.
		\end{theorem}

		It is easy to see that the restriction to (sufficiently) regular graphs $G$ in Theorem~\ref{thm:almostregpackingKO} is necessary. Indeed, for any fixed $\varepsilon>0$, if $n$ is sufficiently large and $G$ is the complete bipartite graph $K_{(1/2-\varepsilon)n,(1/2+\varepsilon)n}$ (with the part sizes rounded to integers) and $F$ is an edge, then every $F$-packing leaves at least $2\varepsilon n-O(1)$ vertices uncovered. The restriction to bipartite graphs $F$ in Theorem~\ref{thm:almostregpackingKO} is also necessary because if $c \le 1/2$, then $G$ could be a bipartite graph (in which case it cannot contain any subgraph $F$ which is non-bipartite).

		This raises the natural question of whether the bound $\alpha n$ on the number of uncovered vertices in Theorem~\ref{thm:almostregpackingKO} can be lowered significantly to obtain an $F$-packing which is close to being perfect. Indeed, K\"uhn and Osthus~\cite{kuhn2005packings} proposed the following general problem in 2005.

		\begin{prob}[K\"uhn and Osthus~\cite{kuhn2005packings}]
			\label{Problem1KO}
			Is it true that, for every bipartite graph $F$ and every $0<c<1$, there is a constant $C = C(F, c)$ such that every $d$-regular graph $G$ of order $n$, where $d \ge cn$, has an $F$-packing that covers all but at most $C$ vertices of $G$?
		\end{prob}

		K\"uhn and Osthus~\cite{kuhn2005packings} resolved Problem~\ref{Problem1KO} in the special case when the parts of $F$ have unequal sizes. More precisely, they showed the following.

		\begin{theorem}[K\"uhn and Osthus~\cite{kuhn2005packings}]
	\label{KOdifferentvertexclasses}
			For every bipartite graph $F$ whose vertex classes have different sizes and for every $0<c<1$, there is a constant $C = C(F, c)$ such that every $d$-regular graph $G$ of order $n$, where $d \ge cn$, has an $F$-packing which covers all but at most $C$ vertices of $G$.
		\end{theorem}

		Despite extensive research on graph packings, this problem (Problem~\ref{Problem1KO}) remained open for nearly twenty years. Our main result is a complete resolution of Problem~\ref{Problem1KO} as follows.

		\begin{theorem}
			\label{mainthm:packingsubgraphs}
			For every bipartite graph $F$ and every $0<c<1$, there is a constant $C = C(F, c)$ such that every $d$-regular graph $G$ of order $n$, where $d \ge cn$, has an $F$-packing that covers all but at most $C$ vertices of $G$.
		\end{theorem}

		As noted in~\cite{kuhn2005packings}, it is necessary to allow for uncovered vertices in \Cref{mainthm:packingsubgraphs}, and the upper bound $C$ on the number of uncovered vertices must depend on $c$ and $F$ (even if $n$ is divisible by $|V(F)|$), in sharp contrast to the result of Koml\'os, S\'ark\"ozy and Szemer\'edi~\cite{komlos2001proof} mentioned earlier. Indeed, for instance, already for a connected bipartite graph $F$, if $G$ is the vertex-disjoint union of cliques of order $k|V(F)|-1=d+1$ (for some positive integer $k$), then $G$ is $d$-regular and every $F$-packing leaves at least $(|V(F)|-1)\frac{n}{d+1}$ vertices uncovered. Moreover, this example also shows that $d$ must be linear in $n$ in order for the number of uncovered vertices to be bounded by a constant $C$, so the requirement that the vertex degrees of $G$ are linear in $n$ in Theorem~\ref{mainthm:packingsubgraphs} is also necessary.

Let us briefly explain why the balanced case (i.e., when the two vertex classes of $F$ have equal size) presents substantial difficulties for the strategy in \cite{kuhn2005packings}. Let $F$ be a bipartite graph with parts of sizes $s$ and $t$, where $s\le t$. It suffices to prove the result for $K_{s,t}$, since every copy of $K_{s,t}$ contains a spanning copy of $F$. In their proof of Theorem~\ref{KOdifferentvertexclasses}, K\"uhn and Osthus apply Szemer\'edi's Regularity Lemma to a spanning almost-regular bipartite subgraph $G^\ast$, find a large matching $M$ in the reduced graph, and for each matched edge $A_iB_j\in M$ pass to equal-sized subclusters $A_i',B_j'$ so that $(A_i',B_j')$ is super-regular. Vertices that do not lie in these chosen super-regular pairs (that is, the leftover vertices in the partition classes after passing to the subclusters) are collected into an exceptional set $A_0\cup B_0$. They then \emph{assign} each exceptional vertex $x\in A_0\cup B_0$ to a matched super-regular pair in which $x$ has many neighbours, and subsequently \emph{absorb} $x$ by choosing a copy of $K_{s,t}$ that contains $x$ and otherwise lies inside that pair, with all such absorbing copies chosen disjointly. The strict inequality $s<t$ is used precisely at this absorption step: by choosing the \emph{orientation} of each absorbing $K_{s,t}$ across a fixed pair $(A_i',B_j')$, one can remove $t$ vertices from one side and $s$ from the other (or vice versa), and hence select the absorbing copies so that the leftover subsets $A_i^\ast\subseteq A_i'$ and $B_j^\ast\subseteq B_j'$ satisfy $\bigl|\,|A_i^\ast|-|B_j^\ast|\,\bigr|\le t-s$ for every matched pair. This bounded discrepancy is what enables the subsequent application of the blow-up lemma to obtain an almost-perfect tiling inside each super-regular pair.

When $s=t$, however, this balancing mechanism disappears---every copy of $K_{s,s}$ uses the same number of vertices from each side---so the quantity $|A_i^\ast|-|B_j^\ast|$ cannot be adjusted during absorption. Consequently, if a matched pair is left noticeably unbalanced after distributing the exceptional vertices $A_0\cup B_0$ among the pairs and removing the vertices used for absorption, then any $K_{t,t}$-packing inside it must leave at least $\bigl|\,|A_i^\ast|-|B_j^\ast|\,\bigr|$ vertices uncovered, and the proof strategy in \cite{kuhn2005packings} can no longer yield a packing covering all but at most a constant number of vertices in the balanced case. We also emphasise that when $s=t$, the regularity assumption on the host graph $G$ is unavoidable: already a small imbalance in the bipartition of a complete bipartite host graph can force every $K_{t,t}$-packing to leave more than $C$ vertices uncovered. Hence, controlling irregularity while constructing the desired $K_{t,t}$-packing constitutes a significant additional difficulty compared to the unbalanced case. In this paper we therefore adopt a novel approach, based on robust expansion together with several new balancing procedures for expanders, as explained below.

First, it suffices to prove \Cref{mainthm:packingsubgraphs} when the graph being packed is $K_{t,t}$. Indeed, if the two vertex classes of $F$ have sizes $a$ and $b$, set $t=a+b$. Then two vertex-disjoint copies of $F$, with their two vertex classes placed in opposite orientations, partition the vertices of $K_{t,t}$. Thus, replacing every copy in a $K_{t,t}$-packing by these two copies of $F$ gives an $F$-packing with exactly the same uncovered vertices. To prove \Cref{mainthm:packingsubgraphs} for $K_{t,t}$, we start by partitioning the vertex set of our dense regular graph $G$ into a small number of expanders using a structural decomposition result from the work of Gruslys and Letzter \cite{gruslys2021cycle}, which goes back to the work of K\"uhn, Lo, Osthus and Staden~\cite{kuhn2015robust, kuhn2016solution} on Hamilton cycles in dense regular graphs. These expanders are, in turn, closely related to a powerful notion called \emph{robust expanders} that was introduced by Kühn, Osthus and Treglown \cite{kuhn2010hamiltonian} and has since been used to prove several longstanding conjectures (see, e.g., \cite{kuhn2011proof, kuhn2011approximate, kuhn2013hamilton}).

		One of the key contributions of our paper is a novel method for `balancing' these expanders. By this, we mean transforming expanders that are bipartite (or close to being bipartite) with imbalanced part sizes into \emph{balanced} bipartite expanders. This method thus enables us to obtain expanders that are either balanced bipartite or far from bipartite by removing a carefully constructed, small $K_{t,t}$-packing.
		Our main idea is to track a certain subgraph, which initially consists of the edges between expanders. We then sequentially move vertices between expanders and update the subgraph accordingly, ensuring that by the end it has bounded maximum degree while still containing enough edges to admit a $K_{t,t}$-packing that balances the expanders (up to an additive constant, which is sufficient for our purposes).
		Note that while balancing expanders also plays a key role in~\cite{kuhn2015robust, kuhn2016solution}—in resolving a problem of Bollobás and H\"aggkvist on Hamilton cycles in regular graphs—and in~\cite{gruslys2021cycle}, which proves a conjecture of Magnant and Martin~\cite{magnant2009note} for dense graphs, our approach differs significantly from the methods used in these works, so it is of independent interest. See Section~\ref{subsec:balancingexpanders} for further details on our balancing procedure.

		After balancing the expanders, the original problem reduces to finding a $K_{t,t}$-packing covering all but a constant number of vertices, either in a balanced bipartite expander or in an expander that is far from bipartite. To address the latter case, our first goal is to apply the regularity lemma (\Cref{lem:regularity}) to obtain vertex-disjoint super-regular pairs covering all but a small set of exceptional vertices. For this, rather than directly seeking a perfect matching in the reduced graph, we exploit the Hamiltonicity of expanders to find a perfect \emph{fractional} matching (inspired by \cite{kelly2008dirac, keevash2009exact, korandi2021minimum}). This fractional matching is used to construct a collection of edges and odd cycles covering all vertices of the reduced graph. We then refine the clusters by splitting them each into two equal parts and then remove a small number of vertices from the new clusters to obtain vertex-disjoint super-regular pairs covering all but a small exceptional set of vertices, as required.

		Second, we cover the exceptional vertices with a $K_{t,t}$-packing, while also ensuring that none of the super-regular pairs are overused. Third, we construct another small $K_{t,t}$-packing, whose purpose is to ensure that the number of remaining vertices in each part of every super-regular pair is divisible by $t$. Ideally, at this point, one could attempt to complete the proof by applying the blow-up lemma to find a perfect $K_{t,t}$-packing covering the uncovered portion of each super-regular pair. However, this is not immediately feasible because the number of remaining vertices in the two parts of a super-regular pair may differ.
		To address this, we introduce another novel balancing strategy: we construct a suitable matching within a small subgraph induced by the uncovered vertices, once again exploiting the Hamiltonicity of expanders. This matching serves as a ‘template’ for constructing a carefully designed balancing $K_{t,t}$-packing, whose removal allows us to find a perfect $K_{t,t}$-packing covering all remaining vertices using the blow-up lemma (\Cref{lem:blowup}), as intended. We give a detailed sketch of this balancing strategy in Section~\ref{subsec:proofsketchmatchingtemplate}, and a more detailed outline of our methods in Section~\ref{sec:proosketchmain1}.

	\subsection{Packing subdivisions in dense regular graphs}

		Given a graph $F$, a \emph{subdivision} of $F$, denoted by $\TF$, is a graph obtained by replacing the edges of $F$ with pairwise internally vertex-disjoint paths between the corresponding ends, whose interiors avoid the vertices of $F$. In this case, we refer to the vertices of $F$ as the \emph{branch vertices} of $\TF$. For instance, a subdivision $\TK_t$ of the complete graph of order $t$ consists of $t$ (branch) vertices $\{v_1, \ldots, v_t\}$ and $\binom{t}{2}$ pairwise internally vertex-disjoint paths $P_{i,j}$, $1 \le i < j \le t$, such that $P_{i,j}$ joins $v_i$ and $v_j$ and avoids all other vertices in $\{v_1, \ldots, v_t\}$.

		The notion of subdivisions has played an important role in topological graph theory since the seminal result of Kuratowski~\cite{kuratowski1930probleme} from 1930 showing that a graph is planar if and only if it does not contain a subdivision of $K_5$ or $K_{3,3}$. One of the most classical results in this area is due to Mader~\cite{MR0220616}, who showed that
		there is some $d = d(t)$ such that every graph with average degree at least $d$ contains a
		subdivision of the complete graph $K_t$. Mader~\cite{MR0220616}, and independently Erd\H{o}s and Hajnal~\cite{MR0173247}, conjectured that $d(t)=O(t^2)$. In the 1990s, Koml\'os and Szemer\'edi~\cite{komlos1994topological, komlos1996topological}, and independently Bollob\'as and Thomason~\cite{bollobas1998proof}, confirmed this conjecture. Since then, various extensions and strengthenings of this result have been studied, see, e.g., \cite{fernandez2023disjoint, liu2017proof, montgomery2015logarithmically}.

		Given graphs $F$ and $G$, a \emph{$\TF$-packing} in $G$ is a collection of pairwise vertex-disjoint copies of subdivisions of $F$ in $G$ (which are not required to be isomorphic). In 2002, Verstra\"ete~\cite{verstraete2002note} made the following conjecture.

		\begin{conj}[Verstra\"ete~\cite{verstraete2002note}]
			\label{conj:packingsubdivisons}
			For every graph $F$ with at least one edge and every $\eta>0$, there exists $d_0=d_0(F,\eta)$ such that, for all integers $d \ge d_0$, every $d$-regular graph $G$ of order $n$ contains a $\TF$-packing which covers all but at most $\eta n$ vertices of $G$.
		\end{conj}

		Quite a lot of research has established Conjecture~\ref{conj:packingsubdivisons} in several natural special cases.
		The conjecture holds when $F$ is a tree by a result of Kelmans, Mubayi and Sudakov~\cite{kelmans2001asymptotically}, and in 2003, Alon~\cite{alon2003problems} proved it when $F$ is a cycle.
		In 2005, K\"uhn and Osthus~\cite{kuhn2005packings} observed that when $G$ is dense (i.e., $d = \Omega(n)$) the conjecture follows from their results on packings in dense regular graphs.
		Making significant progress towards the general case, Letzter, Methuku and Sudakov~\cite{LetzterMetukuSudakov} applied their work on nearly-Hamilton cycles in sublinear expanders to show that the conjecture holds for all graphs $F$ when $G$ has degree at least polylogarithmic in $n$.
		Finally, very recently, Montgomery, Petrova, Ranganathan and Tan~\cite{montgomery2025packing} resolved the conjecture in full.

		For subdivisions of complete graphs $K_t$, K\"uhn and Osthus~\cite{kuhn2005packings} went even further and showed that one can actually guarantee perfect packings in dense regular graphs for $t = 4$ and $t = 5$. More precisely, they showed that for every $0<c<1$, there is a positive number $n_0 = n_0(c)$ such that every $d$-regular graph $G$ of order $n$, with $d \ge cn$ and $n \ge n_0$, has a perfect $\TK_t$-packing for $t = 4$ and $t = 5$.
		By a result of Gruslys and Letzter \cite{gruslys2021cycle}, it follows that this also holds for $t = 2$ and $t = 3$.

		In 2005, K\"uhn and Osthus asked whether the same holds for $t \ge 6$.

		\begin{prob}[K\"uhn and Osthus~\cite{kuhn2005packings}]
			\label{prob:TKt}
			Given $t \ge 6$ and $0<c<1$, does every $d$-regular graph $G$ of order $n$, where $d \ge cn$ and $n$ is sufficiently large, have a perfect $\TK_t$-packing?
		\end{prob}

		K\"uhn and Osthus's result \cite{kuhn2005packings} actually guarantees a perfect $\TK_t$-packing, with $t \in \{4,5\}$, even when $G$ is only assumed to be almost regular. In contrast, for $t = 3$ or $t \ge 6$, the restriction to regular graphs $G$ in \Cref{prob:TKt} is necessary. Indeed, the complete bipartite graph $K_{m,m+1}$ does not have a perfect $\TK_3$-packing: a $\TK_3$-packing is a collection of pairwise vertex-disjoint cycles, and every cycle uses the same number of vertices from the two parts. K\"uhn and Osthus showed that $K_{m, m+1}$ also does not have a perfect $\TK_t$-packing if $t \ge 6$ (see Proposition~5.1 in \cite{kuhn2005packings}).

		Our second result is the complete resolution of Problem~\ref{prob:TKt}. In fact, we prove a stronger result, showing that there is a perfect $\TF$-packing for every graph $F$ with at least one edge.

		\begin{theorem} \label{thm:Kt-subdivision-dense}
			For every graph $F$ with at least one edge and $0<c<1$, there is a positive number $n_0=n_0(F,c)$ such that every $d$-regular graph $G$ of order $n$, where $d \ge cn$ and $n \ge n_0$, has a perfect $\TF$-packing.
		\end{theorem}

		It is unclear whether the requirement that the vertex degrees in $G$ are linear in $n$ in Theorem~\ref{thm:Kt-subdivision-dense} is necessary. In this direction, K\"uhn and Osthus~\cite{kuhn2005packings} gave an example showing that for all $t \ge 3$ the vertex degrees in $G$ must be at least $\sqrt{n/2}$ to guarantee a perfect $\TK_t$-packing in $G$, even if $G$ is regular.
		Interestingly, the case $t=2$ is different. Every $d$-regular graph admits a perfect $\TK_2$-packing, i.e., its vertex set can be partitioned into vertex-disjoint nontrivial paths (that is, paths containing at least one edge). (In fact, if $d$ is even, then by Petersen's 2-factor theorem~\cite{Petersen1891} it also admits a perfect $\TK_3$-packing.) It is therefore natural to ask how few paths are required to achieve such a partition.
		In this direction, Magnant and Martin \cite{magnant2009note} conjectured that the vertices of every $d$-regular graph on $n$ vertices can be covered by at most $n / (d+1)$ pairwise vertex-disjoint paths. This conjecture of Magnant and Martin was confirmed for $d \le 5$ by Magnant and Martin themselves and for dense graphs $G$ by Gruslys and Letzter \cite{gruslys2021cycle}.
		Montgomery, M\"uyesser, Pokrovskiy and Sudakov \cite{montgomery2024approximate} showed that almost all vertices of every $n$-vertex $d$-regular graph can be covered by at most $n/(d+1)$ vertex-disjoint paths. Very recently, Christoph, Dragani\'c, Gir\~ao, Hurley, Michel and M\"uyesser~\cite{christoph2025new} proved that every $d$-regular graph can be covered by at most $2n/(d+1)$ vertex-disjoint paths, thereby proving the Magnant--Martin conjecture up to a constant factor and confirming a conjecture of Feige and Fuchs (see also~\cite{christoph2025cycle}).

		To prove \Cref{thm:Kt-subdivision-dense}, we begin by partitioning the vertex set of the dense regular graph $G$ into a small number of expanders, following the approach used in the proof of \Cref{mainthm:packingsubgraphs}. In this setting, however, we employ a lemma from~\cite{gruslys2021cycle} that provides a small linear forest $H$ that balances the expanders, ensuring that each expander contains either zero or two leaves of $H$.
		We then modify this linear forest so that at most one component of the new linear forest is associated with each expander, with both leaves of the component contained within that expander. Next, we construct a pair of small, vertex-disjoint $F$-subdivisions in each expander, while also ensuring that their union remains balanced. Finally, we \emph{absorb} the components of $H$, along with any remaining uncovered vertices in each expander, into these $F$-subdivisions. This absorption step again relies on the Hamiltonicity of expanders and a robust connectivity property, which guarantees that any two vertices can be joined by a short path that avoids a small set of forbidden vertices. See Section~\ref{sec:Kt-subdivisions-dense} for a more detailed sketch of these ideas.

		\subsection{Organisation of the paper}
			In the next subsection, we present the notation used throughout the paper.
			In Section~\ref{sec:proosketchmain1}, we outline a detailed proof sketch of our main result, \Cref{mainthm:packingsubgraphs}, on almost perfect packings in dense regular graphs. In Section~\ref{sec:tools} we collect some tools and lemmas that are used throughout the paper. Key to our proof of \Cref{mainthm:packingsubgraphs} are two lemmas -- \Cref{lem:balancing} (for balancing expanders) and \Cref{lem:Ktt-packing-expander} (for packing $K_{t,t}$'s in expanders). We prove \Cref{lem:balancing} in Section~\ref{sec:balancing} and \Cref{lem:Ktt-packing-expander} in Section~\ref{sec:packings}, and complete the proof of \Cref{mainthm:packingsubgraphs} in \Cref{sec:proof}. In Section~\ref{sec:Kt-subdivisions-dense}, we prove our second result, \Cref{thm:Kt-subdivision-dense}, which addresses perfect packings of subdivisions in dense regular graphs. A detailed sketch of its proof is given at the beginning of Section~\ref{sec:Kt-subdivisions-dense}.

	\subsection{Notation}
			Let $[m]$ denote the set $\{1,\ldots,m\}$. We write $c = a \pm b$ if $a-b \le c \le a+b$. We use the `$\ll$' notation to state many of our results. We write $a \ll b$ to mean that there exists a non-decreasing function $f : (0, 1] \to (0, 1]$ such that the subsequent result holds for all $a,b \in (0,1]$ with $a \le f(b)$. We will not calculate these functions explicitly. Hierarchies with more constants are defined in a similar way and are to be read from right to left.

		For a graph $G$, let $e(G)$ denote the number of edges in $G$ and, for two sets $X, Y \subseteq V(G)$, we let $G[X, Y]$ denote the subgraph of $G$ with the vertex set $X \cup Y$, whose edges are the edges of $G$ with one end in $X$ and the other end in $Y$, and we let $e_G(X,Y)$ denote the number of edges in $G[X, Y]$. If $X = Y$, for convenience, we write $G[X]$ instead of $G[X,X]$, and $e_G(X)$ instead of $e_G(X,X)$.
		The maximum and minimum degrees of a graph $G$ are denoted by $\Delta(G)$ and $\delta(G)$, respectively. For a vertex $v \in V(G)$, let $N_G(v)$ denote the set of neighbours of $v$ in $G$. For a set $S \subseteq V(G)$ and $v \in V(G)$, let $d_G(v,S)$ denote the number of edges in $G$ between $v$ and $S$, and write $d_G(v) \coloneqq d_G(v,V(G))$ for the degree of $v$ in $G$. If the host graph is clear from the context, sometimes we drop the subscript and write $e(X,Y), e(X), d(v, S), N(v)$ instead of $e_G(X,Y), e_G(X), d_G(v, S), N_G(v)$ respectively. For sets $A$ and $B$, we write $A-B$ for the set difference $A\setminus B$. For a set $U \subseteq V(G)$, let $\comp{U}$ denote the set $V(G) \setminus U$, and let $G \setminus U$ denote the subgraph of $G$ induced by $V(G) \setminus U$.

		For a graph $G$, a \textit{cut} of $G$ is a partition $\{X,Y\}$ of $V(G)$, where $X$ and $Y$ are both non-empty. We say that a cut $\{X, Y\}$ is \textit{$\alpha$-sparse} if $e_G(X, Y) \le \alpha|X||Y|$. If $G$ is an $n$-vertex graph, we say that $G$ is \textit{$\alpha$-almost-bipartite} if there exists a partition $\{X, Y\}$ of $V(G)$ such that $G$ has
		at most $\alpha n^2$ edges that are not in $G[X, Y]$. Otherwise we say that $G$ is \textit{$\alpha$-far-from-bipartite}.

		For a graph $G$ and vertices $x,y$ in $G$, an \emph{$(x,y)$-path} in $G$ is a path with ends $x$ and $y$. If $P$ is an $(x,y)$-path and $Q$ is a $(y,z)$-path, we denote by $xPyQz$ the concatenation of the paths $P$ and $Q$.

		Using a slight abuse of notation, for a collection $\cK$ of subgraphs in a graph $G$, we denote the set of vertices covered by these subgraphs as $V(\cK)$. Throughout this paper, we often omit floor and ceiling signs when dealing with large numbers whenever they are not crucial.

\section{Sketch of the proof of \Cref{mainthm:packingsubgraphs}}
	\label{sec:proosketchmain1}

	In this section, we present a sketch of the proof of our main result, \Cref{mainthm:packingsubgraphs}. At a high level, our approach begins by partitioning the vertex set of the regular graph $G$ into subsets with certain expansion properties, which we refer to as expanders. We then move some vertices between these expanders and remove a carefully constructed collection of $K_{t,t}$'s to `balance' these expanders, using our first key lemma (\Cref{lem:balancing}). More precisely, \Cref{lem:balancing} allows us to transform expanders that are bipartite (or close to bipartite) with imbalanced part sizes into balanced bipartite expanders. Next, with the help of our second key lemma (\Cref{lem:Ktt-packing-expander}), we find a $K_{t,t}$-packing in each resulting expander that covers all but a constant number of its vertices. Several significant challenges arise in implementing this approach, which we now address in detail.

	Let $F$ be a bipartite graph, and let $0<c<1$. Let $G$ be a $d$-regular graph of order $n$, where $d \ge cn$.
	Let $t = |V(F)|$.
	Recall that \Cref{mainthm:packingsubgraphs} aims to find an $F$-packing that covers all but at most a constant number of vertices of $G$.
		It suffices to find a $K_{t,t}$-packing in $G$ covering all but at most $C$ vertices for some constant $C=C(t,c)$. Indeed, if the two vertex classes of $F$ have sizes $a$ and $b$, then $t=a+b$, and two vertex-disjoint copies of $F$, with their two vertex classes placed in opposite orientations, partition the vertices of $K_{t,t}$.

	Our proof begins by partitioning the vertex set of the regular graph $G$ into a small number of subsets $Z_1, \ldots, Z_r$ such that, for each $i \in [r]$, the induced subgraph $G[Z_i]$ has strong expansion properties, is either almost bipartite or far from bipartite, and the number of edges of $G$ between different sets $Z_i$ is small. See Lemma~\ref{lem:expander-decomposition} for the precise statement.

	\subsection{Balancing the expanders}  \label{subsec:balancingexpanders}

		Given the partition $\{Z_1, \ldots, Z_r\}$ of the vertices of $G$, a natural strategy for obtaining a $K_{t,t}$-packing covering almost all vertices of $G$ is to find such a packing in each of the expanders $G[Z_i]$.
		However, it is possible that some of these expanders are rather imbalanced bipartite graphs, rendering this plan impossible.
		To overcome this obstacle, in our first key lemma (\Cref{lem:balancing}), we modify the sets $Z_i$ and remove a small collection of vertex-disjoint $K_{t,t}$-copies and a very small number of additional vertices so that the remainder of each modified set $Z_i$ spans an expander that is either bipartite and balanced or is far from bipartite. (In other words, this `balances' the expanders $G[Z_i]$ that are almost bipartite.)

		The first idea towards balancing the almost bipartite expanders is to use the regularity of $G$.
		For the rest of this section, let us assume for simplicity that $G$ itself is bipartite, and denote the bipartition of $G$ by $\{X,Y\}$ and the corresponding bipartition of $Z_i$ by $\{X_i,Y_i\}$. Relabelling the two parts if necessary, assume that $|X_i|\ge |Y_i|$ for every $i\in[r]$.
		Write $V := V(G)$, and let $H$ denote the subgraph of $G$ on vertex set $V$ consisting of all edges whose endpoints lie in distinct sets $Z_i$.

		A simple counting argument, using that $G$ is $d$-regular, shows that the following holds for every $i \in [r]$.
		\begin{equation} \label{eqn:invariant}
        e_H(X_i, V \setminus Z_i) - e_H(Y_i, V \setminus Z_i)
			 = e_G(X_i, V \setminus Z_i) - e_G(Y_i, V \setminus Z_i) = e_G(X_i,V)-e_G(Y_i,V) = (|X_i|-|Y_i|)\cdot d.
		\end{equation}
		Pretending for a moment that $t = 1$, so that a $K_{t,t}$-packing is just a matching, a natural aim is thus to find a matching consisting of $\frac{1}{d} e_H(X_i, V \setminus Z_i)$ edges leaving $Z_i$ and incident to $X_i$, and $\frac{1}{d} e_H(Y_i, V \setminus Z_i)$ edges leaving $Z_i$ and incident to $Y_i$. By~\eqref{eqn:invariant}, removing this matching leaves the same number of vertices uncovered in $X_i$ and $Y_i$.\footnote{We ignore rounding errors due to taking ceilings or floors; this is justified as we are allowed $O(1)$ uncovered vertices.} To achieve this simultaneously for all $i \in [r]$, it suffices to find a matching consisting of $\frac{1}{d} e_H(X_i, Y_j)$ edges between $X_i$ and $Y_j$, for every distinct $i,j \in [r]$.

		It turns out that it is quite easy to construct such a matching if all bipartite subgraphs $H[X_i, Y_j]$ for $i \neq j$ have maximum degree at most $\rho n$, for some small constant $0 < \rho \ll c$.
Indeed, we consider the pairs $(X_i, Y_j)$ with $i \neq j$ one by one, in increasing order of $e_H(X_i, Y_j)$, and at each step we construct a matching of size $\frac{1}{d} e_H(X_i, Y_j)$ in $H[X_i, Y_j]$ that is vertex-disjoint from all previously defined matchings. We claim that this is always possible. Indeed, consider a pair $(X_i, Y_j)$. Denote by $e$ the number of edges between $X_i$ and $Y_j$ and by $X_i' \subseteq X_i$ and $Y_j' \subseteq Y_j$ the sets of vertices that were used in previously defined matchings. Then $|X_i'|,|Y_j'|\le r\cdot e/d$, by the order in which we process the pairs and the fact that each fixed part occurs in at most $r-1$ previously processed pairs. By the maximum degree assumption and since $d\geq cn$, we have
		\begin{equation*}
			e_H(X_i \setminus X_i', Y_j \setminus Y_j')
			\ge e_H(X_i, Y_j) - \rho n \cdot (|X_i'| + |Y_j'|)
			\ge e-e/2=e/2.
		\end{equation*}
		Here we used $d\ge cn$, the bound $r\le R:=\lceil1/\min\{c,1/t\}\rceil$ supplied by the application of Lemma~\ref{lem:expander-decomposition}, and $\rho\ll c,1/t$. Since every graph $H'$ has a matching of size at least $e(H')/(2\Delta(H'))$,\footnote{Indeed, the endpoints of a maximal matching meet every edge, so $e(H')\le 2\Delta(H')|M|$.} it follows that there is a matching of size at least $e/(4\rho n)\ge e/d$ in $H[X_i \setminus X_i',Y_j \setminus Y_j']$, as required.

		For general $t \ge 1$, our aim is the same, except that we replace matching edges between $X_i$ and $Y_j$ with copies of $K_{t,t}$ in $G$, each consisting of one vertex in $X_i$, $t$ vertices in $Y_j$, and $t-1$ vertices in $X_j$, so as to achieve the same balancing effect as moving one vertex from $X_i$ to $X_j$. The above greedy algorithm still applies if we use an asymmetric version of the Kővári--Sós--Turán theorem (\Cref{lem:unbalanced-kst}) for all pairs $(X_i, Y_j)$ such that $H[X_i, Y_j]$ contains at least $Cn$ edges, where $C$ is a large constant. Since we allow $O(1)$ uncovered vertices, we may ignore pairs $(X_i, Y_j)$ with fewer edges. Indeed, this yields a `balancing' $K_{t,t}$-packing $\cK$, which satisfies $|X_i \setminus V(\cK)| = |Y_i \setminus V(\cK)| + O(1)$ for all $i$.

		The main challenge in proving \Cref{lem:balancing} is therefore to emulate the situation in which the maximum degree (between the sets $Z_i$) is bounded by $\rho n$, which we achieve as follows.
		We perform the procedure below, sequentially modifying the graph $H$ and the sets $Z_1, \ldots, Z_r$. If there is a pair $(X_i, Y_j)$ with $e_H(X_i, Y_j) \ge d$ and a vertex $v \in X_i$ that has at least $\rho n$ neighbours in $H$ lying in $Y_j$, then we perform the following three operations: we move $v$ from $X_i$ to $X_j$; we remove all edges between $v$ and $Y_j$ from $H$; and we remove $d - d_H(v)$ edges between $X_i \setminus \{v\}$ and $Y_j$ from $H$, chosen arbitrarily. (The fact that the third operation is possible follows from the assumption that $e_H(X_i, Y_j) \ge d$.) Note that this procedure only removes edges from $H$ in each step and never adds any. In particular, although $v$ is now moved to $X_j$ we do not add to $H$ the edges from $v$ to $Y_i$. This shows that the number of edges in $H$ decreases by at least $\rho n$ in each step of the procedure.
		The first two operations are quite natural: since $v$ has large degree into $Y_j$, adding it to $X_j$ preserves the expansion properties needed later for $G[X_j,Y_j]$. Since $H$ initially contains no edges between $X_\ell$ and $Y_\ell$ for any $\ell$, the second operation preserves this condition after each step of the procedure.
		The motivation for the third operation is the observation that if $H, X_1, \ldots, X_r, Y_1, \ldots, Y_r$ satisfy \eqref{eqn:invariant} and we modify them by performing the three operations above, then the resulting $H, X_1, \ldots, X_r, Y_1, \ldots, Y_r$ still satisfy \eqref{eqn:invariant}. Indeed, for example, to see that \eqref{eqn:invariant} continues to hold for $X_i, Y_i$, note that after performing the three operations above, both the left-hand and right-hand sides of \eqref{eqn:invariant} decrease by exactly $d$.

		Continuing this procedure for as long as possible (also with the roles of $X_i$ and $Y_i$ reversed), the resulting graph $H$ and sets $X_1, \ldots, X_r, Y_1, \ldots, Y_r$ still satisfy \eqref{eqn:invariant}, and for every $i \neq j$, the bipartite subgraph $H[X_i, Y_j]$ either has maximum degree at most $\rho n$ or contains fewer than $d$ edges. We may therefore apply the greedy $K_{t,t}$-packing construction described before the vertex-moving procedure to this final graph $H$. It produces a packing $\cK$ such that $|X_i \setminus V(\cK)| = |Y_i \setminus V(\cK)| + O(1)$ for every $i$. Moreover, Lemma~\ref{lem:expander-decomposition} gives $e(H)\le \eta n^2$ initially, since every edge of $H$ has ends in distinct sets $Z_i$, and each step removes at least $\rho n$ edges. Hence the procedure has at most $\eta n/\rho$ steps and moves at most that many vertices. The bound $\eta n/\rho$ on the number of moved vertices is used in Section~\ref{subsec:checking-parameters-after-moves} to show that the modified sets $Z_i$ are still expanders.

		Balancing expanders (in different contexts) was also a key step in the proofs of K\"uhn, Lo, Osthus and Staden in~\cite{kuhn2015robust, kuhn2016solution} and of Gruslys and Letzter in~\cite{gruslys2021cycle}. The method we develop here to construct a balanced $K_{t,t}$-packing differs significantly from these previous approaches.

	\subsection{Packing $K_{t,t}$'s in expanders} \label{subsec:proofsketchmatchingtemplate}

		Having balanced the expanders (using \Cref{lem:balancing}), the original problem is reduced to finding a $K_{t,t}$-packing covering all but a constant number of vertices in an expander $G$ that is dense, almost regular, and either far from bipartite or balanced and bipartite. We address this reduced problem in our second key lemma, \Cref{lem:Ktt-packing-expander}, where we construct a \emph{perfect} $K_{t,t}$-packing in such an expander under the assumption that the number of vertices is divisible by $2t$.

		The proof of \Cref{lem:Ktt-packing-expander} uses the regularity lemma and combines two main themes. First, we exploit the fact that dense, almost regular expanders that are either far from bipartite or bipartite and balanced contain Hamilton cycles, following the work of Kühn, Osthus, and Treglown~\cite{kuhn2010hamiltonian} on robust expanders. Second, as in the previous section, we again use $K_{t,t}$-packings for balancing purposes, but now in two distinct ways that differ from the approach used earlier.

		To begin, we aim to find disjoint sets $U_1, \ldots, U_{2m}$ of equal size that cover almost all vertices of $G$, such that almost every pair $(U_i, U_j)$ is $\eps$-regular in $G$, and each pair $G[U_{2i-1}, U_{2i}]$ is somewhat dense and super-regular (that is, almost every pair $(U_i, U_j)$ is $\eps$-regular in $G$, and moreover $G[U_{2i-1}, U_{2i}]$ has reasonably large minimum degree).
To achieve this, we apply the regularity lemma (\Cref{lem:regularity}) to $G$ to obtain vertex-disjoint clusters $V_1, \ldots, V_m$ of equal size that cover almost all vertices of $G$ and such that almost all pairs are $\eps$-regular. We then use the expansion of $G$, together with Lemma~\ref{lem:hamilton-cycle}, to deduce that the reduced graph $\Gamma$ (whose vertices are $V_1, \ldots, V_m$ and whose edges correspond to dense $\eps$-regular pairs) admits a perfect \emph{fractional matching}. This, in turn, implies the existence of a perfect $2$-matching in $\Gamma$, that is, a collection of vertex-disjoint edges and odd cycles covering all vertices.
Consequently, splitting each cluster $V_i$ into two equal parts yields a perfect matching in the corresponding reduced graph. By possibly removing a small proportion of vertices from each of these new clusters, we can ensure that each pair selected by this matching is super-regular, as required. The idea of seeking perfect $2$-matchings in reduced graphs has appeared previously, for example in~\cite{kelly2008dirac, keevash2009exact, korandi2021minimum}.

		Having found the sets $U_1, \ldots, U_{2m}$, we can apply the blow-up lemma (\Cref{lem:blowup}) to each pair $G[U_{2i-1}, U_{2i}]$ to cover all but fewer than $2t$ vertices of $U_{2i-1} \cup U_{2i}$ with a $K_{t,t}$-packing. However, this approach may only produce a $K_{t,t}$-packing in $G$ that leaves linearly many vertices uncovered, namely those outside the sets $U_1, \ldots, U_{2m}$. To cover all vertices of $G$ (rather than almost all), we instead find a $K_{t,t}$-packing $\cK$ that covers every vertex outside $U_1 \cup \cdots \cup U_{2m}$ and only very few vertices in each $U_i$, and, crucially, leaves the same number of uncovered vertices in each $U_i$, where this number is divisible by $t$. Once such a $K_{t,t}$-packing $\cK$ is found, we can apply the blow-up lemma to cover the remaining vertices in each pair $G[U_{2i-1}, U_{2i}]$ with a perfect $K_{t,t}$-packing, thereby completing $\cK$ to a perfect $K_{t,t}$-packing in $G$.

		To construct such a $K_{t,t}$-packing $\cK$, we proceed in three steps. First, we construct a packing $\cK_1$ that covers all vertices outside $U_1 \cup \cdots \cup U_{2m}$ while using very few vertices from each set $U_i$. This is straightforward using the Kővári--Sós--Turán theorem together with the properties of $G$. Next, we construct two further $K_{t,t}$-packings, $\cK_2$ and $\cK_3$, each serving a different balancing role. The packing $\cK_2$ is chosen so that, for every $i$, the number of vertices in $U_i$ uncovered by $\cK_1 \cup \cK_2$ is divisible by $t$. We then construct $\cK_3$ so that the number of vertices in each $U_i$ uncovered by $\cK_1 \cup \cK_2 \cup \cK_3$ is the same for all $i$, and still divisible by $t$. Taking $\cK := \cK_1 \cup \cK_2 \cup \cK_3$ allows us to complete the proof using the blow-up lemma, as explained in the previous paragraph.

		For the construction of $\cK_2$, the key idea is that a walk
$U_{i_1}\ldots U_{i_{2s-1}}$ of even length in the reduced graph allows us
to ``transfer'' vertices from $U_{i_1}$ to $U_{i_{2s-1}}$ (modulo $t$).
More precisely, we construct a collection of copies of $K_{t,t}$ between
the clusters corresponding to this walk so as to create a balancing effect
on the clusters at the two ends of the walk when these copies are removed.
By combining several such walks, we can construct a collection $\cK_2$ of
copies of $K_{t,t}$ such that the number of uncovered vertices in each
$U_i$ is divisible by $t$, as desired (see
Section~\ref{subsec:Kttdivisibility} for further details).

We now describe the construction of $\cK_3$. For each $i\in[2m]$, let
$U_i'$ denote the set of vertices in $U_i$ that remain uncovered. Our first
step is to find a \emph{template} matching $\cM$ in
$G[U_1'\cup\ldots\cup U_{2m}']$ such that
$|U_i'| - t\cdot |V(\cM)\cap U_i'|$ is the same for all $i\in[2m]$. Given
such a matching $\cM$, we apply the blow-up lemma to construct a
$K_{t,t}$-packing $\cK_3$ by replacing each edge of $\cM$ between $U_i'$
and $U_j'$ with a copy of $K_{t,t}$ in $G[U_i',U_j']$. Consequently, the
number of vertices in $U_i'$ that remain uncovered by $\cK_3$ is exactly
$|U_i'| - t\cdot |V(\cM)\cap U_i'|$. By the choice of $\cM$, this quantity
is the same for all $i\in[2m]$; moreover, it is divisible by $t$ by
choice of $\cK_2$, as required. Finally, to construct the template matching $\cM$, we select subsets
$U_i''\subseteq U_i$ of appropriate size and show that the induced graph
$G[U_1''\cup\ldots\cup U_{2m}'']$ inherits the expansion properties of $G$ needed to apply Lemma~\ref{lem:hamilton-cycle}. By that lemma, this graph contains a Hamilton cycle and hence a perfect matching, which serves as the desired template matching $\cM$
(see Section~\ref{subsec:templatematching} for further details).

\section{Useful lemmas and tools}
\label{sec:tools}

	\subsection{Concentration inequality}\label{sec:concdense}

		We need the following concentration inequality for the hypergeometric distribution. Let $m, n, N \in \mathbb{N}$ such that $\max\{m, n\} < N$. Recall that a random variable $X$ has \emph{hypergeometric distribution} with parameters $N,m,n$ if $X \coloneqq |S \cap [m]|$, where $S$ is a uniformly random subset of $[N]$ of size $n$.

		\begin{lemma}[Theorem 2.10 in \cite{janson2000wiley}]
			\label{hypergeometricconcentration}
			Let $X\sim\operatorname{Hypergeo}(N,m,n)$ be a hypergeometric random variable with mean $\mathbb{E}[X]=nm/N$. Then for all $0 \le \eps  \le 3/2$, we have
			\begin{equation*}
				\mathbb{P}\big(|X-\mathbb{E}[X]|\ge\varepsilon \mathbb{E}[X] \big) \le 2e^{-\eps^2\mathbb{E}[X]/3}.
			\end{equation*}
		\end{lemma}

	\subsection{Expanders in dense regular graphs}

		We will use the following structural result (\Cref{lem:expander-decomposition}) due to Gruslys and Letzter~\cite{gruslys2021cycle}, which goes back to the work of Kühn, Lo, Osthus, and Staden~\cite{kuhn2015robust, kuhn2016solution}. This result serves as the first step in the proofs of both \Cref{mainthm:packingsubgraphs} and \Cref{thm:Kt-subdivision-dense}. It shows that the vertices of a dense regular graph $G$ can be partitioned into a small number of sets such that the subgraphs induced by these sets exhibit strong expansion properties and have a large minimum degree. Throughout the remainder of the paper, we refer to these subgraphs as \emph{expanders}. The notion of expansion used here is characterized by the absence of sparse cuts. Specifically, a graph $G$ has no $\zeta$-sparse cuts if, for every partition $\{X, Y\}$ of $V(G)$, the number of edges between $X$ and $Y$ is at least $\zeta |X| |Y|$. This notion of expansion was first introduced (using different terminology and with $\zeta$ as a function of $n$) in a paper by Conlon, Fox, and Sudakov~\cite{conlon2014cycle}. We will distinguish between expanders that are close to bipartite and those that are far from bipartite. Recall that a graph $G$ on $n$ vertices is said to be \emph{$\gamma$-almost-bipartite} if one can make $G$ bipartite by removing up to $\gamma n^2$ edges, and otherwise, $G$ is called \emph{$\gamma$-far-from-bipartite}.

		\begin{lemma}[Lemma 2.1 from \cite{gruslys2021cycle}] \label{lem:expander-decomposition}
			Let $0<c<1$ and $n_0 \in \bN$ satisfy $1/n_0 \ll c$.
			Let $G$ be a $d$-regular graph on $n$ vertices, where $n \ge n_0$ and $d \ge cn$.
			Then there exist a positive integer $r\le\ceil{1/c}$, positive numbers $\eta, \beta, \gamma, \zeta, \delta$ with $1/n_0 \ll \eta \ll \beta \ll \gamma \ll \zeta \ll \delta \ll c$, and a partition $\{Z_1, \ldots, Z_r\}$ of $V(G)$ satisfying the following properties.
			\begin{enumerate}[label = \rm(E\arabic*)]
				\item \label{itm:expander-decomp-1}
					$G$ has at most $\eta n^2$ edges with ends in distinct $Z_i$'s.
				\item \label{itm:expander-decomp-2}
					$G[Z_i]$ has minimum degree at least $\delta n$, for $i \in [r]$.
				\item \label{itm:expander-decomp-3}
					$G[Z_i]$ has no $\zeta$-sparse cuts, for $i \in [r]$.
				\item \label{itm:expander-decomp-4}
					$G[Z_i]$ is either $\beta$-almost-bipartite or $\gamma$-far-from-bipartite, for $i \in [r]$.
			\end{enumerate}
		\end{lemma}

		The following lemma is used multiple times throughout the paper. Its proof is very similar to that of Lemma 2.2 in~\cite{gruslys2021cycle}, and relies on results concerning Hamiltonicity of robust expanders, due to Kühn, Osthus and Treglown~\cite{kuhn2010hamiltonian}. The lemma shows that expanders far from being bipartite are Hamiltonian and remain Hamiltonian even after the removal of any small set of vertices.
		However, expanders that are bipartite may not be Hamiltonian -- for example, an imbalanced bipartite graph cannot be Hamiltonian. In this case, the lemma shows that such expanders still become Hamiltonian after the removal of any small set of vertices that balances its two parts. We prove Lemma~\ref{lem:hamilton-cycle} in \Cref{sec:ham}.

		\begin{lemma} \label{lem:hamilton-cycle}
			Let $0<c<1$ and $1/n \ll \eta \ll \xi \ll \gamma, \zeta \ll c$, and let $d \ge cn$.
Suppose $G$ is an $n$-vertex graph with maximum degree at most $d$ and average degree at least $d - \eta n$, such that $G$ does not have $\zeta$-sparse cuts. Additionally, assume that $G$ is bipartite or $\gamma$-far-from-bipartite. Then the following holds.
			\begin{enumerate}[label = \rm(HC\arabic*)]
				\item \label{lem:hamilton:bipartite}
					If $G$ is bipartite with the bipartition $\{X, Y\}$, then, for every subset $W \subseteq V(G)$ of size at most $\xi n$
					satisfying $|X \setminus W| = |Y \setminus W|$ and any vertices $x \in X \setminus W$, $y \in Y \setminus W$, there is a Hamilton path in $G \setminus W$ with ends $x$ and $y$. In particular, $G \setminus W$ has a Hamilton cycle.

				\item
					If $G$ is $\gamma$-far-from-bipartite, then, for every subset $W \subseteq V(G)$ of size at most $\xi n$ and any two distinct vertices $x,y \in V(G) \setminus W$, there is a Hamilton path in $G \setminus W$ with ends $x$ and $y$. In particular, $G \setminus W$ has a Hamilton cycle.
			\end{enumerate}
		\end{lemma}

		We use the following lemma from \cite{gruslys2021cycle} in the proof of Theorem~\ref{thm:Kt-subdivision-dense}. When presented with a partition of a dense regular graph into expanders (by applying Lemma~\ref{lem:expander-decomposition}), \Cref{lem:path-forest} produces a small linear forest (i.e., a small collection of vertex-disjoint paths) whose removal balances the expanders that are close to being bipartite. A similar linear forest was also constructed in~\cite{kuhn2015robust, kuhn2016solution}.

		\begin{lemma}[Lemma 2.3 from \cite{gruslys2021cycle}] \label{lem:path-forest}
			Let $\eta,\beta,\xi,\gamma,\zeta,\delta\in(0,1)$, let $0<c<1$, and let $n\in\bN$ satisfy $1/n \ll \eta \ll \beta \ll \xi \ll \gamma \ll \zeta \ll \delta \ll c$.
			Let $G$ be a $d$-regular graph on $n$ vertices, where $d \ge cn$.
			Suppose that $\{Z_1, \ldots, Z_r\}$ is a partition of $V(G)$ satisfying properties \ref{itm:expander-decomp-1}-- \ref{itm:expander-decomp-4} in Lemma~\ref{lem:expander-decomposition}, where $r \le \ceil{1/c}$. For $i \in [r]$ such that $G[Z_i]$ is $\beta$-almost-bipartite, let $\{X_i, Y_i\}$ be a partition of $Z_i$ maximising $e_G(X_i, Y_i)$. Then there is a linear forest $H$ in $G$ with the following properties.
			\begin{enumerate}[label = \rm(P\arabic*)]
				\item \label{itm:Hforest-1}
					$|V(H)| \le \xi n$.
				\item \label{itm:Hforest-2}
					$H$ has no isolated vertices.
				\item \label{itm:Hforest-3}
					For each $i \in [r]$, $Z_i$ contains either zero or two leaves of $H$. Moreover, for each $i \in [r]$ such that $G[Z_i]$ is $\beta$-almost-bipartite and $Z_i$ contains two leaves of $H$, one of the leaves is in $X_i$ and the other leaf is in $Y_i$.
				\item \label{itm:Hforest-4}
					For each $i \in [r]$ such that $G[Z_i]$ is $\beta$-almost-bipartite, $|X_i \setminus V(H)| = |Y_i \setminus V(H)|$.
			\end{enumerate}
		\end{lemma}

	\subsection{The regularity lemma and the blow-up lemma}
		Given a graph $G$, and sets $X, Y \subseteq V(G)$, let $d(X,Y) \coloneqq \frac{e(X, Y)}{|X||Y|}$.
		The density of a bipartite graph $G$ with the bipartition $\{A, B\}$ is denoted by $d(A,B).$

		Given $\eps > 0$, we say that a bipartite graph $G$ with the bipartition $\{A, B\}$ is \emph{$\eps$-regular} if, for all sets $X \subseteq A$ and $Y \subseteq B$ with $|X| \ge \eps|A|$ and $|Y | \ge \eps |B|$, we have $|d(A, B) - d(X, Y )| < \eps.$ The following is the degree form of Szemer\'edi's regularity lemma which can be easily derived from the classical version (see, e.g.,~\cite{bollobas2013modern} and~\cite{diestel2005graph}). It is a slight variation of Lemma~2.4 in~\cite{kuhn2005packings}.

		\begin{lemma}[Regularity lemma] \label{lem:regularity}
			For every $\eps > 0$, there exists $M = M(\eps)$ such that the following holds. Let $0 \le \mu \le 1$, and let $G$ be an $n$-vertex graph.
			Then there is a partition $\{V_0, \ldots, V_m\}$ of $V(G)$ and a spanning subgraph $G'$ of $G \setminus V_0$, such that the following properties hold.
			\begin{enumerate}[label = \rm(R\arabic*)]
				\item \label{itm:reg-1}
					$m \le M$.
				\item \label{itm:reg-2}
					If $G$ is bipartite, then every $V_i$ with $i \in [m]$ is contained in one of the two parts of $G$.
				\item \label{itm:reg-3}
					$|V_0| \le \eps n$ and $|V_1| = \ldots = |V_m| \le \eps n$.
				\item \label{itm:reg-4}
					$d_{G'}(v) \ge d_G(v) - (\mu + \eps)n$ for every $v \in V(G')$.
				\item \label{itm:reg-5}
					$e_{G'}(V_i) = 0$ for every $i \in [m]$.
				\item \label{itm:reg-6}
					For all $1 \le i,j \le m$, the graph $G'[V_i, V_j]$ is $\eps$-regular with density either $0$ or more than $\mu$.
			\end{enumerate}
		\end{lemma}

		The sets $V_1, \ldots, V_m$ in \Cref{lem:regularity} are called \emph{clusters}, and the set $V_0$ is called the \emph{exceptional set}. Given clusters and $G'$ as in Lemma~\ref{lem:regularity}, the \emph{reduced graph} $\Gamma$ is the graph whose vertices are $V_1, \ldots, V_m$ (note that $V_0$ is omitted here) and in which $V_i$ is adjacent to $V_j$ whenever $G'[V_i, V_j]$ is $\eps$-regular and has density more than $\mu$. Thus, $V_iV_j$ is an edge of $\Gamma$ if and only if $G'$ has an edge between $V_i$ and $V_j$.

		Given $\eps > 0$ and $d \in [0, 1]$, a bipartite graph $G$ with the bipartition $\{A, B\}$ is called \emph{$(\eps, d)$-super-regular}, if all sets $X \subseteq A$ and $Y \subseteq B$ with $|X| \ge \eps |A|$ and $|Y| \ge \eps |B|$ satisfy $d(X,Y) > d$, and moreover, every vertex in $A$ has degree more than $d|B|$, and every vertex in $B$ has degree more than $d|A|$.

		\begin{obs}[Regular-pair restriction]\label[obs]{obs:regular-pair-restriction}
			Let $0<\eps_0\le\alpha\le1$. If $(A,B)$ is an $\eps_0$-regular pair and $A'\subseteq A$, $B'\subseteq B$ satisfy $|A'|\ge\alpha|A|$ and $|B'|\ge\alpha|B|$, then $(A',B')$ is $\eps'$-regular, where
			\begin{equation*}
				\eps'\coloneqq\max\{\eps_0/\alpha,2\eps_0\},
			\end{equation*}
			and $|d(A',B')-d(A,B)|<\eps_0$.
		\end{obs}
		\begin{proof}
			Since $\alpha\ge\eps_0$, the sets $A'$ and $B'$ satisfy $|A'|\ge\eps_0|A|$ and $|B'|\ge\eps_0|B|$. Hence, the $\eps_0$-regularity of $(A,B)$ gives $|d(A',B')-d(A,B)|<\eps_0$. Now let $X\subseteq A'$ and $Y\subseteq B'$ satisfy $|X|\ge\eps'|A'|$ and $|Y|\ge\eps'|B'|$. Since $\eps'\ge\eps_0/\alpha$, we have
			\begin{equation*}
				|X|\ge\frac{\eps_0}{\alpha}|A'|\ge\eps_0|A|\qquad\text{and}\qquad |Y|\ge\frac{\eps_0}{\alpha}|B'|\ge\eps_0|B|.
			\end{equation*}
			Applying the $\eps_0$-regularity of $(A,B)$ to $(X,Y)$ and to $(A',B')$, we obtain
			\begin{equation*}
				|d(X,Y)-d(A',B')|\le |d(X,Y)-d(A,B)|+|d(A,B)-d(A',B')|<2\eps_0\le\eps'.
			\end{equation*}
			Thus $(A',B')$ is $\eps'$-regular.
		\end{proof}

		We will need the following special case of the blow-up lemma of Koml\'os, S\'ark\"{o}zy and Szemer\'edi~\cite{komlos1997blow}.

		\begin{lemma}[Special case of the blow-up lemma \cite{komlos1997blow}] \label{lem:blowup}
			For every $q>0$ and every positive integer $\Delta$, there exists a positive number $\eps_0=\eps_0(q,\Delta)$ such that, for every $0<\eps\le\eps_0$, the following holds. If $H$ is a balanced bipartite graph with $n$ vertices in each part and $\Delta(H)\le\Delta$, then every $(\eps,q)$-super-regular bipartite graph with $n$ vertices in each part contains $H$ as a subgraph.
		\end{lemma}

		We also need the following standard fact concerning super-regular graphs.

		\begin{prop}[Proposition 2.3 in~\cite{kuhn2005packings}] \label[prop]{claim:super-regular}
			Every $\eps$-regular bipartite graph $G$ with bipartition $\{A,B\}$ and density $q>2\eps$ can be made into a $(\eps/(1-\eps),q-2\eps)$-super-regular graph by removing at most $\eps|A|$ vertices from $A$ and at most $\eps|B|$ vertices from $B$.
		\end{prop}

	\subsection{The K\H{o}v\'ari--S\'os--Tur\'an theorem}




		The following is an unbalanced variant of the well-known K\H{o}v\'ari--S\'os--Tur\'an theorem~\cite{kHovari1954problem}.

		\begin{lemma} \label{lem:unbalanced-kst}
			Let $t \ge 2$, let $\delta > 0$, let $G$ be a bipartite graph on at most $n$ vertices with the bipartition $\{X,Y\}$ such that $|X| \ge 2t \left(\frac{e}{\delta}\right)^t$ and every vertex in $X$ has at least $\delta n$ neighbours in $Y$. Then $G$ contains a $K_{t,t}$.
		\end{lemma}

		\begin{proof}
			Suppose for a contradiction that $G$ is $K_{t,t}$-free. A copy of $K_{1, t}$ in $G$ is called a \emph{$t$-star}, the vertex of degree $t$ in a $t$-star is called its \emph{center}, and the $t$ vertices of degree $1$ in a $t$-star are called its \emph{leaves}. Since $G$ is $K_{t,t}$-free, the number of $t$-stars in $G$ whose centre is in $X$ and whose $t$ leaves are in $Y$ is at most
			\begin{equation*}
				(t-1)\binom{|Y|}{t} \le t \left(\frac{e|Y|}{t}\right)^t \le \frac{(en)^t}{t^{t-1}}.
			\end{equation*}
			On the other hand, since every vertex in $X$ has at least $\delta n$ neighbours in $Y$, the number of such $t$-stars in $G$ is at least
			\begin{equation*}
				|X| \binom{\delta n}{t} \ge |X| \left(\frac{\delta n}{t}\right)^t.
			\end{equation*}
			Thus, we obtain $|X| \le t \left(\frac{e}{\delta}\right)^t.$
			This is a contradiction, proving the lemma.
		\end{proof}

\section{Balancing the expanders} \label{sec:balancing}

	As discussed in the proof sketch, our proof of \Cref{mainthm:packingsubgraphs} starts by
	partitioning the vertex set of our dense regular graph $G$ into a small number of expanders (using Lemma~\ref{lem:expander-decomposition}). A natural strategy for finding a $K_{t, t}$-packing in $G$ is to find a $K_{t, t}$-packing in each of these expanders.
	However, if any of these expanders happens to be a rather imbalanced bipartite graph, then it is impossible to find a $K_{t, t}$-packing in it covering all but a constant number of vertices. 
	To overcome this difficulty, we need the following key lemma which shows that we can partition the vertex set of any dense regular graph $G$ into a set $L$ that contains a perfect $K_{t,t}$-packing, a small set of constant size, and a small number of expanders, each of which is either far from bipartite or bipartite and \emph{balanced}.
	Using \Cref{cor:Ktt-packing-expander} we can then find a $K_{t,t}$-packing in each of these expanders covering all but a constant number of its vertices, giving us the desired $K_{t,t}$-packing in $G$.

	The proof has four stages. First, the $d$-regularity of $G$ gives, for every $i\in[r]$, the identity
\begin{equation*}
\left(e_G(X_i,V(G)\setminus Z_i)-e_G(Y_i,V(G)\setminus Z_i)\right)
+2\left(e_G(X_i)-e_G(Y_i)\right)
=d\left(|X_i|-|Y_i|\right).
\end{equation*}
Second, we move a small number of vertices between the sets $Z_i$ and delete carefully chosen edges from an auxiliary graph $H$ so that the displayed identity remains valid with $H$ in place of $G$ throughout the procedure. When the procedure stops, every graph $G[X_i,Y_i]$ still has minimum degree at least $\rho n/2$; moreover, for every choice of one side of $Z_i$ and one side of $Z_j$, the auxiliary graph between the chosen sides either has fewer than $d$ edges or has maximum degree at most $\rho n$. Third, we use these degree bounds to construct a small $K_{t,t}$-packing whose removal leaves the two sides of each initially almost-bipartite part with sizes differing by a constant depending only on $c$ and $t$; deleting this bounded discrepancy from the larger side makes the two sides equal. Fourth, we verify that only a small number of vertices and edges have been removed, so the resulting parts retain the required lower bound on average degree, have no sparse cuts, and remain either balanced bipartite or far from bipartite.

	\begin{lemma} \label{lem:balancing}
		Let $0<c<1$, let $t \ge 2$ be an integer, and let $n_0$ be sufficiently large. Suppose that $G$ is a $d$-regular graph on $n$ vertices, where $d \ge cn$ and $n \ge n_0$.
			Then there exist positive numbers $\eta, \gamma, \zeta$ and a positive integer $r$, with
			\[
			r \le R\coloneqq\ceil{\frac1{\min\{c,1/t\}}}
			\qquad\text{and}\qquad
			1/n_0 \ll \eta \ll \zeta \ll \gamma  \ll c, 1/t,
			\]
			such that the following holds.
		There exist pairwise disjoint sets $Z_1, \cdots, Z_{r}, L \subseteq V(G)$ and a subgraph $G' \subseteq G$ with $V(G') = V(G)$ satisfying the following properties.
		\begin{enumerate}[label = \rm(B\arabic*)]
			\item \label{lem:balance:leftover}
				$|V(G) \setminus (Z_1 \cup \cdots \cup Z_{r} \cup L)| \le 64r^2 \cdot \frac{t}{c} \cdot \left(\frac{e}{8\zeta}\right)^t$.

			\item \label{lem:balance:highaveragedegree}
				For every $i \in [r]$, $G'[Z_i]$ has average degree at least $d - \eta |Z_i|$.

			\item \label{lem:balance:nosparsecut}
				For every $i \in [r]$, $G'[Z_i]$ has no $\zeta$-sparse cuts.
			\item \label{lem:balance:dichotomy}
				For every $i \in [r]$, $G'[Z_i]$ is either bipartite and balanced, or one needs to remove at least $\gamma n^2$ edges from $G'[Z_i]$ to make it bipartite.
			\item \label{lem:balance:perfectpackL}
				$G[L]$ has a perfect $K_{t,t}$-packing.
		\end{enumerate}
	\end{lemma}

	\begin{proof}[ of Lemma~\ref{lem:balancing}]
			Let $c, t, n_0, n, d$ be as in the statement of the lemma, and let $G$ be a $d$-regular graph on $n$ vertices, where $d \ge cn$. Apply \Cref{lem:expander-decomposition} to $G$ with $\min\{c,1/t\}$ in place of its density parameter $c$. This gives a positive integer $r\le R$, positive numbers $\eta_0,\beta,\gamma_0,\zeta_0,\delta_0$, and a partition $\{Z_1, \ldots, Z_r\}$ of $V(G)$ satisfying \ref{itm:expander-decomp-1}--\ref{itm:expander-decomp-4} and
		\begin{equation*}
			1/n_0 \ll \eta_0 \ll \beta \ll \gamma_0 \ll \zeta_0 \ll \delta_0 \ll c, 1/t.
		\end{equation*}

		Let $\sigma,\rho$ be new auxiliary constants satisfying
		\begin{equation*}
			\beta \ll \sigma \ll \rho \ll \gamma_0\delta_0^2.
		\end{equation*}
		Such constants exist because $\beta\ll\gamma_0\ll\delta_0$.
		For every $i \in [r]$ such that $G[Z_i]$ is $\beta$-almost-bipartite, fix $\{X_i, Y_i\}$ to be a partition of $Z_i$ that maximises the number of edges between $X_i$ and $Y_i$.
		For every $i \in [r]$ such that $G[Z_i]$ is $\gamma_0$-far-from-bipartite, let $\{X_i,Y_i\}$ be a partition of $Z_i$ satisfying
		\begin{enumerate}[label = \rm(G\arabic*)]
			\item \label{item:xiyiequal}
				$|e_G(X_i) - e_G(Y_i)| \le \beta n^2$,
			\item
				\label{xiyimindegree}
				$\delta(G[X_i,Y_i]) \ge \delta_0 n / 3$.
		\end{enumerate}
		Notice that this can easily be achieved by taking $\{X_i, Y_i\}$ to be a random partition of $Z_i$. Moreover, these conditions are satisfied also when $G[Z_i]$ is $\beta$-almost-bipartite. Indeed, if $G[Z_i]$ is $\beta$-almost-bipartite, then $e_G(X_i) + e_G(Y_i) \le \beta n^2$, implying that \ref{item:xiyiequal} holds. Moreover, since $\{X_i, Y_i\}$ is a partition maximising the number of edges in $G[X_i, Y_i]$ and since the minimum degree in $G[Z_i]$ is at least $\delta_0 n$ (by \ref{itm:expander-decomp-2}), the graph $G[X_i,Y_i]$ has minimum degree at least $\delta_0 n / 2$, showing \ref{xiyimindegree}.

		As mentioned in the proof sketch, we use the regularity of $G$ to correct the imbalance between $X_i$ and $Y_i$. The required identity, which generalises \eqref{eqn:invariant} from bipartite graphs to general graphs, is recorded separately for later reference.

		\begin{claim}\label{claim:main-property-H}
			For every $i\in[r]$,
			\begin{equation}\label{eqn:main-property-H}
				\left(e_G(X_i,\comp{Z_i})-e_G(Y_i,\comp{Z_i})\right)+2\left(e_G(X_i)-e_G(Y_i)\right)-d\left(|X_i|-|Y_i|\right)=0.
			\end{equation}
		\end{claim}

		\begin{proof}
			Counting the edges incident to $X_i$ and $Y_i$, respectively, gives
			\begin{align*}
				d|X_i| &= e_G(X_i,Y_i)+e_G(X_i,\comp{Z_i})+2e_G(X_i), \\
				d|Y_i| &= e_G(X_i,Y_i)+e_G(Y_i,\comp{Z_i})+2e_G(Y_i).
			\end{align*}
			Subtracting the second identity from the first proves \eqref{eqn:main-property-H}.
		\end{proof}

		As explained in the proof sketch (see also \Cref{claim:find-Ktts} below), \eqref{eqn:main-property-H} shows that to find a balancing $K_{t,t}$-packing, it suffices to find disjoint $K_{t,t}$-packings for each pair $(W_i,W_j)$ with $W_i\in\{X_i,Y_i\}$ and $W_j\in\{X_j,Y_j\}$ of size roughly $e_G(W_i,W_j)/d$, whose copies effectively move one vertex from $W_i$ to $Z_j\setminus W_j$. This can be done greedily if the maximum degree of the edges between expanders is small (at most $\rho n$ for a small constant $\rho$). However, the maximum degree of the edges between expanders need not be small, even if we are allowed to modify the sets $Z_i$ slightly. Nevertheless, we emulate the small-maximum-degree situation by constructing a subgraph $H$ of $G$ and modifying the sets $Z_i$ so that \eqref{eqn:main-property-H} is maintained and the maximum degree of the final $H$ is small.

		Let $H$ be the graph obtained from $G$ by removing all edges between $X_i$ and $Y_i$, and also arbitrarily removing $\min\{e_G(X_i), e_G(Y_i)\}$ edges from each of $G[X_i]$ and $G[Y_i]$ for all $i \in [r]$. Then $H$ satisfies the following properties.
		\begin{enumerate}[label = \rm(H\arabic*)]
			\item \label{itm:H-1}
				$\left(e_H(X_i, \comp{Z_i}) - e_H(Y_i, \comp{Z_i})\right) + 2\left(e_H(X_i) - e_H(Y_i)\right) - d\left(|X_i| - |Y_i|\right) = 0$,
			\item \label{itm:H-1b}
				$e_H(X_i, Y_i) = 0$ for every $i \in [r]$,
			\item \label{itm:H-2}
				$H$ has maximum degree at most $d - \delta_0 n / 3$,
			\item \label{itm:H-3}
				$e(H) \le \eta_0 n^2 + r\beta n^2 \le 2r\beta n^2$.
		\end{enumerate}

		Note that for \ref{itm:H-2} we used \ref{xiyimindegree}, and for \ref{itm:H-3} we used \ref{itm:expander-decomp-1} and \ref{item:xiyiequal}.

		In the next subsection, we use a procedure for moving vertices between the expanders $G[Z_i]$, $i \in [r]$, in order to ensure that the degrees of vertices in $H$ between the resulting subgraphs of $G$ are appropriately bounded (whenever there are sufficiently many such edges), while also ensuring that the subgraphs have large minimum degree. More precisely, we will show that, when the procedure ends, \ref{itm:H-1}--\ref{itm:H-3} still hold, and \ref{itm:H-4} and \ref{itm:H-5} below also hold.

	\subsection{Procedure for moving vertices between the expanders for better degree control}
		\label{subsec:degreecontrol}

		Start with $H, (X_i)_{i \in [r]}, (Y_i)_{i \in [r]}, (Z_i)_{i \in [r]}$ satisfying \ref{itm:H-1}--\ref{itm:H-3}, and modify them using the following procedure, by running it as long as possible. For convenience, in the rest of the proof we denote the values of these quantities before the start of the procedure by $H^{\textrm{start}}$, $(X^{\textrm{start}}_i)_{i \in [r]}$, $(Y^{\textrm{start}}_i)_{i \in [r]}$, $(Z^{\textrm{start}}_i)_{i \in [r]}$, respectively.

		\begin{algorithm}[!ht]
			\begin{algorithmic}[1]
				\label{alg:balancing}
				\Procedure{Move-high-degree-vertices}{}
				\State Start with $H, (X_i)_{i \in [r]}, (Y_i)_{i \in [r]}, (Z_i)_{i \in [r]}$ satisfying \ref{itm:H-1}--\ref{itm:H-3}.
				\While{ there exist $i,j \in [r],  W_i \in \{X_i, Y_i\},  W_j \in \{X_j, Y_j\}$ \text{and} $v \in W_i$ such that
				\begin{equation} \label{eqn:WiWj}
					d_H(v, W_j) \ge \rho n \quad\text{ and }\quad e_H(W_i, W_j) \ge d.
				\end{equation}}
				\State $W_i \gets W_i \setminus \{ v\}$,
				\State $U_j \gets U_j \cup \{v\}$, where  $\{U_j, W_j\} = \{X_j, Y_j\}$,
				\State $Z_i \gets X_i \cup Y_i$ and $Z_j \gets X_j \cup Y_j$,
				\State {$H \gets H \setminus (E_{i,j} \cup E_H(v,W_j))$, where $E_{i,j} \subseteq H[W_i \setminus \{v\}, W_j \setminus \{v\}]$ is a set of exactly $d - d_H(v)$ edges, and $E_H(v,W_j)$ is the set of all edges in $H$ between $v$ and $W_j$.}
				\EndWhile
				\EndProcedure
			\end{algorithmic}
		\end{algorithm}

		In other words, whenever there are $i,j \in [r]$, $W_i \in \{X_i, Y_i\}$, $W_j \in \{X_j, Y_j\}$ and $v \in W_i$ such that $d_H(v, W_j) \ge \rho n$ and $e_H(W_i, W_j) \ge d$ (see \eqref{eqn:WiWj}), we move $v$ from $W_i$ to $U_j$ (where $\{U_j, W_j\} = \{X_j, Y_j\}$), we remove $d - d_H(v)$ edges between $W_i \setminus \{v\}$ and $W_j \setminus \{v\}$ from $H$, and we also remove all edges between $v$ and $W_j$ from $H$. Note that it is possible to remove $d - d_H(v)$ edges between $W_i \setminus \{v\}$ and $W_j \setminus \{v\}$ by the assumption that $e_H(W_i,W_j) \ge d \ge d - d_H(v)+d_H(v, W_j)$.

		\begin{claim}
			\label{claim:H123hold}
			Throughout the procedure \textsc{Move-high-degree-vertices}, \ref{itm:H-1} to \ref{itm:H-3} hold.
		\end{claim}

		\begin{proof}
				The procedure only deletes edges from $H$, so neither $\Delta(H)$ nor $e(H)$ can increase; hence \ref{itm:H-2} and \ref{itm:H-3} remain valid. Property~\ref{itm:H-1b} also remains valid: it holds initially, and whenever $v$ is moved into $U_j$, every edge of $H$ from $v$ to the opposite part $W_j$ is deleted. It remains to prove that \ref{itm:H-1} is preserved by any single step. Let $H'$, $(X'_i)_{i\in[r]}$, $(Y'_i)_{i\in[r]}$ denote the quantities after that step. We need to show that for all $i \in [r]$,
			\begin{equation} \label{eqn:claim11}
				(e_{H'}(X'_i, \comp{Z'_i}) - e_{H'}(Y'_i, \comp{Z'_i})) + 2\left(e_{H'}(X'_i) - e_{H'}(Y'_i)\right) - d\left(|X'_i| - |Y'_i|\right) = 0.
			\end{equation}
			It suffices to prove \eqref{eqn:claim11} for $(W_i,W_j)=(X_i,Y_j)$ when $i\ne j$, and for $(X_i,X_i)$. Indeed, if $i\ne j$, then for each $k\in\{i,j\}$, interchanging $X_k$ and $Y_k$ preserves all assumptions and changes the sign of the left-hand side of \eqref{eqn:claim11} for $k$. Thus every choice of $(W_i,W_j)$ reduces to $(X_i,Y_j)$. If $i=j$, the choices with $W_i\ne W_j$ are impossible because $e_H(X_i,Y_i)=0<d$, and interchanging $X_i$ and $Y_i$ reduces $(Y_i,Y_i)$ to $(X_i,X_i)$.

			First, consider the case when $(W_i,W_j) = (X_i,Y_j)$ with $i \ne j$. In this case, note that the following twelve equations hold.

			\begin{align*}
				e_{H'}(X_i',\comp{Z_i'}) - e_H(X_i,\comp{Z_i}) & = -(d - d_H(v)) - d_H(v, \comp{Z_i}) + d_H(v, X_i)\\
				e_{H'}(Y_i',\comp{Z_i'}) - e_H(Y_i,\comp{Z_i}) & = 0 \\
				e_{H'}(X_i') - e_H(X_i) & = -d_H(v, X_i) \\
				e_{H'}(Y_i') - e_H(Y_i) & = 0 \\
				|X_i'| - |X_i| & = -1\\
				|Y_i'| - |Y_i| & = 0.\\ \\
				e_{H'}(X_j',\comp{Z_j'}) - e_H(X_j,\comp{Z_j}) & = d_H(v, \comp{Z_j}) - d_H(v, X_j)\\
				e_{H'}(Y_j',\comp{Z_j'}) - e_H(Y_j,\comp{Z_j}) & = -(d - d_H(v)) - d_H(v, Y_j)\\
				e_{H'}(X_j') - e_H(X_j) & = d_H(v, X_j)\\
				e_{H'}(Y_j') - e_H(Y_j) & = 0\\
				|X_j'| - |X_j| & = 1\\
				|Y_j'| - |Y_j| & = 0.
			\end{align*}
			Multiplying the first six equations by $1, -1, 2, -2, -d, d$, respectively, and adding them all up, we find that
			\begin{align*}
				& (e_{H'}(X_i',\comp{Z_i'}) - e_{H'}(Y_i',\comp{Z_i'})) + 2\left(e_{H'}(X_i') - e_{H'}(Y_i')\right) - d\left(|X_i'| - |Y_i'|\right) \\
				= & \left(e_H(X_i,\comp{Z_i}) - e_H(Y_i,\comp{Z_i})\right) + 2\left(e_H(X_i) - e_{H}(Y_i)\right) - d\left(|X_i| - |Y_i|\right) = 0.
			\end{align*}
			Multiplying the last six equations, which describe the changes for $Z_j$, by $1, -1, 2, -2, -d, d$, respectively, and adding them gives
			\begin{align*}
				& (e_{H'}(X_j',\comp{Z_j'}) - e_{H'}(Y_j',\comp{Z_j'})) + 2\left(e_{H'}(X_j') - e_{H'}(Y_j')\right) - d\left(|X_j'| - |Y_j'|\right) \\
				= & \left(e_H(X_j,\comp{Z_j}) - e_H(Y_j,\comp{Z_j})\right) + 2\left(e_H(X_j) - e_{H}(Y_j)\right) - d\left(|X_j| - |Y_j|\right) = 0.
			\end{align*}
				For $\ell \in [r] \setminus \{i,j\}$, no vertex changes membership in $X_\ell$ or $Y_\ell$, and no deleted edge has both ends in $Z_\ell$ or exactly one end in $Z_\ell$. Thus all six quantities in the left-hand side of \eqref{eqn:claim11} for $\ell$ are unchanged, proving \eqref{eqn:claim11} in this case.

			Now consider the case when $(W_i,W_j) = (X_i,X_i)$. In this case, note that the following six equations hold.
			\begin{align*}
				e_{H'}(X_i',\comp{Z_i'}) - e_H(X_i,\comp{Z_i}) & = -d_H(v, \comp{Z_i}) \\
				e_{H'}(Y_i',\comp{Z_i'}) - e_H(Y_i,\comp{Z_i}) & = d_H(v, \comp{Z_i}) \\
				e_{H'}(X_i') - e_H(X_i) & = -d_H(v, X_i) -(d - d_H(v))\\
				e_{H'}(Y_i') - e_H(Y_i) & = 0 \\
				|X_i'| - |X_i| & = -1\\
				|Y_i'| - |Y_i| & = 1.
			\end{align*}
			Multiplying these six equations by $1, -1, 2, -2, -d, d$, respectively, and adding the resulting equations gives
			\begin{align*}
				& (e_{H'}(X_i',\comp{Z_i'}) - e_{H'}(Y_i',\comp{Z_i'})) + 2\left(e_{H'}(X_i') - e_{H'}(Y_i')\right) - d\left(|X_i'| - |Y_i'|\right) \\
				= & \left(e_H(X_i,\comp{Z_i}) - e_H(Y_i,\comp{Z_i})\right) + 2\left(e_H(X_i) - e_{H}(Y_i)\right) - d\left(|X_i| - |Y_i|\right) = 0.
			\end{align*}
				For $\ell \in [r] \setminus \{i\}$, neither its two parts nor any edge incident with $Z_\ell$ changes. Hence all six quantities in the left-hand side of \eqref{eqn:claim11} for $\ell$ are unchanged. This proves \eqref{eqn:claim11} in this case and completes the proof of the claim.
		\end{proof}

		Let us bound the number of steps taken by the procedure \textsc{Move-high-degree-vertices} by the time it ends, that is, by the time \eqref{eqn:WiWj} is no longer satisfied by any choice of parameters.
		Notice that after each step of the procedure, the number of edges in $H$ decreases by at least $\rho n$.
		Since $H$ satisfies \ref{itm:H-3} at the beginning of the procedure, the total number of steps is at most
		\[
		\frac{2r\beta n^2}{\rho n}=\frac{2r\beta}{\rho}n\le\frac{\rho n}{2},
		\]
		where we used $r\le R$ and $\beta\ll\rho,c,1/t$. Since at most one vertex is moved in each step, we obtain the following.

		\begin{enumerate}[label = \rm(M)]
			\item \label{movedverticesitem}
				At most $(2r\beta/\rho)n \le \rho n / 2$ vertices are moved during the procedure \textsc{Move-high-degree-vertices}.
		\end{enumerate}

		In the rest of the proof, let $(X_i)_{i \in [r]}$, $(Y_i)_{i \in [r]}$, $(Z_i)_{i \in [r]}$ and $H$ denote the parts and the graph obtained at the end of the procedure \textsc{Move-high-degree-vertices}. Then we claim that the following properties hold.

		\begin{enumerate}[label = \rm(H\arabic*)]
			\setcounter{enumi}{4}
			\item \label{itm:H-4}
				$\delta(G[X_i,Y_i]) \ge \rho n / 2$ for $i \in [r]$.
			\item \label{itm:H-5}
				For every $i,j \in [r]$, $W_i \in \{X_i,Y_i\}$, $W_j \in \{X_j, Y_j\}$ either $e_H(W_i, W_j) < d$ or $\Delta(H[W_i,W_j]) \le \rho n$.
		\end{enumerate}

		We first verify \ref{itm:H-4}, which concerns degrees in the fixed graph $G$, not in the changing subgraph $H$. Fix a vertex $v$ in one of the final sets $X_i\cup Y_i$. If $v$ was never moved, then at the start of the procedure it had at least $\delta_0n/3$ neighbours in the opposite part by \ref{xiyimindegree}; at most $\rho n/2$ vertices are moved in total by \ref{movedverticesitem}, so at least $\delta_0n/3-\rho n/2\ge\rho n/2$ of these neighbours remain in the opposite final part. If $v$ was moved, consider its last move. At that moment it was placed opposite a set containing at least $\rho n$ neighbours of $v$, by \eqref{eqn:WiWj}. No more than $\rho n/2$ vertices are moved afterwards, so again at least $\rho n/2$ of those neighbours remain in the opposite final part. This proves \ref{itm:H-4}. Property \ref{itm:H-5} follows directly from the stopping condition of the procedure.

		In the next subsection, we carefully construct a small collection of vertex-disjoint copies of $K_{t,t}$ whose removal makes the subgraphs $G[X_i, Y_i]$, for $i \in [r]$, more balanced.

	\subsection{Constructing a collection $\cK$ of $K_{t,t}$'s in $G$ for balancing the subgraphs $G[X_i, Y_i]$}

		In the rest of this proof, let $$T \coloneqq \frac{4t}{c} \left(\frac{4e}{\rho}\right)^t.$$ By repeatedly applying an unbalanced variant of the K\H{o}v\'ari--S\'os--Tur\'an theorem (Lemma~\ref{lem:unbalanced-kst}), we prove the following claim. Recall that the vertex set of a given copy $K$ of $K_{t,t}$ in $G$ is denoted by $V(K)$, and the set of vertices contained in a given collection $\cK$ of copies of $K_{t,t}$ in $G$ is denoted by $V(\cK)$.

		\begin{claim} \label{claim:find-Ktts}
			There exists a collection $\cK \coloneqq \cup_{1 \le i \le j \le r} \big(\cK(X_i, X_j) \cup \cK(X_i, Y_j) \cup \cK(Y_i, X_j) \cup \cK(Y_i, Y_j)\big)$ of pairwise vertex-disjoint copies of $K_{t,t}$ in $G$, such that for every $i,j \in [r]$ with $i \le j$ and every $W_i \in \{X_i, Y_i\}$, $W_j \in \{X_j, Y_j\}$, the following two properties hold.
			\begin{enumerate}[label=\textup{(\roman*)}]
				\item\label{itm:find-Ktts-count}
					If $e_H(W_i, W_j) \ge Td$ then $\big|\cK(W_i, W_j)\big| = \floor{\frac{e_H(W_i,W_j)}{d}}$ and, otherwise, $\cK(W_i, W_j) = \emptyset$.
					In particular,
					\[
					|\cK|
					\le \frac1d\sum_{1\le i\le j\le r}
					\sum_{(W_i,W_j)\in\{X_i,Y_i\}\times\{X_j,Y_j\}}e_H(W_i,W_j)
					=\frac{e(H)}d.
					\]
				\item\label{itm:find-Ktts-placement}
					Every copy $K \in \cK(W_i, W_j)$ satisfies $V(K) \subseteq Z_i \cup Z_j$, and for every $\ell \in \{i,j\}$, we have
					\begin{equation} \label{eqn:K}
						|W_{\ell} \cap V(K)| - |U_{\ell} \cap V(K)| = \left\{
							\begin{array}{ll}
								1 & \text{if $i \neq j$} \\
								2 & \text{if $i = j$,}
							\end{array}
							\right.
						\end{equation}
						where $U_{\ell} = Z_{\ell} \setminus W_{\ell}$.
			\end{enumerate}
			\end{claim}

			Note that the final equality in \ref{itm:find-Ktts-count} holds because, for every $i\in[r]$, \ref{itm:H-1b} gives $e_H(X_i,Y_i)=0$. Thus the terms with $(W_i,W_j)=(X_i,Y_i)$ or $(Y_i,X_i)$ are zero, while the remaining terms in the displayed double sum count every edge of $H$ exactly once.

			\begin{proof}
					Order the pairs $(W_i, W_j)$ with $1 \le i \le j \le r$ and $e_H(W_i,W_j) \ge Td$ in nondecreasing order of $e_H(W_i,W_j)$, breaking ties arbitrarily, and denote the resulting total order by $\preceq$. We will show by induction that for every such pair $(W_i,W_j)$ there is a collection $\cK$ of pairwise vertex-disjoint copies of $K_{t,t}$ in $G$ satisfying the following: for every $(W_{i'}, W_{j'})$ such that $(W_{i'}, W_{j'}) \preceq (W_i,W_j)$, $\cK$ consists of exactly $\floor{\frac{e_H(W_{i'},W_{j'})}{d}}$ copies $K$ of $K_{t,t}$ satisfying $V(K) \subseteq Z_{i'} \cup Z_{j'}$ and \eqref{eqn:K}.

					To this end, fix $(W_i,W_j)$ such that $1 \le i \le j \le r$ and $e_H(W_i,W_j) \ge Td$, and suppose that there is a collection $\cK$ of pairwise vertex-disjoint copies of $K_{t,t}$ in $G$ satisfying the following: for every $(W_{i'},W_{j'})$ such that $(W_{i'},W_{j'}) \prec (W_i,W_j)$, $\cK$
				consists of exactly $\floor{\frac{e_H(W_{i'},W_{j'})}{d}}$ copies $K$ of $K_{t,t}$ satisfying $V(K) \subseteq Z_{i'} \cup Z_{j'}$ and \eqref{eqn:K} (and $\cK$ contains no other copies of $K_{t,t}$). Let $f \coloneqq e_H(W_i,W_j)$, and let $\cK'$ be a maximal collection of at most $\floor{f/d}$ copies $K$ of $K_{t,t}$ that are pairwise vertex-disjoint and vertex-disjoint from the copies of $K_{t,t}$ in $\cK$ and satisfy $V(K) \subseteq Z_{i} \cup Z_{j}$ and \eqref{eqn:K}.

					If $|\cK'| = \floor{f/d}$, then we are done with the proof of the induction step, so suppose $|\cK'| < \floor{f/d}$. Let $V$ be the set of all vertices contained in a copy of $K_{t,t}$ in $\cK \cup \cK'$.  Then $|V| \le (2r)^2 \cdot (f / d) \cdot 2t = 8 r^2 f t/d$ because every pair $(W_{i'}, W_{j'})$ satisfying $(W_{i'}, W_{j'}) \preceq (W_i,W_j)$ also satisfies $e_H(W_{i'},W_{j'}) \le f$ by the definition of $\preceq$, and there are at most $(2r)^2$ such pairs $(W_{i'}, W_{j'})$. Thus, the number of edges in $H[W_i,W_j]$ incident to $V$ is at most
				\[
				\rho n|V|\le \rho n\cdot\frac{8r^2ft}{d}\le\frac{8\rho r^2t}{c}f\le\frac f2.
				\]
				Here we used \ref{itm:H-5}, $1/d\le1/(cn)$, $r\le R$, and $\rho\ll c,1/t$.
				This implies that there is a vertex $v$ in $W_i$ with at least $\frac{f}{2n} \ge \frac{Td}{2n} \ge 2t\left(\frac{4e}{\rho}\right)^t$ neighbours in $W_j \setminus V$; denote the set of these neighbours by $A$.
				Observe that $G[W_j\setminus V,U_j\setminus V]$ has minimum degree at least $\rho n/4$. Indeed, \ref{itm:H-4} gives minimum degree at least $\rho n/2$ before deleting $V$, while
				\[
				|V|\le 2t\,\frac{e(H)}d\le 2t\,\frac{2r\beta n^2}{cn}=\frac{4tr\beta}{c}n\le\frac{\rho n}{4},
				\]
				using \ref{itm:H-3}, $d\ge cn$, $r\le R$, and $\beta\ll\rho,c,1/t$.
				Thus, by the K\H{o}v\'ari--S\'os--Tur\'an theorem (\Cref{lem:unbalanced-kst}), there is a copy of $K_{t,t-1}$ in $G[A, U_j \setminus V]$, with $t$ vertices in $A$ and $t-1$ vertices in $U_j \setminus V$. Together with $v$, this copy of $K_{t,t-1}$ forms a copy $K$ of $K_{t,t}$ in $G$ satisfying $V(K) \subseteq Z_i \cup Z_j$ and \eqref{eqn:K}, contradicting the maximality of the collection $\cK'$, and proving the claim.
			\end{proof}

			Let $\cK$ be the collection of copies of $K_{t,t}$ guaranteed by \Cref{claim:find-Ktts}.

			\begin{claim}\label{claim:K-balancing}
				For every $i\in[r]$,
				\begin{equation}\label{eqn:K-balancing}
					|X_i\cap V(\cK)|-|Y_i\cap V(\cK)|=|X_i|-|Y_i|\pm8rT.
				\end{equation}
			\end{claim}

			\begin{proof}
			Fix $i\in[r]$. Whenever $j<i$, $W_i\in\{X_i,Y_i\}$, and $W_j\in\{X_j,Y_j\}$, use the convention $\cK(W_i,W_j)\coloneqq\cK(W_j,W_i)$. By \eqref{eqn:K}, a copy selected for $(X_i,W_j)$, with $j \neq i$, contributes $1$ to $|X_i\cap V(\cK)|-|Y_i\cap V(\cK)|$, while a copy selected for $(Y_i,W_j)$ contributes $-1$; when both vertex classes of the selected copy lie in $Z_i$, a copy selected for $(X_i,X_i)$ contributes $2$, while one selected for $(Y_i,Y_i)$ contributes $-2$. Part~\ref{itm:find-Ktts-count} of \Cref{claim:find-Ktts} ensures that for every $j \in [r]$, $W_i \in \{X_i, Y_i\}$, $W_j \in \{X_j, Y_j\}$ such that $e_H(W_i, W_j) \ge Td$ and $W_j \neq Z_i \setminus W_i$, there are exactly $\floor{\frac{e_H(W_i,W_j)}{d}}$ copies in $\mathcal K(W_i,W_j)$. By part~\ref{itm:find-Ktts-placement}, every such copy $K$ satisfies $V(K) \subseteq Z_i \cup Z_j$ and~\eqref{eqn:K}. Summing these contributions over all copies in $\cK$ gives the first equality below. Here $\one\{E\}$ is the indicator of $E$: it equals $1$ when $E$ holds and $0$ otherwise.

			\begin{align*}
				& |X_i \cap V(\cK)| - |Y_i \cap V(\cK)| \\[.2em]
				& \quad = \,\sum_{W_j \in \{X_j, Y_j\}, \,\, j \in [r] \setminus \{i\}} \left(\floor{\frac{e_H(X_i, W_j)}{d}} \one\{e_H(X_i, W_j) \ge Td\} - \floor{\frac{e_H(Y_i, W_j)}{d}} \one\{e_H(Y_i, W_j) \ge Td\} \right) \\[.2em]
				& \qquad \quad + \left(2\floor{\frac{e_H(X_i)}{d}} \one\{e_H(X_i) \ge Td\} - 2\floor{\frac{e_H(Y_i)}{d}} \one\{e_H(Y_i) \ge Td\}\right) \\[.2em]
				&\quad = \, \sum_{W_j \in \{X_j, Y_j\}, \,\, j \in [r] \setminus \{i\}} \left(\frac{e_H(X_i, W_j)}{d} - \frac{e_H(Y_i, W_j)}{d} \right) + \left(\frac{2e_H(X_i)}{d} - \frac{2e_H(Y_i)}{d} \right) \pm 8rT \\[.2em]
				& \quad = \, \frac{1}{d} \cdot \left(e_H(X_i, \comp{Z_i}) - e_H(Y_i, \comp{Z_i}) + 2e_H(X_i) - 2e_H(Y_i) \right) \pm 8rT
				= |X_i| - |Y_i| \pm 8rT,
			\end{align*}
			where for the last equality we used \ref{itm:H-1}, which holds by \Cref{claim:H123hold}. This proves the claim.
			\end{proof}

			For every $i\in[r]$, subtracting the numbers of vertices covered by $\cK$ from the original side sizes and using \eqref{eqn:K-balancing} gives $|X_i \setminus V(\cK)|-|Y_i \setminus V(\cK)|=(|X_i|-|Y_i|)-(|X_i\cap V(\cK)|-|Y_i\cap V(\cK)|)=\pm 8rT$.

			For every $i\in[r]$, choose equal-sized subsets $\widehat X_i\subseteq X_i\setminus V(\cK)$ and $\widehat Y_i\subseteq Y_i\setminus V(\cK)$ by deleting at most $8rT$ vertices in total, all from the larger of the two sets. Write $\widehat Z_i\coloneqq\widehat X_i\cup\widehat Y_i$. Thus $Z_i\setminus(\widehat Z_i\cup V(\cK))$ consists of the at most $8rT$ vertices just deleted. Since $Z_1,\cdots,Z_r$ partition $V(G)$, summing over $i\in[r]$ gives
			\begin{equation}\label{eqn:leftovercalculation}
				\big|V(G)\setminus\big(\widehat Z_1\cup\cdots\cup\widehat Z_r\cup V(\cK)\big)\big|\le 8r^2T\le64r^2\frac tc\left(\frac{4e}{\rho}\right)^t.
			\end{equation}
			We now define the spanning subgraph $G'$ promised in the lemma. 
			Denote by $X_i^{\mathrm{start}}$, $Y_i^{\mathrm{start}}$, and $Z_i^{\mathrm{start}}$ the sets $X_i$, $Y_i$, and $Z_i$, respectively, before the procedure \textsc{Move-high-degree-vertices}.
			For each $i\in[r]$ for which $G[Z_i^{\mathrm{start}}]$ is $\beta$-almost-bipartite, delete from $G$ all edges within $\widehat X_i$ and within $\widehat Y_i$; make no deletion inside $\widehat Z_i$ for the remaining indices. Thus $G'[\widehat Z_i]$ is bipartite and balanced when $G[Z_i^{\mathrm{start}}]$ is $\beta$-almost-bipartite, while $G'[\widehat Z_i]$ equals $G[\widehat Z_i]$ when $G[Z_i^{\mathrm{start}}]$ is $\gamma_0$-far-from-bipartite. 

			We next compare $\widehat Z_i$ with $Z_i^{\mathrm{start}}$. By \ref{movedverticesitem} and the definition of $\widehat Z_i$,
			\[
			|Z_i^{\mathrm{start}}-\widehat Z_i|\le \frac{2r\beta}{\rho}n+|V(\cK)|+8rT,
			\qquad
			|\widehat Z_i-Z_i^{\mathrm{start}}|\le \frac{2r\beta}{\rho}n.
			\]
			Using $r\le R$ and $\beta\ll\rho,\sigma,c,1/t$, we have $(2r\beta/\rho)n\le\sigma n/4$. Moreover, by \Cref{claim:find-Ktts}, \ref{itm:H-3}, $d\ge cn$, and $\beta\ll\sigma,c,1/t$,
			\[
			|V(\cK)|\le 2t\frac{e(H)}d\le 2t\frac{2r\beta n^2}{cn}\le\frac{\sigma n}{4}.
			\]
			Finally, $8rT\le\sigma n/2$ because $r,T$ are bounded in terms of $c,t,\rho$ and $n_0$ is sufficiently large. Therefore,
			\begin{subequations}\label{eqn:changeinvertexsets}
				\begin{align}
				|Z_i^{\rm{start}}\setminus \widehat Z_i|&\le \sigma n
				\quad\text{and}\quad
				|\widehat Z_i\setminus Z_i^{\rm{start}}|\le \sigma n,
				\label{eqn:changeinvertexsets-sets}\\
				\max\big\{|\widehat X_i\mathbin{\triangle}X_i^{\mathrm{start}}|,
				|\widehat Y_i\mathbin{\triangle}Y_i^{\mathrm{start}}|\big\}
				&\le \frac{2r\beta}{\rho}n+|V(\cK)|+8rT\le\sigma n.
				\label{eqn:changeinvertexsets-sides}
				\end{align}
			\end{subequations}
			Indeed, every vertex in either symmetric difference in the second line satisfies one of the following three properties: it was moved during \textsc{Move-high-degree-vertices}, it was covered by $\cK$, or it was among the at most $8rT$ vertices deleted when $\widehat X_i$ and $\widehat Y_i$ were formed.

	\subsection{Showing that expansion is preserved after balancing}\label{subsec:checking-parameters-after-moves}

		In this subsection, we will prove the following claim, which shows that the subgraphs $G'[\widehat Z_i]$ for $i \in [r]$ (which are obtained after balancing) still have good expansion properties.

		\begin{claim}\label{claim:Zi-expander}
			For every $i\in[r]$, the graph $G'[\widehat Z_i]$ has no $\rho/8$-sparse cuts.
		\end{claim}

		\begin{proof}
			Fix $i\in[r]$ and consider a partition $\{A,B\}$ of $\widehat Z_i$ with $A,B\ne\emptyset$ and $|A|\ge|B|$. We show that $e_{G'}(A,B)\ge(\rho/8)|A||B|$.

			Suppose first that $|B|\le\rho n/8$. By \ref{itm:H-4}, $G[X_i,Y_i]$ has minimum degree at least $\rho n/2$. Since $|X_i-\widehat X_i|,|Y_i-\widehat Y_i|\le |V(\cK)|+8rT\le\sigma n\le\rho n/4$ by the construction above, the spanning bipartite subgraph $G[\widehat X_i,\widehat Y_i]$ of $G'[\widehat Z_i]$ has minimum degree at least $\rho n/4$. Hence
			\[
			e_{G'}(A,B)\ge |B|(\rho n/4-|B|)\ge |B|\rho n/8\ge(\rho/8)|A||B|.
			\]

			Now suppose that $|B|>\rho n/8$. Put
			\[
			U\coloneqq\widehat Z_i-Z_i^{\mathrm{start}},\qquad
			W\coloneqq Z_i^{\mathrm{start}}-\widehat Z_i,
			\]
			and define $A^{\mathrm{start}}\coloneqq(A-U)\cup W$ and $B^{\mathrm{start}}\coloneqq B-U$. Then $\{A^{\mathrm{start}},B^{\mathrm{start}}\}$ is a partition of $Z_i^{\mathrm{start}}$. By \ref{itm:expander-decomp-3}, $G[Z_i^{\mathrm{start}}]$ has no $\zeta_0$-sparse cuts. Therefore
			\begin{align}\label{eqn:cut}
			e_G(A^{\mathrm{start}},B^{\mathrm{start}})
			&\ge \zeta_0|A^{\mathrm{start}}||B^{\mathrm{start}}| \\
			&\ge \zeta_0(|A|-\sigma n)(|B|-\sigma n)
			\ge (\zeta_0/4)|A||B|,\label{eqn:cut-final}
			\end{align}
			where for the last inequality we used $|A|\ge|B|>\rho n/8\ge2\sigma n$.

			If $G'[\widehat Z_i]\ne G[\widehat Z_i]$, then $G[Z_i^{\mathrm{start}}]$ is $\beta$-almost-bipartite by \ref{itm:expander-decomp-4}, and $G'$ deletes only edges internal to $\widehat X_i$ or to $\widehat Y_i$. Consequently,
			\begin{align}\label{eqn:edges}
			e_G(\widehat Z_i)-e_{G'}(\widehat Z_i)
			&\le e_G(\widehat X_i)+e_G(\widehat Y_i) \\
			&\le e_G(X_i^{\mathrm{start}})+e_G(Y_i^{\mathrm{start}})+2\sigma n^2 \\
			&\le \beta n^2+2\sigma n^2\le3\sigma n^2,\label{eqn:edges-final}
			\end{align}
			Indeed, by \eqref{eqn:changeinvertexsets-sides}, each of $\widehat X_i$ and $\widehat Y_i$ differs from its corresponding starting set in at most $\sigma n$ vertices. Changing a vertex set by at most $\sigma n$ vertices changes the number of edges it induces by at most $\sigma n^2$, which proves the second inequality. For the first inequality in \eqref{eqn:edges-final}, since $G[Z_i^{\mathrm{start}}]$ is $\beta$-almost-bipartite, there is a partition of $Z_i^{\mathrm{start}}$ with at most $\beta|Z_i^{\mathrm{start}}|^2$ edges within its two parts. The partition $\{X_i^{\mathrm{start}},Y_i^{\mathrm{start}}\}$ was chosen to maximise the number of crossing edges, and hence to minimise the number of edges within the two parts. Therefore
			\[
			e_G(X_i^{\mathrm{start}})+e_G(Y_i^{\mathrm{start}})
			\le \beta|Z_i^{\mathrm{start}}|^2\le\beta n^2.
			\]
			The second inequality in \eqref{eqn:edges-final} follows from $\beta\ll\sigma$. If $G'[\widehat Z_i]=G[\widehat Z_i]$, then $e_G(\widehat Z_i)-e_{G'}(\widehat Z_i)=0\le3\sigma n^2$, so \eqref{eqn:edges-final} also holds in this case.

			Using $|U|,|W|\le\sigma n$, \eqref{eqn:edges-final}, and \eqref{eqn:cut-final}, we obtain
			\begin{align*}
			e_{G'}(A,B)
			&\ge e_{G'}(A^{\mathrm{start}},B^{\mathrm{start}})-\sigma n^2 \\
			&\ge e_G(A^{\mathrm{start}},B^{\mathrm{start}})-4\sigma n^2 \\
			&\ge (\zeta_0/4)|A||B|-4\sigma n^2 \\
			&\ge (\zeta_0/8)|A||B|\ge(\rho/8)|A||B|.
			\end{align*}
			For the penultimate inequality, we used $|A||B|\ge \rho^2n^2/64$ and $\sigma\ll\rho^2\zeta_0$; the last follows from $\rho\ll\zeta_0$. This proves the claim.
		\end{proof}

		\subsection{Putting everything together and completing the proof of Lemma~\ref{lem:balancing}}

			To complete the proof of Lemma~\ref{lem:balancing}, we will show that the following properties hold for all $i \in [r]$.
			\begin{enumerate}[label = \rm(\alph*)]
				\item \label{itm:Zprime1}
					$G'[\widehat Z_i]$ has no $\rho/8$-sparse cuts.
				\item \label{itm:Zprime2}
					Either $G'[\widehat Z_i]$ is bipartite and balanced, or one needs to remove at least $\gamma_0\delta_0^2 n^2/2$ edges from $G'[\widehat Z_i]$ to make it bipartite.
				\item \label{itm:Zprime3}
					The average degree of $G'[\widehat Z_i]$ is at least $d -  \frac{144 \sigma}{\rho^2} |\widehat Z_i|$.
			\end{enumerate}
			Indeed, \Cref{claim:Zi-expander} shows \ref{itm:Zprime1}.

			For proving \ref{itm:Zprime2}, notice that if $G[Z_i^{\textrm{start}}]$ is $\beta$-almost-bipartite, then $G'[\widehat Z_i]$ is indeed bipartite and balanced, by the choice of $\widehat X_i$ and $\widehat Y_i$.
			So consider $i \in [r]$ such that $G[Z_i^{\textrm{start}}]$ is $\gamma_0$-far-from-bipartite. Property~\ref{itm:expander-decomp-2} gives $\delta(G[Z_i^{\textrm{start}}])\ge\delta_0n$, and hence $|Z_i^{\textrm{start}}|\ge\delta_0n$. Recall also that $|Z_i^{\textrm{start}} \setminus \widehat Z_i| \le \sigma n$ by \eqref{eqn:changeinvertexsets-sets}. Suppose for a contradiction that one can remove fewer than $\gamma_0\delta_0^2n^2/2$ edges from $G'[\widehat Z_i] = G[\widehat Z_i]$ to make it bipartite. Removing in addition all edges incident to $Z_i^{\textrm{start}}\setminus\widehat Z_i$ would make $G[Z_i^{\textrm{start}}]$ bipartite after deleting fewer than
			\[
			\frac{\gamma_0\delta_0^2n^2}{2}+|Z_i^{\textrm{start}}\setminus\widehat Z_i|n
			\le \frac{\gamma_0\delta_0^2n^2}{2}+\sigma n^2
			<\gamma_0\delta_0^2n^2
			\le\gamma_0|Z_i^{\textrm{start}}|^2
			\]
			edges. This contradicts the fact that $G[Z_i^{\textrm{start}}]$ is $\gamma_0$-far-from-bipartite and proves \ref{itm:Zprime2}.

			Finally, for proving \ref{itm:Zprime3}, first recall that $e_G(\widehat Z_i) - e_{G'}(\widehat Z_i) \le 3\sigma n^2$, by \eqref{eqn:edges-final}.
			Second, by \eqref{eqn:changeinvertexsets-sets}, and since the sets $Z_i^{\textrm{start}}$ satisfy \ref{itm:expander-decomp-1}, we have
			\begin{align*}
				e_G(\widehat Z_i, \comp{\widehat Z_i})
				\le e_G(Z_i^{\textrm{start}}, \comp{Z_i^{\textrm{start}}}) + |Z_i^{\textrm{start}} \setminus \widehat Z_i |n + |\widehat Z_i \setminus Z_i^{\textrm{start}}|n
				\le
				\eta_0 n^2+2\sigma n^2
				\le3\sigma n^2.
			\end{align*}
			Here the last inequality uses $\eta_0\ll\sigma$.
			Combining the bounds on $e_G(\widehat Z_i)-e_{G'}(\widehat Z_i)$ and $e_G(\widehat Z_i,\comp{\widehat Z_i})$, we have
			\begin{align}
				\label{avdegG'Zi'}
				d|\widehat Z_i| - 2e_{G'}(\widehat Z_i)
				\le e_{G}(\widehat Z_i, \comp{\widehat Z_i}) + 2(e_G(\widehat Z_i) - e_{G'}(\widehat Z_i))
				\le 9\sigma n^2
				\le \frac{144 \sigma}{\rho^2} |\widehat Z_i|^2,
			\end{align}
			where in the last inequality we used $|\widehat Z_i| \ge \rho n/4$, which follows from $G[\widehat Z_i]$ having minimum degree at least $\rho n / 4$, as pointed out at the beginning of the proof of \Cref{claim:Zi-expander}.
			Hence, by \eqref{avdegG'Zi'}, the average degree of $G'[\widehat Z_i]$ is at least $d - \frac{144 \sigma}{\rho^2} |\widehat Z_i|$, proving \ref{itm:Zprime3}.

			Set
			\[
			\eta\coloneqq\frac{144\sigma}{\rho^2},\qquad
			\zeta\coloneqq\frac{\rho}{32},\qquad
			\gamma\coloneqq\frac{\gamma_0\delta_0^2}{2}.
			\]
			Then $1/n_0\ll\eta\ll\zeta\ll\gamma\ll c,1/t$. Properties \ref{itm:Zprime1}--\ref{itm:Zprime3} imply \ref{lem:balance:highaveragedegree}--\ref{lem:balance:dichotomy} with $\widehat Z_i$ playing the role of $Z_i$. In particular, \ref{itm:Zprime1} says that $G'[\widehat Z_i]$ has no $\rho/8$-sparse cuts, and hence it has no $\zeta$-sparse cuts because $\rho/8>\zeta$, as required in \ref{lem:balance:nosparsecut}. Since $\cK$ is a collection of pairwise vertex-disjoint copies of $K_{t,t}$, property \ref{lem:balance:perfectpackL} follows with $L=V(\cK)$. Finally, \eqref{eqn:leftovercalculation} and the identity $4e/\rho=e/(8\zeta)$ give
			\[
			\big|V(G)\setminus\big(\widehat Z_1\cup\cdots\cup\widehat Z_r\cup V(\cK)\big)\big|
			\le64r^2\frac tc\left(\frac{e}{8\zeta}\right)^t,
			\]
			which proves \ref{lem:balance:leftover} and completes the proof of \Cref{lem:balancing}.
\end{proof}

\section{Packing $K_{t,t}$'s in expanders} \label{sec:packings}

	Our second key lemma for proving Theorem~\ref{mainthm:packingsubgraphs} is as follows. This lemma shows that expanders admit a perfect $K_{t,t}$-packing if their number of vertices is divisible by $2t$. It will be used to find a $K_{t,t}$-packing that covers all but at most $2t-1$ vertices in each of the expanders provided by Lemma~\ref{lem:balancing}.

	The proof has four stages. First, we apply the regularity lemma and pair the clusters into super-regular pairs. Second, we partition each cluster into five specified subsets and show that the subgraph induced by the third subset from every cluster retains the required expansion and is either balanced bipartite or far from bipartite. Third, in the next two subsections, we construct collections $\cK_1$ and $\cK_2$ of copies of $K_{t,t}$ such that all copies in $\cK_1\cup\cK_2$ are pairwise vertex-disjoint, the `exceptional' vertices given by the regularity lemma (namely the vertices in $V_0$) are all in $V(\cK_1)$, and $|U_i\setminus V(\cK_1\cup\cK_2)|$ is divisible by $t$ for every $i\in[2m]$. Fourth, we find a perfect matching in a suitable auxiliary graph that serves as a template for constructing a balancing $K_{t,t}$-packing, after which the blow-up lemma covers all remaining vertices in the matched super-regular pairs.

	\begin{lemma} \label{lem:Ktt-packing-expander}
		Let $t\ge2$ and $n_0$ be integers, let $0<c<1$ and $\gamma,\zeta,\eta\in(0,1)$, and suppose that $1/n_0 \ll \eta \ll \zeta \ll \gamma \ll c, 1/t$. Let $n \ge n_0$ be an integer divisible by $2t$, let $d \ge cn$ and let $G$ be an $n$-vertex graph with the following properties.
		\begin{enumerate}[label = \rm(Q\arabic*)]
			\item \label{itm:pack-expander-degree}
				$G$ has average degree at least $d - \eta n$ and maximum degree at most $d$.
			\item \label{itm:pack-expander-cuts}
				$G$ has no $\zeta$-sparse cuts.
			\item \label{itm:pack-expander-dichotomy}
				$G$ is either bipartite and balanced, or $\gamma$-far-from-bipartite.
		\end{enumerate}
		Then $G$ has a perfect $K_{t,t}$-packing.
	\end{lemma}

	Note that the following corollary, where we drop the divisibility assumption from \Cref{lem:Ktt-packing-expander} and allow up to $2t-1$ uncovered vertices, follows directly from that lemma. Indeed, given $G$ as in the next corollary, remove up to $2t-1$ vertices to obtain a graph $G'$ whose order $n'$ is divisible by $2t$. If $G$ is balanced bipartite with bipartition $\{A,B\}$, then $n-n'$ is even. In this case, delete $(n-n')/2$ vertices from $A$ and $(n-n')/2$ vertices from $B$, so $G'$ remains balanced. It is easy to see that properties \ref{itm:pack-expander-degree}--\ref{itm:pack-expander-dichotomy} hold for $G'$ with slightly weaker parameters, say with $n_0-2t,c/2,\gamma/2,2\eta,\zeta/2,n',d$ in place of $n_0,c,\gamma,\eta,\zeta,n,d$, respectively.

	\begin{corollary} \label{cor:Ktt-packing-expander}
		Let $t\ge2$ and $n_0$ be integers, let $0<c<1$ and $\gamma,\zeta,\eta\in(0,1)$, and suppose that $1/n_0 \ll \eta \ll \zeta \ll \gamma \ll c, 1/t$. Let $n \ge n_0$, let $d \ge cn$ and let $G$ be an $n$-vertex graph with the following properties.
		\begin{enumerate}[label = \rm(\arabic*)]
			\item
				$G$ has average degree at least $d - \eta n$ and maximum degree at most $d$.
			\item
				$G$ has no $\zeta$-sparse cuts.
			\item
				$G$ is either bipartite and balanced, or $\gamma$-far-from-bipartite.
		\end{enumerate}
		Then, $G$ has a $K_{t,t}$-packing covering all but at most $2t-1$ vertices.
	\end{corollary}

	Let us now proceed with the proof of \Cref{lem:Ktt-packing-expander}.

	\begin{proof}[ of Lemma~\ref{lem:Ktt-packing-expander}]
		Let $\eps, \mu, \xi$ be positive numbers satisfying
		\begin{equation*}
			1/n_0 \ll \eps \ll \mu \ll \eta \ll \xi \ll \zeta \ll \gamma \ll c,1/t.
		\end{equation*}
		Apply the regularity lemma (\Cref{lem:regularity}) to $G$, with its $\eps$ replaced by $\eps/2$ and $\mu$ unchanged, to obtain positive numbers $M=M(\eps/2)$ and $m\le M$, a subgraph $G'$, and a partition $\{V_0,\ldots,V_m\}$ of $V(G)$ into clusters $V_1,\ldots,V_m$ and an exceptional set $V_0$ satisfying properties \ref{itm:reg-1}--\ref{itm:reg-6}.
		Let $\Gamma$ be the corresponding reduced graph: it has vertex set $\{V_1,\ldots,V_m\}$, and $V_iV_j$ is an edge in $\Gamma$ if and only if $G'[V_i,V_j]$ is $(\eps/2)$-regular with density more than $\mu$. Equivalently, $V_iV_j$ is an edge in $\Gamma$ if and only if there is an edge of $G'$ between $V_i$ and $V_j$ (by \ref{itm:reg-6}).

		Notice that if $G$ is bipartite, then both $G'$ and $\Gamma$ are bipartite. Indeed, if $G$ is bipartite then $G'$ is bipartite since it is a subgraph of $G$, and $\Gamma$ is bipartite because every $V_i$ with $i \in [m]$ is contained in one of the parts of $G$ by \ref{itm:reg-2}.

		In the rest of the proof, we will assume that if $G$ is bipartite and balanced, then both $G'$ and $\Gamma$ are bipartite and balanced. Indeed, since all clusters $V_i$ with $i \in [m]$ are equal in size — each of size at most $\varepsilon n/2$ by \ref{itm:reg-3} — and each $V_i$ lies entirely within one of the two parts of the bipartition of $G$ by \ref{itm:reg-2}, this adjustment can be made, if necessary, by moving from $G'$ to $V_0$ all vertices of some clusters whose total size is at most $\eps n$. Moreover, after (potentially) moving these vertices, it still follows from \ref{itm:reg-3} that
		\begin{equation}
			\label{eqn:boundV0}
			|V_0| \le 2\eps n,
		\end{equation}
		and by \ref{itm:reg-4}, for every $v \in V(G')$ we have
		\begin{equation}
			\label{eqn:lossofedgesG'toG}
			d_{G'}(v) \ge d_G(v) - (\mu + 2 \eps)n.
		\end{equation}

		We claim that, after this potential adjustment, $G'$ satisfies the following properties.

		\begin{enumerate}[label = \rm(G'\arabic*)]
			\item \label{itm:Gprime-1}
				$G'$ has average degree at least $d - 3\eta n$.
			\item \label{itm:Gprime-2}
				$G'$ has no $\zeta/4$-sparse cuts. (In particular, $G'$ has minimum degree at least $\zeta/4 \cdot (|V(G')|-1) \ge \zeta n/5.$)
			\item \label{itm:Gprime-3}
				If $G$ is bipartite and balanced, then $G'$ is  bipartite and balanced. Otherwise, one needs to remove at least $\gamma n^2/2$ edges from $G'$ to make it bipartite. (In particular, $G'$ is $\gamma/2$-far-from-bipartite.)
		\end{enumerate}
			To prove \ref{itm:Gprime-1}, let $n'=|V(G')|$. Since $|V_0|\le2\eps n$ by \eqref{eqn:boundV0}, $\Delta(G)\le d$, and $G$ has average degree at least $d-\eta n$, we have
			\[
			\sum_{v\in V(G')}d_G(v)\ge n(d-\eta n)-|V_0|d.
			\]
			By \eqref{eqn:lossofedgesG'toG}, the degree sum in $G'$ is therefore at least $n(d-\eta n)-|V_0|d-n'(\mu+2\eps)n$. Dividing by $n'\ge(1-2\eps)n$ and using $\eps,\mu\ll\eta$ gives average degree at least $d-3\eta n$, proving \ref{itm:Gprime-1}. To prove \ref{itm:Gprime-2}, we use \eqref{eqn:lossofedgesG'toG} together with property~\ref{itm:pack-expander-cuts}, namely the fact that $G$ has no $\zeta$-sparse cuts, as follows. For every partition $\{X, Y\}$ of $V(G')$ with $|X| \le |Y|$, we have
		\begin{align*}
			e_{G'}(X,Y)
			& \ge e_G(X, V(G) \setminus X) - |X|(\mu + 2\eps )n \\
			& \ge \zeta |X| \cdot (n - |X|) - |X|(\mu + 2\eps)n \\
			& \ge \left(\frac{\zeta}{2} - (\mu + 2\eps)\right) \cdot |X| \cdot n
			\ge \frac{\zeta}{4} \cdot |X| \cdot |Y|,
		\end{align*}
		using $|X| \le n/2$ for the third inequality, and $\eps, \mu \ll \zeta$ for the last inequality. This proves \ref{itm:Gprime-2}.
		Finally, to prove \ref{itm:Gprime-3}, suppose first that $G$ is bipartite and balanced. In the adjustment immediately preceding \eqref{eqn:boundV0}, we chose the surviving clusters so that their union $V(G')$ meets the two parts of $G$ equally; since $G'\subseteq G$, this makes $G'$ bipartite and balanced. Otherwise, by our assumption in \Cref{lem:Ktt-packing-expander}, $G$ is $\gamma$-far-from-bipartite, in which case we claim that at least $\gamma n^2/2$ edges need to be removed from $G'$ to make it bipartite. Indeed, by \eqref{eqn:boundV0} and \eqref{eqn:lossofedgesG'toG}, the number of edges in $G$ that are not present in $G'$ is at most $|V_0| \cdot n + n \cdot (\mu + 2\varepsilon) n \le (\mu + 4\varepsilon) n^2 \le \gamma n^2/2$. Since $G$ is $\gamma$-far-from-bipartite, at least $\gamma n^2$ edges must be removed to make it bipartite. Therefore, even after accounting for the fact that at most $\gamma n^2/2$ of these edges are missing in $G'$, one still needs to remove at least $\gamma n^2/2$ edges from $G'$, proving \ref{itm:Gprime-3}.

	\subsection{Finding vertex-disjoint super-regular subgraphs that cover all vertices outside $V_0$}

			Let $G_0'$ denote the graph $G'$ at the beginning of this subsection, and let $n_0'\coloneqq|V(G_0')|$. We keep $G_0'$ fixed throughout the subsection, whereas $G'$ is updated as vertices are moved to $V_0$ later in the proof.

			Recall that the clusters $V_i$, $i \in [m]$, have the same size by \ref{itm:reg-3}, and let $s\coloneqq|V_1|$ denote this common size before refinement. If $s$ is odd, move one vertex from every $V_i$ to $V_0$. Then, for each $i\in[m]$, partition the remaining vertices of $V_i$ into two equal sets $V_i'$ and $V_i''$. Thus $|V_i'|=|V_i''|=\lfloor s/2\rfloor$ for every $i\in[m]$. Let $\Gamma_{\rm{ref}}$ be the graph with vertex set $\{V_1',V_1'',\ldots,V_m',V_m''\}$ whose edges are all the pairs $AB$ with $A \in \{V_i', V_i''\}$, $B \in \{V_j', V_j''\}$ where $V_i V_j$ is an edge in $\Gamma$.
		Since $n_0$ is sufficiently large, $\lfloor s/2\rfloor\ge s/3$. By \Cref{obs:regular-pair-restriction} with $(\eps_0,\alpha)=(\eps/2,1/3)$, every edge in $\Gamma_{\rm{ref}}$ corresponds to a $(3\eps/2)$-regular, and hence $2\eps$-regular, subgraph of $G'$ whose density is more than $\mu-\eps/2>\mu/2$.

		We claim that $\Gamma_{\rm{ref}}$ contains a perfect matching. To see this, we first prove the following claim.
		A \emph{$2$-matching} is a collection of pairwise vertex-disjoint edges and odd cycles.
		A \emph{perfect $2$-matching} in a graph is a $2$-matching covering all vertices in the graph.

		\begin{claim} \label{claim:perfect-2-matching}
			$\Gamma$ has a perfect $2$-matching.
		\end{claim}

		Before we prove the claim, recall that a \emph{fractional matching} of a graph $G$ is an edge weighting $w :E(G) \to [0,1]$, such that $\sum_{u \in N_G(v)}w(uv) \le 1$, for every vertex $v$ in $G$. A \emph{perfect fractional matching} is a fractional matching $w$ where we have equality $\sum_{u \in N_G(v)}w(uv) = 1$ for every vertex $v$ in $G$.
		We need the following useful fact.

		\begin{fact} \label{fact:fractional-matching}
			Let $G$ be a graph which has a perfect fractional matching. Then it has a perfect $2$-matching.
		\end{fact}

		This is a standard fact, but for completeness we prove it in Appendix~\ref{fracmatchingimplies2matching}.

			\begin{proof}[ of Claim~\ref{claim:perfect-2-matching}]
				Let $n'=|V(G')|\ge(1-2\eps)n$. Properties~\ref{itm:Gprime-1}--\ref{itm:Gprime-3} satisfy Lemma~\ref{lem:hamilton-cycle} with $G'$, $n'$, $d$, $c/2$, $4\eta$, $\xi$, $\gamma/2$, $\zeta/4$, and $W=\emptyset$ playing the respective roles of $G,n,d,c,\eta,\xi,\gamma,\zeta,W$. Hence $G'$ has a Hamilton cycle, which we denote by $C = (v_1 \ldots v_{n'})$.
			Let $w : E(\Gamma) \to \bR^{\ge 0}$ be the edge-weighting where $w(V_iV_j)$ is the number of edges in $C$ with one end in $V_i$ and the other in $V_j$, divided by $2|V_1|$.
			We claim that $w$ is a perfect fractional matching in $\Gamma$; that is, $\sum_{j \in [m]} w(V_iV_j) = 1$ for every $i \in [m]$.
			Indeed, for every  $i \in [m]$, since $G'$ has no edges with both ends in $V_i$, the number of edges in $C$ with an end in $V_i$ is exactly $2|V_i| = 2|V_1|$, so $\sum_{j \in [m]} w(V_iV_j) = 1$ for every $i \in [m]$, as claimed.
			By \Cref{fact:fractional-matching}, this implies that $\Gamma$ has a perfect $2$-matching, as required.
		\end{proof}

		Let $\cM$ be a perfect $2$-matching in $\Gamma$.
		Replace each isolated edge $V_iV_j$ in $\cM$ with two vertex-disjoint edges $V_i'V_j''$ and $V_i''V_j'$ in $\Gamma_{\rm{ref}}$ and each odd cycle $(V_{i_1}, \ldots, V_{i_{2k+1}})$ (for some $k \ge 1$) in $\cM$ with the edges $V_{i_j}' V_{i_{j+1}}'' \in E(\Gamma_{\rm{ref}})$ for $j \in [2k+1]$ (where addition in the indices is taken modulo $2k+1$), resulting in a perfect matching  $\cM_{\rm{ref}}$ in $\Gamma_{\rm{ref}}$.

			Let $\widetilde U_1\widetilde U_2,\ldots,\widetilde U_{2m-1}\widetilde U_{2m}$ denote the edges of $\cM_{\rm{ref}}$. Every edge in $\Gamma_{\rm{ref}}$ corresponds to a $2\eps$-regular subgraph of $G'$ with density more than $\mu/2$. Thus, for every $i\in[m]$, the pair $G'[\widetilde U_{2i-1},\widetilde U_{2i}]$ is $2\eps$-regular with density more than $\mu/2$, and for all $i,j\in[2m]$, the pair $G'[\widetilde U_i,\widetilde U_j]$ is $2\eps$-regular with density either $0$ or more than $\mu/2$.

				For every $i\in[m]$, apply \Cref{claim:super-regular} to $G'[\widetilde U_{2i-1},\widetilde U_{2i}]$, with its $\eps$ replaced by $2\eps$. This gives sets $D_{2i-1}^0\subseteq \widetilde U_{2i-1}$ and $D_{2i}^0\subseteq \widetilde U_{2i}$, each of size at most $2\eps|\widetilde U_1|$, such that $G'[\widetilde U_{2i-1}\setminus D_{2i-1}^0,\widetilde U_{2i}\setminus D_{2i}^0]$ is $(3\eps,\mu/3)$-super-regular. Let $s_{\rm del}\coloneqq\max_{j\in[2m]}|D_j^0|$, and enlarge every $D_j^0$ to a set $D_j$ of size exactly $s_{\rm del}$. Since all the $\widetilde U_j$ have the same size and $s_{\rm del}\le2\eps|\widetilde U_1|$, deleting the sets $D_j$ leaves all clusters with the same size. Next, choose an integer $s_{\rm div}<2t$ so that $|\widetilde U_1|-s_{\rm del}-s_{\rm div}$ is divisible by $2t$. For every $j\in[2m]$, delete an additional $s_{\rm div}$ vertices from $\widetilde U_j\setminus D_j$, move all deleted vertices to $V_0$, and denote the resulting cluster by $U_j$. We continue to use $G'$ for the subgraph induced by $U_1\cup\cdots\cup U_{2m}$.

		We claim that after these deletions the following properties hold.
		\begin{enumerate}[resume, label = \rm(G'\arabic*)]
			\item \label{itm:Gprime-4}
				We have $\bigcup_{i=1}^{2m}U_i=V(G)\setminus V_0=V(G')$, $|U_1|=\cdots=|U_{2m}|$, and $|V_0|,|U_1|,\ldots,|U_{2m}|$ are divisible by $2t$. Furthermore, if $G$ is bipartite and balanced with bipartition $\{X,Y\}$, then we can assume that $\bigcup_{i=1}^{m}U_{2i-1}=X\setminus V_0$ and $\bigcup_{i=1}^{m}U_{2i}=Y\setminus V_0$ (and hence $G'$ is still bipartite and balanced).
			\item \label{itm:Gprime-5}
				For every $i\in[m]$, $G'[U_{2i-1},U_{2i}]$ is $(5\eps,\mu/5)$-super-regular.
			\item \label{itm:Gprime-6}
				For all $i,j\in[2m]$, the graph $G'[U_i,U_j]$ is $5\eps$-regular with density either $0$ or more than $\mu/5$.
		\end{enumerate}

		We now verify \ref{itm:Gprime-4}--\ref{itm:Gprime-6}. All $\widetilde U_j$ have the same size, $s_{\rm del}\le2\eps|\widetilde U_1|$, and $s_{\rm div}<2t\le\eps|\widetilde U_1|$, where the last inequality holds because $n_0$ is sufficiently large. Hence, for every $j\in[2m]$, $|U_j|=|\widetilde U_j|-s_{\rm del}-s_{\rm div}\ge(1-3\eps)|\widetilde U_j|$. Before these deletions, every pair $G'[\widetilde U_i,\widetilde U_j]$ with $i,j\in[2m]$ was $2\eps$-regular. Fix $i,j\in[2m]$ and arbitrary subsets $A\subseteq U_i$ and $B\subseteq U_j$ satisfying $|A|\ge5\eps|U_i|$ and $|B|\ge5\eps|U_j|$. Then $|A|\ge5\eps(1-3\eps)|\widetilde U_i|\ge2\eps|\widetilde U_i|$ and $|B|\ge5\eps(1-3\eps)|\widetilde U_j|\ge2\eps|\widetilde U_j|$. Moreover, $|U_i|\ge2\eps|\widetilde U_i|$ and $|U_j|\ge2\eps|\widetilde U_j|$. Thus the $2\eps$-regularity of $G'[\widetilde U_i,\widetilde U_j]$ applies to both pairs $(A,B)$ and $(U_i,U_j)$, giving
		\begin{align*}
			\bigl|d(A,B)-d(\widetilde U_i,\widetilde U_j)\bigr|&<2\eps, \tag{\theequation a}\label{eqn:Gprime-test-density-comparison}\\
			\bigl|d(U_i,U_j)-d(\widetilde U_i,\widetilde U_j)\bigr|&<2\eps. \tag{\theequation b}\label{eqn:Gprime-cluster-density-comparison}
		\end{align*}
		By the triangle inequality, $|d(A,B)-d(U_i,U_j)|<4\eps$, so $G'[U_i,U_j]$ is $4\eps$-regular, and hence $5\eps$-regular. If $d(\widetilde U_i,\widetilde U_j)=0$, then $d(U_i,U_j)=0$; otherwise $d(\widetilde U_i,\widetilde U_j)>\mu/2$, and \eqref{eqn:Gprime-cluster-density-comparison} gives $d(U_i,U_j)>\mu/2-2\eps>\mu/5$. Thus \ref{itm:Gprime-6} holds.

				For every $j\in[2m]$, write $W_j^0\coloneqq\widetilde U_j\setminus D_j^0$. For each $i\in[m]$, Proposition~\ref{claim:super-regular} was used to choose $D_{2i-1}^0$ and $D_{2i}^0$ so that $G'[W_{2i-1}^0,W_{2i}^0]$ is $(3\eps,\mu/3)$-super-regular. Hence $d_{G'}(v,W_{2i}^0)\ge(\mu/3)|W_{2i}^0|$ for every $v\in W_{2i-1}^0$ and $d_{G'}(v,W_{2i-1}^0)\ge(\mu/3)|W_{2i-1}^0|$ for every $v\in W_{2i}^0$. We have $|W_j^0|\ge(1-2\eps)|\widetilde U_1|$. After forming $W_j^0$, at most $(s_{\rm del}-|D_j^0|)+s_{\rm div}\le2\eps|\widetilde U_1|+2t$ additional vertices are deleted from cluster $j$. Consequently, every $v\in U_{2i-1}$ satisfies
			\begin{equation*}
			d_{G'}(v,U_{2i})\ge(\mu/3)(1-2\eps)|\widetilde U_1|-2\eps|\widetilde U_1|-2t>(\mu/5)|\widetilde U_1|\ge(\mu/5)|U_{2i}|,
			\tag{\theequation c}\label{eqn:Gprime-superregular-left-degree}
			\end{equation*}
				where the strict inequality follows from $\eps\ll\mu$ and sufficiently large $n_0$. Likewise, every $v\in U_{2i}$ satisfies
				\begin{equation*}
				d_{G'}(v,U_{2i-1})\ge(\mu/3)(1-2\eps)|\widetilde U_1|-2\eps|\widetilde U_1|-2t>(\mu/5)|U_{2i-1}|.
				\tag{\theequation d}\label{eqn:Gprime-superregular-right-degree}
				\end{equation*}
				By \ref{itm:Gprime-6}, the graph $G'[U_{2i-1},U_{2i}]$ is $5\eps$-regular. Moreover, \eqref{eqn:Gprime-test-density-comparison} and \eqref{eqn:Gprime-cluster-density-comparison}, applied to the pair $(U_{2i-1},U_{2i})$, together with $d(\widetilde U_{2i-1},\widetilde U_{2i})>\mu/2$, imply that, for all $A\subseteq U_{2i-1}$ and $B\subseteq U_{2i}$ with $|A|\ge5\eps|U_{2i-1}|$ and $|B|\ge5\eps|U_{2i}|$,
				\begin{equation*}
					d(A,B)>d(U_{2i-1},U_{2i})-4\eps>\mu/2-6\eps>\mu/5.
				\tag{\theequation e}\label{eqn:Gprime-superregular-subset-density}
				\end{equation*}
				Thus \eqref{eqn:Gprime-superregular-subset-density} verifies the density condition in the definition of super-regularity, while \eqref{eqn:Gprime-superregular-left-degree} and \eqref{eqn:Gprime-superregular-right-degree} verify its two minimum-degree conditions. This proves \ref{itm:Gprime-5}.

		To see \ref{itm:Gprime-4}, recall that the original clusters $V_i$ have the same size and hence all refined clusters $\widetilde U_j$ have the same size. We removed the same number $s_{\rm del}+s_{\rm div}$ of vertices from each of them. The choice of $s_{\rm div}$ gives the required divisibility, and then $|V_0|$ is divisible by $2t$ because $n$ and every $|U_j|$ are. If $G$ is bipartite and balanced, then $\Gamma$ is bipartite and balanced, so its perfect $2$-matching is in fact a perfect matching. We may therefore label the endpoints of each edge $\widetilde U_{2i-1}\widetilde U_{2i}$ of $\cM_{\rm ref}$ so that $\widetilde U_{2i-1}\subseteq X$ and $\widetilde U_{2i}\subseteq Y$ for every $i\in[m]$. Since the same number of vertices is removed from every cluster, the asserted balanced bipartition for the resulting clusters $U_{2i-1},U_{2i}$ is preserved.

		At most $m$, $2\eps n$, and $4tm$ vertices, respectively, are moved to $V_0$ to make the $V_i$ even, to delete the sets $D_j$, and to make the final cluster sizes divisible by $2t$. Thus at most $m+2\eps n+4tm\le4\eps n$ vertices move from $G'$ to $V_0$, so~\eqref{eqn:boundV0} gives
		\begin{equation}
			\label{eqn:boundV0aftermove}
			|V_0| \le 6\eps n,
		\end{equation}

			Recall that $G_0'$ is the graph at the beginning of this subsection. Let $k\coloneqq|V(G_0')\setminus V(G')|$ be the number of vertices moved from it to $V_0$, so $k\le4\eps n$.
		Since that graph has average degree at least $d-3\eta n$ by \ref{itm:Gprime-1} and maximum degree at most $d$, the degree sum after the deletions is at least $n_0'(d-3\eta n)-2kd$. Dividing by the new order $n_0'-k$ and using $n_0'\ge(1-2\eps)n$ and $\eps\ll\eta,c$ gives the following.
		\begin{enumerate}[resume, label = \rm(G'\arabic*)]
			\item \label{itm:Gprime-7}
				$G'$ has average degree at least $d - 6 \eta n$.
		\end{enumerate}
		Furthermore, we claim that the following two properties hold.
		\begin{enumerate}[resume, label = \rm(G'\arabic*)]
			\item \label{itm:Gprime-8}
				$G'$ has no $\zeta /12$-sparse cuts. Moreover, $G'$ has minimum degree at least $\zeta n/6.$
			\item \label{itm:Gprime-9}
				If $G$ is bipartite and balanced, then $G'$ is  bipartite and balanced. Otherwise, one needs to remove at least $\gamma n^2/4$ edges from $G'$ to make it bipartite.
		\end{enumerate}

		We now prove \ref{itm:Gprime-8}. By \ref{itm:Gprime-2}, every vertex of $G_0'$ has degree at least $\zeta n/5$, so the final graph $G'$ has minimum degree at least $\zeta n/5-4\eps n\ge\zeta n/6$. Consider a partition $\{X,Y\}$ of the final graph with $|X|\le|Y|$. If $|X|\le\zeta n/12$, the minimum-degree bound gives $e_{G'}(X,Y)\ge(\zeta n/12)|X|\ge(\zeta/12)|X||Y|$.

		Otherwise, put $D\coloneqq V(G_0')\setminus V(G')$ and $Y_0\coloneqq Y\cup D$. Then $\{X,Y_0\}$ is a partition of $V(G_0')$, and $|D|\le4\eps n$. Property~\ref{itm:Gprime-2} applied to this partition gives
		\[
			e_{G_0'}(X,Y_0)\ge(\zeta/4)|X||Y_0|\ge(\zeta/4)|X||Y|.
		\]
		The graph $G'$ is obtained from $G_0'$ by deleting the vertices in $D$. Hence the only edges counted by $e_{G_0'}(X,Y_0)$ that are not counted by $e_{G'}(X,Y)$ are edges between $X$ and $D$. Consequently,
		\[
			e_{G'}(X,Y)
			\ge e_{G_0'}(X,Y_0)-e_{G_0'}(X,D)
			\ge(\zeta/4)|X||Y|-|X||D|
			\ge(\zeta/4)|X||Y|-4\eps n^2
			\ge(\zeta/12)|X||Y|,
		\]
		where the last inequality follows from $|X|,|Y|\ge\zeta n/12$ and $\eps\ll\zeta$. This proves \ref{itm:Gprime-8}. 

		Finally, if $G$ is bipartite and balanced, then $G'$ is bipartite and balanced by \ref{itm:Gprime-4}. Otherwise, \ref{itm:Gprime-3} shows that at least $\gamma n^2/2$ edges had to be deleted from $G_0'$ to make it bipartite, and moving at most $4\eps n$ vertices removes at most $4\eps n^2\le\gamma n^2/4$ edges. Thus at least $\gamma n^2/4$ further edges must be removed from $G'$, proving \ref{itm:Gprime-9}.

	\subsection{Splitting the super-regular subgraphs}

		Property~\ref{itm:Gprime-5} gives vertex-disjoint super-regular subgraphs $G'[U_{2i-1}, U_{2i}]$ covering every vertex outside $V_0$. We next show that
		if we randomly partition every $U_i$ into five sets $U_i^{(1)}, \ldots, U_i^{(5)}$ of the sizes specified in \ref{itm:U-1}--\ref{itm:U-4}, then the subgraph $H$ of $G'$ induced by $\bigcup_{i \in [2m]}U_i^{(3)}$ has no $\zeta/20$-sparse cut by \ref{itm:U-6}. Moreover, by \ref{itm:U-7}, it is balanced bipartite when $G$ is balanced bipartite, and otherwise at least $\gamma n^2/(32t^2)$ edges must be removed from it to make it bipartite.

		\begin{claim}
			For every $i \in [2m]$, there is a partition $\left\{U_i^{(1)}, \ldots, U_i^{(5)}\right\}$ of $U_i$ satisfying the following properties.
			\begin{enumerate}[label = \rm(U\arabic*)]
				\item \label{itm:U-1}
					$\big|U_i^{(1)}\big| = \big|U_i^{(2)}\big| = \xi |U_i|$ for $i \in [2m]$.

				\item \label{itm:U-2}
					$\big|U_i^{(3)}\big| = \frac{|U_i|}{2t}$ for $i \in [2m]$.

				\item \label{itm:U-3}
					$\big|U_i^{(4)}\big| = \frac{2 |U_i|}{3}$ for $i \in [2m]$.

				\item \label{itm:U-4}
					$\big|U_i^{(5)}\big| = (\frac{1}{3} - \frac{1}{2t} - 2\xi)|U_i|$ for $i \in [2m]$.

				\item \label{itm:U-5}
					For every $v \in V(G)$, $i \in [2m]$ and $j \in [5]$, if $d_G(v, U_i) \ge \eps|U_i|$, then $\frac{d_G\big(v, U_i^{(j)}\big)}{d_G(v,U_i)} = (1 \pm \eps) \frac{|U_i^{(j)}|}{|U_i|}$.

					For degrees measured in $G'$, for every $v \in V(G')$, $i \in [2m]$ and $j \in [5]$, if $d_{G'}(v, U_i) \ge \eps |U_i|$ then $\frac{d_{G'}\big(v, U_i^{(j)}\big)}{d_{G'}(v,U_i)} = (1 \pm \eps) \frac{|U_i^{(j)}|}{|U_i|}$.

				\item \label{itm:U-6}
					The graph $H \coloneqq G'\left[\bigcup_{i \in [2m]}U_i^{(3)}\right]$ has no $\frac{\zeta}{20}$-sparse cuts.
				\item \label{itm:U-7}
					If $G$ is bipartite and balanced, then $H = G'\left[\bigcup_{i \in [2m]}U_i^{(3)}\right]$ is also bipartite and balanced. Otherwise, one needs to remove at least $\frac{\gamma}{32t^2} n^2$ edges from $H$ to make it bipartite.
			\end{enumerate}
		\end{claim}

		\begin{proof}[ of claim]
			For every $i \in [2m]$, let $\left\{U_i^{(1)}, \ldots, U_i^{(5)}\right\}$ be a random partition of $U_i$ into sets of sizes $\xi |U_i|$, $\xi |U_i|$, $(1/2t)|U_i|$, $(2/3)|U_i|$, $(1/3 - 1/2t - 2\xi)|U_i|$, respectively.  Then \ref{itm:U-1}--\ref{itm:U-4} hold by definition.
			Moreover, by \Cref{hypergeometricconcentration} and a union bound, \ref{itm:U-5} holds with probability at least $1 - o(1)$; fix an outcome such that it holds.

			For proving \ref{itm:U-6}, let $\{X, Y\}$ be a partition of $V(H)$ with $|X| \le |Y|$. Fix $v\in V(H)$ and let $I(v)=\{i\in[2m]:d_{G'}(v,U_i)\ge\eps|U_i|\}$. The total number of neighbours of $v$ in clusters whose indices are not in $I(v)$ is less than $\eps|V(G')|$. Therefore, by \ref{itm:U-2}, \ref{itm:U-5}, and the minimum-degree bound in \ref{itm:Gprime-8},
			\[
			d_H(v)
			\ge\frac{1-\eps}{2t}\bigl(d_{G'}(v)-\eps|V(G')|\bigr)
			\ge\frac{1-\eps}{2t}\bigl(\zeta n/6-\eps n\bigr)
			\ge\frac{\zeta}{7}|V(H)|,
			\]
			using that $|V(H)|=|V(G')|/(2t)$. Thus, if $|X| \le (\zeta/14) \cdot |V(H)|$, then every vertex in $X$ has at least $(\zeta/14) \cdot |V(H)|$ neighbours in $Y$, showing that $e_H(X,Y) \ge |X| \cdot (\zeta/14) \cdot |V(H)| \ge (\zeta / 14) \cdot |X| \cdot |Y|$, as required.

			So suppose that $|X| \ge (\zeta / 14) \cdot |V(H)|$.
			For every $i\in[2m]$, choose $X_i' \subseteq U_i$ to have size $2t|X\cap U_i^{(3)}|$, and set $X' := \bigcup_{i \in [2m]}X_i'$ and $Y' := V(G') \setminus X'$. Then $\{X',Y'\}$ is a partition of $V(G')$ such that, for every $i \in [2m]$, we have
            \begin{align}
				\label{choosingXiY'proportionally}
				\begin{split}
					|X' \cap U_i| &= |X_i'| = 2t|X \cap U_i^{(3)}| = 2t|X \cap U_i| \\
					|Y' \cap U_i| &= |U_i| - |X' \cap U_i| = 2t(|U_i^{(3)}| - |X \cap U_i^{(3)}|) = 2t |Y \cap U_i^{(3)}| = 2t|Y \cap U_i|.
				\end{split}
            \end{align}
			We claim that for any $i, j \in [2m]$, we have
			\begin{equation} \label{eqn:HtoG'}
				e_H(X \cap U_i, Y \cap U_j)
				\ge \frac{1}{4t^2} \cdot e_{G'}(X' \cap U_i, Y' \cap U_j) - 10\eps |U_i| |U_j|
			\end{equation}

			Indeed, if $|X' \cap U_i| \le 10t\eps |U_i|$ or $|Y' \cap U_j| \le 10t\eps |U_j|$, this holds trivially, as then the right-hand side of \eqref{eqn:HtoG'} is at most $(1/4t^2) \cdot 10t\eps |U_i| |U_j| - 10\eps |U_i| |U_j| < 0$. If, instead, $|X' \cap U_i| \ge 10t\eps |U_i|$ and $|Y' \cap U_j| \ge 10t\eps|U_j|$, then $|X \cap U_i| \ge 5\eps |U_i|$ and $|Y \cap U_j| \ge 5\eps |U_j|$ by \eqref{choosingXiY'proportionally} and \ref{itm:U-2}, so by two applications of the $5\eps$-regularity of $G'[U_i,U_j]$ (one to $X\cap U_i^{(3)},Y\cap U_j^{(3)}$ and one to $X'\cap U_i,Y'\cap U_j$), followed by the triangle inequality, we have
			\begin{equation*}
				\frac{e_H(X \cap U_i, Y \cap U_j)}{|X \cap U_i| \cdot |Y \cap U_j|}
				\ge  \frac{e_{G'}(X' \cap U_i, Y' \cap U_j)}{|X' \cap U_i|\cdot| Y' \cap U_j|} - 10\eps.
			\end{equation*}
			This implies that
			\begin{align*}
				e_H(X \cap U_i, Y \cap U_j)
				& \ge e_{G'}(X' \cap U_i, Y' \cap U_j) \cdot \frac{|X \cap U_i|}{|X' \cap U_i|} \cdot \frac{|Y \cap U_j|}{|Y' \cap U_j|} - 10\eps \cdot|X \cap U_i| \cdot |Y \cap U_j| \\
				& \overset{\eqref{choosingXiY'proportionally}}{\ge} \frac{1}{4t^2} \cdot e_{G'}(X' \cap U_i, Y' \cap U_j) - 10\eps |U_i| |U_j|,
			\end{align*}
			as required for \eqref{eqn:HtoG'}.

			Hence, we obtain the following lower bound on the number of edges between $X$ and $Y$ in $H$.
			\begin{align*}
				e_H(X,Y)
				& = \sum_{i,j \in [2m]}e_H(X \cap U_i, Y \cap U_j) \\
				& \overset{\eqref{eqn:HtoG'}}{\ge} \frac{1}{4t^2} \cdot \sum_{i,j \in [2m]} e_{G'}(X' \cap U_i, Y' \cap U_j)
				- 10\eps \cdot \sum_{i,j \in [2m]} (|U_i| \cdot |U_j|) \\
				& = \frac{1}{4t^2} \cdot e_{G'}(X', Y') - 10 \eps \cdot |V(G')|^2 \\
				& \overset{\ref{itm:Gprime-8}}{\ge} \frac{1}{4t^2} \cdot \frac{\zeta}{12} \cdot |X'| \cdot |Y'| - 10\eps n^2 \\
				& = \frac{\zeta}{12} \cdot |X| \cdot |Y| - 10\eps n^2 \\
				& \ge \frac{\zeta}{20} \cdot |X| \cdot |Y|,
			\end{align*}
			using that $|X|, |Y| \ge (\zeta/14)|V(H)| \ge (\zeta/56t)n$ and that $\eps \ll 1/t, \zeta$ for the last inequality.
			This proves \ref{itm:U-6}.

			Finally, it remains to prove \ref{itm:U-7}.
			Note that by \ref{itm:Gprime-4} and \ref{itm:U-2}, if $G$ is bipartite and balanced, then $H$ is also bipartite and balanced, as desired. Now suppose that $G$ is $\gamma$-far-from-bipartite. Then, by \ref{itm:Gprime-9}, one needs to remove at least $\gamma n^2 /4$ edges from $G'$ to make it bipartite. Let $\{X, Y\}$ be an arbitrary partition of $V(H)$. To prove \ref{itm:U-7}, we need to show that $e_H(X)+e_H(Y)\ge\frac{\gamma n^2}{32t^2}$.
			Define the sets $X', Y'$ as in \eqref{choosingXiY'proportionally}. The proof of \eqref{eqn:HtoG'} applies with $X\cap U_j$ and $X'\cap U_j$ in place of $Y\cap U_j$ and $Y'\cap U_j$, respectively. Thus, for every $i, j \in [2m]$,
			\begin{equation}
				\label{eqn:HtoG'forX'}
				e_H(X \cap U_i, X \cap U_j)
				\ge \frac{1}{4t^2} \cdot e_{G'}(X' \cap U_i, X' \cap U_j) - 10\eps |U_i| |U_j|.
			\end{equation}
			It follows that
			\begin{align*}
				e_H(X)
				& = \sum_{1 \le i \le j \le 2m}e_H(X \cap U_i,X \cap U_j) \\
				& \overset{\eqref{eqn:HtoG'forX'}}{\ge} \frac{1}{4t^2} \cdot \sum_{1 \le i \le j \le 2m}e_{G'}(X'\cap U_i,X'\cap U_j)
				- 10\eps \cdot \sum_{1 \le i \le j \le 2m}(|U_i|\cdot|U_j|) \\
				& \ge \frac{1}{4t^2} \cdot e_{G'}(X') - 10 \eps n^2.
			\end{align*}
			Replacing $X,X'$ by $Y,Y'$ in \eqref{eqn:HtoG'forX'} and in the preceding summation gives $e_H(Y) \ge \frac{1}{4t^2} \cdot e_{G'}(Y') - 10\eps n^2$.
			Notice that $e_{G'}(X')+e_{G'}(Y')\ge\gamma n^2/4$ by \ref{itm:Gprime-9} and the assumption that $G$ is not bipartite.
			Hence, altogether, we have
			\begin{align*}
				e_H(X) + e_H(Y)
				\ge \frac{1}{4t^2}\big(e_{G'}(X') + e_{G'}(Y')\big) - 20\eps n^2
				\ge \frac{\gamma}{16t^2} \cdot n^2 - 20\eps \cdot n^2
				\ge \frac{\gamma}{32t^2}\cdot n^2,
			\end{align*}
			using $\eps \ll \gamma, 1/t$ for the last inequality. This proves \ref{itm:U-7}.
		\end{proof}

		\subsection{Covering $V_0$ using a well-distributed collection of copies of $K_{t,t}$}

			In the rest of the proof, let $U^{(j)} \coloneqq \bigcup_{i \in [2m]}U_i^{(j)}$ for $j \in [5]$.
			In the next two subsections, we find a collection of vertex-disjoint copies of $K_{t,t}$ in $G\left[V_0\cup U^{(1)}\cup U^{(2)}\right]$ that covers $V_0$ and ensures that the number of uncovered vertices in every $U_i$ is divisible by $t$.

        First, we cover the vertices in $V_0$ using vertex-disjoint $K_{t,t}$'s in $G\left[V_0 \cup U^{(1)}\right]$ (ignoring the divisibility requirement).

			\begin{claim} \label{claim:K1}
				There is a collection $\cK_1$ of pairwise vertex-disjoint $K_{t,t}$'s in $G[V_0 \cup U^{(1)}]$ such that $V_0 \subseteq V(\cK_1)$ and each copy of $K_{t,t}$ in $\cK_1$ contains exactly one vertex of $V_0$.
			\end{claim}

			\begin{proof}
				Let $\cK_1$ be a maximal collection of pairwise vertex-disjoint copies of $K_{t,t}$ in $G[V_0 \cup U^{(1)}]$, where each copy of $K_{t,t}$ contains exactly one vertex from $V_0$.
				We will show that $V_0 \subseteq V(\cK_1)$, proving the claim.
				Suppose not, and let $v \in V_0 \setminus V(\cK_1)$.
				Write $U \coloneqq U^{(1)}$ and $U' \coloneqq U \setminus V(\cK_1)$.

				We claim that every vertex $u$ in $G$ satisfies $d_G(u, U') \ge (\xi \zeta / 8) n$.
				Indeed, given a vertex $u$ in $G$, let $I$ be the set of indices $i \in [2m]$ such that $d_G(u, U_i) \ge \eps |U_i|$.
				Then
				\begin{align*}
					d_G(u, U)
					= \sum_{i \in [2m]}d_G\big(u, U_i^{(1)}\big)
					& \overset{\ref{itm:U-5}}{\ge} \sum_{i \in I}(1 - \eps) \cdot \frac{\big|U_i^{(1)}\big|}{|U_i|} \cdot d_G(u, U_i) \\
					& \ge (1 - \eps) \cdot \xi \cdot \left(d_G(u) - d_G(u, V_0) - \sum_{i \in [2m] \setminus I} d_G(u, U_i) \right) \\
					& \ge (1 - \eps) \cdot \xi \cdot \left(\frac{\zeta n}{2} - |V_0| - \sum_{i \in [2m] \setminus I} \eps |U_i| \right) \\[.5em]
					& \overset{\eqref{eqn:boundV0aftermove}}{\ge} (1 - \eps) \cdot \xi \cdot \left (\frac{\zeta}{2} - 7\eps \right ) \cdot n
					\ge \frac{\xi \zeta}{4} \cdot n.
				\end{align*}
				For the third inequality we used property~\ref{itm:pack-expander-cuts}, which gives that $G$ has minimum degree at least $\zeta(n-1) \ge \frac{\zeta}{2} n$, which follows from the fact that there are no $\zeta$-sparse cuts in $G$.
				Since $|V(\cK_1)| \le 2t |V_0| \le 12\eps t n$, every vertex $u$ in $G$ satisfies $d_G(u, U') \ge d_G(u, U) - |V(\cK_1)| \ge (\frac{\xi \zeta}{4} - 12\eps t)n \ge (\xi \zeta / 8)n$, as desired. Here the final inequality uses $\eps\ll\xi,\zeta,1/t$.

				Fix a subset $W \subseteq N(v, U')$ of size $(\xi \zeta / 16)n$ and write $Z \coloneqq U' \setminus W$. Since every vertex $u$ in $G$ satisfies $d_G(u, U') \ge (\xi \zeta / 8) n$, every vertex in $W$ has at least $(\xi \zeta / 16) n$ neighbours in $Z$.
				Set $\delta^*\coloneqq\xi\zeta/16$ and apply \Cref{lem:unbalanced-kst} to $G[W,Z]$, taking $(X,Y,\delta)=(W,Z,\delta^*)$. The sets $W,Z$ satisfy $|W|=\delta^*n$ and every vertex of $W$ has at least $\delta^*n$ neighbours in $Z$. Moreover, since $n\ge n_0$ and $n_0$ is sufficiently large in terms of $\delta^*$ and $t$, we have $|W|=\delta^*n\ge2t(e/\delta^*)^t$. Hence $G[W,Z]$ has a $K_{t,t}$, say, with vertex set $A \cup B$ where $A \subseteq W$ and $B \subseteq Z$. Let $B'$ be a subset of $B$ of size $t-1$. Then $G[A, B' \cup \{v\}]$ is a copy of $K_{t,t}$ in $G[V_0 \cup U^{(1)}]$, which is vertex-disjoint from $V(\mathcal{K}_1)$ and contains exactly one vertex from $V_0$, contradicting the maximality of $\mathcal{K}_1$.
			\end{proof}

\subsection{Constructing copies of $K_{t,t}$ to ensure divisibility}
\label{subsec:Kttdivisibility}

			Next, we find a small number of vertex-disjoint copies of $K_{t,t}$ in $G[U^{(2)}]$ that correct the residue of the number of uncovered vertices in every cluster. Recall that the \emph{support} of a function $\varphi:A\to B$, denoted $\supp(\varphi)$, is the set of elements $a\in A$ with $\varphi(a)\ne0$.
			We first prove that, for any function $\varphi:[2m]\to\{0,\ldots,t-1\}$ satisfying $\sum_{i\in[2m]}\varphi(i)\equiv0\pmod t$ and, when $G$ is bipartite, $\sum_{i\in[m]}\varphi(2i)\equiv0\pmod t$, there is a collection $\cK_2$ of pairwise vertex-disjoint copies of $K_{t,t}$ in $G[U^{(2)}]$ such that, for every $i\in[2m]$, $|U_i\cap V(\cK_2)|\equiv\varphi(i)\pmod t$ and $|U_i\cap V(\cK_2)|\le3t|\supp(\varphi)|$. We then take the collection $\cK_1$ from \Cref{claim:K1}. For the fixed graph $G$ and this choice of $\cK_1$, let $\varphi_G(i)$ be the unique element of $\{0,\ldots,t-1\}$ satisfying $\varphi_G(i)\equiv-|U_i\cap V(\cK_1)|\pmod t$. After verifying that $\varphi_G$ satisfies the two congruence assumptions above, we apply \Cref{claim:K2} with $\varphi=\varphi_G$ to choose $\cK_2$. It follows that $|U_i\cap V(\cK_1\cup\cK_2)|\equiv0\pmod t$ for every $i\in[2m]$; since $|U_i|$ is divisible by $t$, the number of vertices of $U_i$ left uncovered by $\cK_1\cup\cK_2$ is also divisible by $t$.

			The following claim is used in the induction below to embed a copy of $K_{t,t}$ with one vertex class in a prescribed cluster and the other split between two clusters that are each adjacent to the first cluster in $\Gamma_{\rm ref}$.

			\begin{claim}\label{claim:three-cluster-embedding}
					Let $i,j,k\in[2m]$ satisfy $\widetilde U_i\widetilde U_j,\widetilde U_i\widetilde U_k\in E(\Gamma_{\rm ref})$, where $j=k$ is allowed. Let $A\subseteq U_i^{(2)}$, $B\subseteq U_j^{(2)}$, and $C\subseteq U_k^{(2)}$ satisfy
				\[
				|A|\ge \frac{|U_i^{(2)}|}{2},\qquad |B|\ge \frac{|U_j^{(2)}|}{2},\qquad |C|\ge \frac{|U_k^{(2)}|}{2},
				\]
				and take $B=C$ when $j=k$. For every $p\in\{1,\ldots,t-1\}$, there is a copy of $K_{t,t}$ with one part consisting of $t$ vertices of $A$ and the other part consisting of $p$ vertices of $B$ and $t-p$ vertices of $C$; when $j=k$, these $t$ vertices are chosen as distinct vertices of $B=C$.
			\end{claim}

			\begin{proof}
				Because $\widetilde U_i\widetilde U_j,\widetilde U_i\widetilde U_k\in E(\Gamma_{\rm ref})$, both $G'[\widetilde U_i,\widetilde U_j]$ and $G'[\widetilde U_i,\widetilde U_k]$ have density more than $\mu/2$. By \eqref{eqn:Gprime-cluster-density-comparison}, the pairs $G'[U_i,U_j]$ and $G'[U_i,U_k]$ consequently have density more than $\mu/2-2\eps>\mu/5$; by \ref{itm:Gprime-6}, both pairs are also $5\eps$-regular. Starting with $B_0=B$ and $C_0=C$, greedily choose distinct vertices $a_1,\ldots,a_t\in A$. After choosing $a_1,\ldots,a_s$ for some $0\le s<t$, maintain
				\[
				|B_s|\ge(\mu/10)^s|B|,\qquad |C_s|\ge(\mu/10)^s|C|,
				\]
					where $B_s$ and $C_s$ are the common neighbourhoods of $a_1,\ldots,a_s$ in $B$ and $C$, respectively. By \ref{itm:U-1} and the hypotheses of the claim, $|B|\ge(\xi/2)|U_j|$ and $|C|\ge(\xi/2)|U_k|$. Hence, for every $s<t$, $|B_s|\ge(\mu/10)^s(\xi/2)|U_j|\ge5\eps|U_j|$ and $|C_s|\ge(\mu/10)^s(\xi/2)|U_k|\ge5\eps|U_k|$, where the last inequalities use $\eps\ll\xi,\mu,1/t$. Since $G'[U_i,U_j]$ is $5\eps$-regular with density more than $\mu/5$, $|B_s|\ge5\eps|U_j|$, and $\mu/5-5\eps\ge\mu/10$, at most $5\eps|U_i|$ vertices $v\in U_i$ satisfy $d_{G'}(v,B_s)<(\mu/10)|B_s|$. Similarly, since $G'[U_i,U_k]$ is $5\eps$-regular with density more than $\mu/5$ and $|C_s|\ge5\eps|U_k|$, at most $5\eps|U_i|$ vertices $v\in U_i$ satisfy $d_{G'}(v,C_s)<(\mu/10)|C_s|$. Also, $|A|-10\eps|U_i|\ge t$ because $|A|\ge(\xi/2)|U_i|$, $\eps\ll\xi,1/t$, and $n_0$ is sufficiently large. Hence, as $s<t$, we can choose a vertex $a_{s+1}\in A\setminus\{a_1,\ldots,a_s\}$ with $d_{G'}(a_{s+1},B_s)\ge(\mu/10)|B_s|$ and $d_{G'}(a_{s+1},C_s)\ge(\mu/10)|C_s|$. Set $B_{s+1}:=B_s\cap N_{G'}(a_{s+1})$ and $C_{s+1}:=C_s\cap N_{G'}(a_{s+1})$. Then
					\[
						|B_{s+1}|\ge\frac{\mu}{10}|B_s|\ge\left(\frac{\mu}{10}\right)^{s+1}|B|,
						\qquad
						|C_{s+1}|\ge\frac{\mu}{10}|C_s|\ge\left(\frac{\mu}{10}\right)^{s+1}|C|.
					\]
					Thus the induction hypothesis holds with $s+1$ in place of $s$. Hence, after choosing $a_1,\ldots,a_t$, the common neighbourhood $B_t$ of these vertices in $B$ and the common neighbourhood $C_t$ of these vertices in $C$ each contain more than $t$ vertices. If $j\ne k$, choose $p$ vertices from $B_t$ and $t-p$ vertices from $C_t$. If $j=k$, then $B_t=C_t$; choose $t$ distinct vertices from this set and partition them into sets of sizes $p$ and $t-p$. Together with $a_1,\ldots,a_t$, the chosen vertices form the asserted copy of $K_{t,t}$.
			\end{proof}

			\begin{claim}\label{claim:K2}
				Let $\varphi:[2m]\to\{0,\ldots,t-1\}$ satisfy $\sum_{i\in[2m]}\varphi(i)\equiv0\pmod t$, and, if $G$ is bipartite, $\sum_{i\in[m]}\varphi(2i)\equiv0\pmod t$. Then there is a collection $\cK_2$ of pairwise vertex-disjoint copies of $K_{t,t}$ in $G[U^{(2)}]$ such that
				\[
				|U_i\cap V(\cK_2)|\le3t\,|\supp(\varphi)|
				\quad\text{and}\quad
				|U_i\cap V(\cK_2)|\equiv\varphi(i)\pmod t
				\]
				for every $i\in[2m]$.
			\end{claim}

			\begin{proof}
				We induct on $|\supp(\varphi)|$. The assertion is immediate when the support is empty. Otherwise, choose $a$ with $\varphi(a)\ne0$. Since the total sum is $0$ modulo $t$, there is a second index $b\ne a$ with $\varphi(b)\ne0$; when $G$ is bipartite, the sum over the indices of the same parity as $a$ is also $0$ modulo $t$, so $b$ may be chosen with the same parity as $a$.

				Define $\varphi':[2m]\to\{0,\ldots,t-1\}$ by
				\[
				\varphi'(a)=0,\qquad
				\varphi'(b)\equiv\varphi(a)+\varphi(b)\pmod t,
				\qquad
				\varphi'(i)=\varphi(i)\ \text{for }i\notin\{a,b\},
				\]
				where $\varphi'(b)$ is the representative in $\{0,\ldots,t-1\}$. Moreover, $|\supp(\varphi')|<|\supp(\varphi)|$. The total sum remains $0$ modulo $t$; when $G$ is bipartite, the even-index sum also remains $0$ modulo $t$. The induction hypothesis therefore gives a collection $\cK'$ of pairwise vertex-disjoint copies of $K_{t,t}$ in $G[U^{(2)}]$ such that, for every $h\in[2m]$,
				\[
					|U_h\cap V(\cK')|\le3t|\supp(\varphi')|
					\quad\text{and}\quad
					|U_h\cap V(\cK')|\equiv\varphi'(h)\pmod t.
				\]
				We will add copies of $K_{t,t}$ to $\cK'$ to obtain a collection $\cK_2$ satisfying these two conclusions with $\varphi$ in place of $\varphi'$.

				We claim that $\Gamma_{\rm ref}$ has an even-length $\widetilde U_a$--$\widetilde U_b$ walk visiting each vertex at most three times. By \ref{itm:Gprime-8} it is connected. If $G$ is bipartite, then $a$ and $b$ have the same parity, so any $\widetilde U_a$--$\widetilde U_b$ path proves the claim. If $G$ is not bipartite, neither is $\Gamma_{\rm ref}$. Indeed, a bipartition of $\Gamma_{\rm ref}$ would partition $[2m]$ into sets $I,J$ such that every edge of $\Gamma_{\rm ref}$ has one endpoint $\widetilde U_i$ with $i\in I$ and the other endpoint $\widetilde U_j$ with $j\in J$. Property~\ref{itm:Gprime-4} would then make $\{\bigcup_{i\in I}U_i,\bigcup_{j\in J}U_j\}$ a bipartition of $G'$, because every edge of $G'$ joins clusters whose corresponding original refined clusters are adjacent in $\Gamma_{\rm ref}$. This contradicts \ref{itm:Gprime-9}. Fix a vertex $\widetilde U_c$ on an odd cycle $C$ of $\Gamma_{\rm ref}$ and shortest $\widetilde U_a$--$\widetilde U_c$ and $\widetilde U_b$--$\widetilde U_c$ paths $P_a,P_b$. Follow $P_a$ from $\widetilde U_a$ to $\widetilde U_c$. If $P_a$ and $P_b$ have the same parity, do not traverse $C$; if they have different parity, traverse $C$ once. Then follow $P_b$ backwards from $\widetilde U_c$ to $\widetilde U_b$. The length of the resulting walk is $|P_a|+|P_b|$ in the first case and $|P_a|+|C|+|P_b|$ in the second, so it is even in both cases. This walk visits each vertex at most three times.

				Write this walk as $\widetilde U_{i_0},\ldots,\widetilde U_{i_{2\ell}}$, with $i_0=a$ and $i_{2\ell}=b$. We now construct pairwise disjoint copies $K_1,\ldots,K_\ell$, all vertex-disjoint from $\cK'$, where $K_j$ uses $t$ vertices of $U_{i_{2j-1}}^{(2)}$, $\varphi(a)$ vertices of $U_{i_{2j-2}}^{(2)}$, and $t-\varphi(a)$ vertices of $U_{i_{2j}}^{(2)}$. Suppose that $K_1,\ldots,K_{j-1}$ have already been chosen. The size bound in the induction hypothesis and $|\supp(\varphi')|\le2m$ give $|U_h\cap V(\cK')|\le6tm$ for every $h\in[2m]$. Moreover, $K_1,\ldots,K_{j-1}$ together use at most $3t$ vertices of any cluster, because the walk visits each cluster at most three times and the copies associated with each occurrence use at most $t$ vertices from that cluster. Since $|U_h^{(2)}|=\xi|U_h|$ for every $h\in[2m]$ and $n_0$ is sufficiently large, for each $h\in\{i_{2j-1},i_{2j-2},i_{2j}\}$ we have
				\[
					|U_h^{(2)}\setminus V(\cK'\cup K_1\cup\cdots\cup K_{j-1})|
					\ge|U_h^{(2)}|-(6tm+3t)
					\ge|U_h^{(2)}|/2.
				\]
				Since $\varphi(a)\ne0$ and $\varphi(a)\in\{0,\ldots,t-1\}$, we have $\varphi(a)\in\{1,\ldots,t-1\}$. Applying \Cref{claim:three-cluster-embedding} with $p=\varphi(a)$ to these unused vertices therefore yields $K_j$. Thus the copies $K_1,\ldots,K_\ell$ can be chosen successively with the required vertex counts.

				Set $\cK_2=\cK'\cup\{K_1,\ldots,K_\ell\}$. The added copies use at most $3t$ vertices of any cluster, so the size bound in the induction hypothesis gives, for every $i\in[2m]$,
				\[
					|U_i\cap V(\cK_2)|
					\le |U_i\cap V(\cK')|+3t
					\le3t|\supp(\varphi')|+3t
					\le3t|\supp(\varphi)|.
				\]
				It remains to verify the congruences. For every $j\in[\ell]$, the copy $K_j$ uses $t$ vertices of $U_{i_{2j-1}}^{(2)}$. For every $j\in[\ell-1]$, the copies $K_j$ and $K_{j+1}$ use $t-\varphi(a)$ and $\varphi(a)$ vertices, respectively, of $U_{i_{2j}}^{(2)}$, and hence together use $t$ vertices there. Thus, apart from the two endpoints of the walk, the added copies contribute $0$ modulo $t$ in every cluster. At the endpoints, $K_1$ contributes $\varphi(a)$ vertices in $U_{i_0}^{(2)}=U_a^{(2)}$, while $K_\ell$ contributes $t-\varphi(a)\equiv-\varphi(a)\pmod t$ vertices in $U_{i_{2\ell}}^{(2)}=U_b^{(2)}$. Since $\varphi'(a)=0$, $\varphi'(b)\equiv\varphi(a)+\varphi(b)\pmod t$, and $\varphi'(i)=\varphi(i)$ for $i\notin\{a,b\}$, it follows that
				\[
					|U_i\cap V(\cK_2)|\equiv\varphi(i)\pmod t
				\]
				for every $i\in[2m]$. Therefore $\cK_2$ satisfies both conclusions of the claim for $\varphi$. This completes the induction step and the proof.
			\end{proof}

			Let $\cK_1$ be a collection of copies of $K_{t,t}$ as guaranteed by \Cref{claim:K1}. For every $i\in[2m]$, define $\varphi_G(i)\in\{0,\ldots,t-1\}$ by
			\[
			\varphi_G(i)\equiv-|U_i\cap V(\cK_1)|
			\equiv-|U_i^{(1)}\cap V(\cK_1)|\pmod t.
			\]
			The second congruence holds because every vertex of $\cK_1$ outside $V_0$ lies in $U^{(1)}$, so $U_i\cap V(\cK_1)=U_i^{(1)}\cap V(\cK_1)$.
			To verify the first hypothesis of \Cref{claim:K2}, we have
			\begin{equation}\label{eqsumfi}
			-\sum_{i\in[2m]}\varphi_G(i)
			\equiv\sum_{i\in[2m]}|U_i\cap V(\cK_1)|
			\equiv|V(\cK_1)|-|V_0|
			\equiv0\pmod t.
			\end{equation}
			The first congruence in \eqref{eqsumfi} follows from the definition of $\varphi_G$. The second holds because $V(G)=V_0\cup U_1\cup\cdots\cup U_{2m}$ and $V_0\subseteq V(\cK_1)$. The final congruence holds because each copy in $\cK_1$ has $2t$ vertices, so $2t\mid|V(\cK_1)|$, while property~\ref{itm:Gprime-4} gives $2t\mid|V_0|$.

			If $G$ is bipartite and balanced, let $\{X,Y\}$ be its bipartition, labelled so that $X-V_0=\bigcup_{i\in[m]}U_{2i-1}$ and $Y-V_0=\bigcup_{i\in[m]}U_{2i}$. Then
			\begin{equation}\label{eq:sumf2i}
			-\sum_{i\in[m]}\varphi_G(2i)
			\equiv |(Y-V_0)\cap V(\cK_1)|
			\equiv\frac{|V(\cK_1)|-|V_0|}{2}
			\equiv0\pmod t.
			\end{equation}
			Indeed, \ref{itm:Gprime-4} gives $|X-V_0|=|Y-V_0|$, and hence $|X\cap V_0|=|Y\cap V_0|$. Every copy in $\cK_1$ contains exactly one vertex of $V_0$ and is balanced across $X,Y$, so it follows that the numbers of vertices of $V(\cK_1)-V_0$ lying in $X$ and $Y$ are equal. Since both $|V(\cK_1)|$ and $|V_0|$ are divisible by $2t$, the integer $(|V(\cK_1)|-|V_0|)/2$ is divisible by $t$, which proves the final congruence in \eqref{eq:sumf2i}.

			Equations~\eqref{eqsumfi} and~\eqref{eq:sumf2i} verify that $\varphi_G$ satisfies the two hypotheses of \Cref{claim:K2}, where the second is needed only when $G$ is bipartite. Applying the claim with $\varphi=\varphi_G$, we may therefore choose a collection $\cK_2$ of pairwise vertex-disjoint copies of $K_{t,t}$ in $G[U^{(2)}]$ such that, for every $i\in[2m]$,
			\[
			|U_i\cap V(\cK_2)|\le3t|\supp(\varphi_G)|\le6tm
			\quad\text{and}\quad
			|U_i^{(2)}\cap V(\cK_2)|\equiv\varphi_G(i)\pmod t.
			\]
			Since $\cK_1$ and $\cK_2$ use the disjoint sets $U_i^{(1)}$ and $U_i^{(2)}$, respectively, the definition of $\varphi_G$ gives
			\[
			\begin{aligned}
			|U_i\cap V(\cK_1\cup\cK_2)|
			&=|U_i^{(1)}\cap V(\cK_1)|+|U_i^{(2)}\cap V(\cK_2)|\\
			&\equiv-\varphi_G(i)+\varphi_G(i)\equiv0\pmod t.
			\end{aligned}
			\]
			Thus $|U_i\cap V(\cK_1\cup\cK_2)|$ is divisible by $t$ for every $i\in[2m]$. Since $2t\mid|U_i|$ by \ref{itm:Gprime-4}, the number $|U_i\setminus V(\cK_1\cup\cK_2)|$ of uncovered vertices is also divisible by $t$.

		\subsection{Covering the remaining vertices of $G$ with vertex-disjoint copies of $K_{t,t}$}
        \label{subsec:templatematching}

			In this subsection, our aim is to cover the remaining vertices of $G$ with vertex-disjoint copies of $K_{t,t}$. To this end, we will find a collection $\cK_3$ of copies of $K_{t,t}$ such that $V(\cK_3) \subseteq V(G') \setminus V(\cK_1 \cup \cK_2)$ and such that the sets
			$U'_i \coloneqq U_i \setminus V(\cK_1 \cup \cK_2 \cup \cK_3)$, $i \in [2m]$, have the same size. This allows us to apply the blow-up lemma to find perfect $K_{t,t}$-packings in $G'[U'_{2i-1}, U'_{2i}]$ for every $i \in [m]$, thereby covering all of the remaining vertices, as desired.

			We now implement this construction. For $i \in [2m]$, let $U_i'^{(3)}$ be an arbitrary subset of $U_i^{(3)}$ such that
			\begin{equation} \label{eqn:Ui}
				\big|U_i'^{(3)}\big| = \big|U_i^{(3)}\big| - \frac{\big|U_i \cap V(\cK_1 \cup \cK_2) \big|}{t}.
			\end{equation}
				The right-hand side of \eqref{eqn:Ui} is an integer by the choice of $\cK_2$. It is also nonnegative. Indeed, every vertex of $\cK_1$ that lies in $U_i$ belongs to $U_i^{(1)}$, while every vertex of $\cK_2$ that lies in $U_i$ belongs to $U_i^{(2)}$. Hence, by \ref{itm:U-1},
				\[
					|U_i\cap V(\cK_1\cup\cK_2)|
					\le |U_i^{(1)}|+|U_i^{(2)}|
					=2\xi|U_i|
					\le |U_i|/2
					=t|U_i^{(3)}|,
				\]
				where the last equality follows from \ref{itm:U-2}. Thus the required subset $U_i'^{(3)}$ exists. Recall that $U^{(3)} = \bigcup_{i \in [2m]}U_i^{(3)}$.
			Write $U'^{(3)} \coloneqq \bigcup_{i \in [2m]}U'^{(3)}_i$.
			In the rest of this section, let $H \coloneqq G'\left[U^{(3)}\right]$, $H' \coloneqq G'\left[U'^{(3)}\right]$ and let $d'\coloneqq d/(2t)$. We will use $d'$ as the parameter $d$ in the application of Lemma~\ref{lem:hamilton-cycle} below; this parameter need not be an integer.

			\begin{claim}
            \label{pminH'}
				There is a perfect matching in $H'$.
			\end{claim}

			\begin{proof}[ of claim]
				We claim that the following properties hold.
				\begin{enumerate}[label = \rm(\alph*)]
					\item \label{itm:HH-1}
						$H$ has maximum degree at most $d'(1 + \eta)$.
					\item \label{itm:HH-2}
						$H$ has average degree at least $d'(1 + \eta) - 32\eta t \cdot |V(H)|$.
					\item \label{itm:HH-3}
						$H$ has no $\zeta/20$-sparse cuts.
					\item \label{itm:HH-4}
						$H$ is either bipartite and balanced, or one needs to  remove at least $\frac{\gamma}{32t^2}n^2$ edges from $H$ to make it bipartite.
					\item \label{itm:HH-5}
						$|V(H) \setminus V(H')| \le 8t\xi \cdot |V(H)|$.
					\item \label{itm:HH-6}
						If $G$ is bipartite and balanced, then $H'$ is also bipartite and balanced.
					\item \label{itm:HH-7}
						$|V(H')|$ is even.
				\end{enumerate}
				Indeed, by \ref{itm:U-2} and \ref{itm:U-5}, any vertex $v$ in $H$ has degree at most $ \sum_{i \in [2m]} \left((1+\eps) \frac{d_{G'}(v, U_i)}{2t}\right) + \eps n \le d'(1+\eps)+\eps n\le d'(1+\eta)$, where the last inequality uses $\eps\ll\eta,c,1/t$, proving \ref{itm:HH-1}.
				For proving \ref{itm:HH-2}, for each $v\in V(H)$ let $I(v)=\{i\in[2m]:d_{G'}(v,U_i)\ge\eps|U_i|\}$. Since $\sum_{i\notin I(v)}d_{G'}(v,U_i)<\eps|V(G')|$, we have the following sequence of inequalities.
				\begin{align*}
					2e(H)
					= \sum_{v \in V(H)} \sum_{i \in [2m]}d_{G'}(v, U^{(3)}_i)
					& \overset{\ref{itm:U-5}}{\ge} (1 - \eps)\sum_{v \in V(H)} \sum_{i \in I(v)} d_{G'}(v, U_i) \cdot \frac{|U_i^{(3)}|}{|U_i|} \\
					& \overset{\ref{itm:U-2}}{\ge} \frac{1-\eps}{2t} \sum_{v \in V(H)} \bigl(d_{G'}(v)-\eps|V(G')|\bigr) \\
					& = \frac{1-\eps}{2t} \left(\sum_{v \in V(G')} d_{G'}(v) - \sum_{v \in V(G') \setminus V(H)} d_{G'}(v)-\eps|V(G')||V(H)|\right) \\[.5em]
					& \overset{\ref{itm:Gprime-7}}{\ge} \frac{1-\eps}{2t} \cdot \Big((d - 6\eta n) \cdot |V(G')| - d \cdot \big(|V(G')| - |V(H)|\big) - \eps|V(G')||V(H)|\Big) \\
					& = \frac{1-\eps}{2t} \cdot \Big(d \cdot |V(H)| - 6\eta n \cdot |V(G')|-\eps|V(G')||V(H)|\Big) \\
					& \overset{\ref{itm:U-2}}{\ge} d'\cdot|V(H)|-\eps d'\cdot|V(H)|-6\eta n\cdot|V(H)|-\eps|V(H)|^2 \\
					& = d'(1 + \eta) \cdot |V(H)| - (\eps d' + 6\eta n + \eta d'+\eps|V(H)|) \cdot |V(H)| \\
					& \ge d'(1 + \eta) \cdot |V(H)| - 8\eta n \cdot |V(H)| \\
					& \overset{\ref{itm:U-2},\eqref{eqn:boundV0aftermove}}{\ge} d'(1+\eta)\cdot|V(H)|-32\eta t\cdot|V(H)|^2.
				\end{align*}
				It follows that $H$ has average degree at least $d'(1 + \eta) - 32 \eta t \cdot |V(H)|$, as required for \ref{itm:HH-2}.
				Since $H=G'[U^{(3)}]$, property~\ref{itm:HH-3} follows directly from \ref{itm:U-6}, and property~\ref{itm:HH-4} follows directly from \ref{itm:U-7}.
				For proving \ref{itm:HH-5}, note that by \ref{itm:U-1}, $|U^{(1)}|+|U^{(2)}|=2\xi|V(G')|$. Hence $|V(H)\setminus V(H')|\le|U^{(1)}|+|U^{(2)}|\le2\xi n\le8t\xi|V(H)|$, where the last inequality follows from \ref{itm:U-2} and \eqref{eqn:boundV0aftermove}.
				Next, for proving \ref{itm:HH-6}, suppose that $G$ is bipartite and balanced, and denote its bipartition by $\{X, Y\}$. Recall that, by \ref{itm:Gprime-4}, $|V_0 \cap X| = |V_0 \cap Y|$ (since $|X| = |Y|$). Every copy of $K_{t,t}$ in the bipartite graph $G$ contains $t$ vertices in $X$ and $t$ vertices in $Y$. Since $V_0\subseteq V(\cK_1)$ and $V_0$ contains equally many vertices of $X$ and $Y$, the vertices of $V(\cK_1)\cap V(G')$ are also equally divided between $X$ and $Y$. All copies in $\cK_2$ lie in $G'$, and their vertices are equally divided between $X$ and $Y$ as well. Therefore
				\[
					|V(G') \cap V(\cK_1 \cup \cK_2) \cap X|
					=
					|V(G') \cap V(\cK_1 \cup \cK_2) \cap Y|,
				\]
				and thus
				\begin{align*}
					|V(H') \cap X|
					& = \sum_{i \in [m]}\left(|U_{2i-1}^{(3)}| - \frac{|U_{2i-1} \cap V(\cK_1 \cup \cK_2)|}{t}\right) \\[.5em]
					& \overset{\ref{itm:U-2}}{=} \left (\sum_{i \in [m]}\frac{|U_{2i-1}|}{2t} \right ) - \frac{|V(G') \cap (V(\cK_1 \cup \cK_2) \cap X)|}{t} \\[.5em]
					& \overset{\ref{itm:Gprime-4} }{=} \left( \sum_{i \in [m]}\frac{|U_{2i}|}{2t} \right )- \frac{|V(G') \cap (V(\cK_1 \cup \cK_2) \cap Y)|}{t} \\[.5em]
					& = \sum_{i \in [m]}\left(|U_{2i}^{(3)}| - \frac{|U_{2i} \cap V(\cK_1 \cup \cK_2)|}{t}\right)
					= |V(H') \cap Y|,
				\end{align*}
				proving that $H'$ is balanced, as required for \ref{itm:HH-6}.

				Finally, for proving \ref{itm:HH-7}, if $G$ is bipartite and balanced, then $H'$ is balanced by \ref{itm:HH-6}, so $|V(H')|$ is even. It remains to prove that $|V(H')|$ is even in the case when $G$ is $\gamma$-far-from-bipartite. In this case, summing \eqref{eqn:Ui} over $i\in[2m]$ gives $|V(H')|=2m|U_1|/(2t)-|V(\cK_2)|/t-|V(G')\cap V(\cK_1)|/t$. The first term is even because $2t$ divides $|U_1|$. Since every copy in $\cK_2$ has $2t$ vertices, $|V(\cK_2)|/t=2|\cK_2|$, so the second subtracted quantity is even. Finally, $\cK_1$ contains exactly $|V_0|$ copies and each $K_{t,t}$ in $\cK_1$ has exactly one vertex from $V_0$, so $|V(G')\cap V(\cK_1)|=(2t-1)|V_0|$ and the third subtracted quantity equals $2(2t-1)|V_0|/(2t)$, which is even because $2t$ divides $|V_0|$. Thus $|V(H')|$ is even.

					Thus, by applying Lemma~\ref{lem:hamilton-cycle} with $H$, $d'(1 + \eta)$, $c/2t$, $32 \eta t$, $8t\xi$, $\gamma/(32t^2)$, $\zeta/20$,  $V(H)\setminus V(H')$, $|V(H)|$ playing the roles of $G$, $d$, $c$, $\eta$, $\xi$, $\gamma$, $\zeta$, $W$, $n$ respectively, we find that there is a Hamilton cycle in $H'$, which contains a perfect matching in $H'$ (since $|V(H')|$ is even), thus proving the claim.
			\end{proof}
			Let $\cM'$ be a perfect matching in $H'$ (guaranteed by \Cref{pminH'}).
			For $i,j \in [2m]$, denote by $g(i,j)$ the number of edges in $\cM'[U_i'^{(3)}, U_j'^{(3)}]$.

			\begin{claim} \label{claim:Ktt-f}
				There is a collection $\cK_3$ of pairwise vertex-disjoint copies of $K_{t,t}$ in $G'\left[U^{(4)}\right]$ such that for every $i,j \in [2m]$, $\cK_3$ consists of $g(i,j)$ copies of $K_{t,t}$ in $G'[U_i,U_j]$ (and $\cK_3$ contains no other copies).
			\end{claim}

			\begin{proof}
				Let $\cK_3$ be a maximal collection of pairwise vertex-disjoint $K_{t,t}$'s in $G'\left[U^{(4)}\right]$, such that for every $i,j \in [2m]$, $\cK_3$ consists of at most $g(i,j)$ copies of $K_{t,t}$ in $G'[U_i^{(4)},U_j^{(4)}]$ (and no other copies of $K_{t,t}$).
				Suppose towards a contradiction that there exist $i,j \in [2m]$ such that the number of copies of $K_{t,t}$ in $\cK_3$ that are in $G'[U_i,U_j]$ is strictly smaller than $g(i,j)$ (implying that $g(i,j) > 0$).
				Write $U_{\ell}'' \coloneqq U_{\ell}^{(4)} \setminus V(\cK_3)$ for every $\ell \in [2m]$.

				We claim that $|U_{\ell}''| \ge |U_{\ell}|/6$ for every $\ell \in [2m]$. Indeed, note that for every $\ell \in [2m]$ we have $|U_{\ell}^{(4)} \cap V(\cK_3)| \le t \cdot \sum_{ j \in [2m]} g(\ell,j) = t \cdot |U_{\ell}'^{(3)}| \le t \cdot |U_{\ell}^{(3)}| = |U_{\ell}|/2$, where in the last equality we used \ref{itm:U-2}. Thus, by \ref{itm:U-3}, for every $\ell \in [2m]$, $|U_{\ell}''| = |U_{\ell}^{(4)}| - |U_{\ell}^{(4)} \cap V(\cK_3)|  \ge (2/3)|U_{\ell}| - (1/2)|U_{\ell}| = |U_{\ell}|/6$, as claimed.

				Every cluster $U_i$ is contained in one of the original regularity clusters, which is independent in $G'$ by \ref{itm:reg-5}; hence $i\ne j$. Moreover, $G'[U_i,U_j]$ has positive density and is $5\eps$-regular with density more than $\mu/5$ by \ref{itm:Gprime-6}. Since $|U_j''|\ge|U_j|/6\ge5\eps|U_j|$, all but at most $5\eps|U_i|$ vertices of $U_i$ have at least $(\mu/5-5\eps)|U_j''|\ge(\mu/6)|U_j''|$ neighbours in $U_j''$. Set $A:=\{v\in U_i'':d_{G'}(v,U_j'')\ge(\mu/6)|U_j''|\}$. Then $|A|\ge |U_i''| - 5\eps|U_i| \ge (1/6-5\eps)|U_i| \ge |U_i|/7$. Moreover, $m\le M$ and \eqref{eqn:boundV0aftermove} imply $|U_i|,|U_j|\ge n/(4M)$. Consequently,
				\[
					|A|\ge\frac{|U_i|}{7}\ge\frac{n}{28M}
					\quad\text{and}\quad
					d_{G'}(v,U_j'')\ge\frac\mu6|U_j''|\ge\frac\mu{36}|U_j|\ge\frac{\mu n}{144M}
				\]
				for every $v\in A$. Thus, with $\delta_*:=\min\{1/(28M),\mu/(144M)\}$, the sets $A,U_j''$ satisfy $|A|\ge\delta_*n$ and $d_{G'}(v,U_j'')\ge\delta_*n$ for every $v\in A$. Since $n_0$ is sufficiently large in terms of $\delta_*$ and $t$, \Cref{lem:unbalanced-kst} gives a copy of $K_{t,t}$ in $G'[A,U_j'']$. This contradicts the maximality of $\cK_3$ and proves the claim.
			\end{proof}

			Let $\cK_3$ be a collection of copies of $K_{t,t}$ as guaranteed by \Cref{claim:Ktt-f}.
			For $i \in [2m]$, write $U_i' \coloneqq U_i \setminus V(\cK_1 \cup \cK_2 \cup \cK_3)$.
			For $i \in [2m]$, by the choice of $U_i'^{(3)}$ as in \eqref{eqn:Ui} and the fact that $|U_i \cap V(\mathcal K_3)| = t \cdot \sum_{j \in [2m]}g(i,j) = t \cdot |U_i'^{(3)}|$, we have
			\begin{align*}
				|U_i'|
				& = |U_i \setminus V(\cK_1 \cup \cK_2 \cup \cK_3)|
				= |U_i| - \big|U_i \cap V(\cK_1 \cup \cK_2)\big| - \big|U_i \cap V(\cK_3)\big|\\
				& = |U_i| - \big|U_i \cap V(\cK_1 \cup \cK_2)\big| - t \cdot \left(\big|U_i^{(3)}\big| - \frac{\big|U_i \cap V(\cK_1 \cup \cK_2)\big|}{t}\right)
				=|U_i|-t|U_i^{(3)}|\overset{\ref{itm:U-2}}{=}\frac{|U_i|}{2}.
			\end{align*}
			Recall that the sets $U_i$ for $i \in [2m]$ have the same size, which is divisible by $2t$ by \ref{itm:Gprime-4}. Hence, the sets $U'_i$ for $i \in [2m]$ have the same size, which is divisible by $t$.
				Moreover, fix $i\in[m]$. By \ref{itm:Gprime-6}, the pair $G'[U_{2i-1},U_{2i}]$ is $5\eps$-regular, and by \ref{itm:Gprime-5} its density is more than $\mu/5$. Since $|U_{2i-1}'|=|U_{2i}'|=|U_1|/2$, Observation~\ref{obs:regular-pair-restriction}, applied with $\eps_0=5\eps$ and $\alpha=1/2$, shows that $G'[U_{2i-1}',U_{2i}']$ is $10\eps$-regular and that
				\begin{equation*}
					d(U_{2i-1}',U_{2i}')>d(U_{2i-1},U_{2i})-5\eps>\mu/5-5\eps.
				\tag{\theequation a}\label{eqn:final-superregular-pair-density}
				\end{equation*}
				For every $v\in U_{2i-1}'$, property~\ref{itm:Gprime-5} gives $d_{G'}(v,U_{2i})\ge(\mu/5)|U_{2i}|\ge\eps|U_{2i}|$. Hence \ref{itm:U-4} and \ref{itm:U-5}, together with $U_{2i}^{(5)}\subseteq U_{2i}'$, give
				\begin{equation*}
				d_{G'}(v,U_{2i}')\ge d_{G'}(v,U_{2i}^{(5)})
				\ge(1-\eps)\Bigl(\frac13-\frac1{2t}-2\xi\Bigr)\frac\mu5|U_{2i}|
				\ge\frac\mu{50}|U_{2i}'|.
				\tag{\theequation b}\label{eqn:final-superregular-left-degree}
				\end{equation*}
					For every $v\in U_{2i}'$, property~\ref{itm:Gprime-5} gives $d_{G'}(v,U_{2i-1})\ge(\mu/5)|U_{2i-1}|\ge\eps|U_{2i-1}|$. Hence \ref{itm:U-4} and \ref{itm:U-5}, together with $U_{2i-1}^{(5)}\subseteq U_{2i-1}'$, give
				\begin{equation*}
					d_{G'}(v,U_{2i-1}')\ge d_{G'}(v,U_{2i-1}^{(5)})
					\ge(1-\eps)\Bigl(\frac13-\frac1{2t}-2\xi\Bigr)\frac\mu5|U_{2i-1}|
					\ge\frac\mu{50}|U_{2i-1}'|.
				\tag{\theequation c}\label{eqn:final-superregular-right-degree}
				\end{equation*}
					Furthermore, if $A\subseteq U_{2i-1}'$ and $B\subseteq U_{2i}'$ satisfy $|A|\ge10\eps|U_{2i-1}'|$ and $|B|\ge10\eps|U_{2i}'|$, then the $10\eps$-regularity of $G'[U_{2i-1}',U_{2i}']$ and the density bound~\eqref{eqn:final-superregular-pair-density} give
				\begin{equation*}
					d(A,B)>d(U_{2i-1}',U_{2i}')-10\eps>\mu/5-15\eps>\mu/50.
				\tag{\theequation d}\label{eqn:final-superregular-subset-density}
				\end{equation*}
				Equation~\eqref{eqn:final-superregular-subset-density} verifies the density condition in the definition of super-regularity, while the minimum-degree inequalities~\eqref{eqn:final-superregular-left-degree} and~\eqref{eqn:final-superregular-right-degree} verify its two minimum-degree conditions. Thus $G'[U_{2i-1}',U_{2i}']$ is $(10\eps,\mu/50)$-super-regular.
			Therefore, using also that $|U_{2i-1}'| = |U_{2i}'| = |U_1|/2$, where this quantity is divisible by $t$, it follows from the blow-up lemma (\Cref{lem:blowup}) that $G'[U_{2i-1}', U_{2i}']$ has a perfect $K_{t,t}$-packing (since $\eps \ll \mu, 1/t$). These perfect $K_{t,t}$-packings of $G'[U_{2i-1}', U_{2i}']$ for $i \in [m]$, along with $\cK_1$, $\cK_2$ and $\cK_3$, form a perfect $K_{t,t}$-packing in $G$. This completes the proof of Lemma~\ref{lem:Ktt-packing-expander}.
		\end{proof}

\section{$F$-packings in dense regular graphs} \label{sec:proof}

	In this section, we combine \Cref{lem:balancing} and \Cref{cor:Ktt-packing-expander} to prove Theorem~\ref{mainthm:packingsubgraphs}, which asserts that for every bipartite graph $F$ and every constant $0<c<1$, there is a constant $C > 0$ such that every $d$-regular graph $G$ of order $n$, with $d \ge cn$, has an $F$-packing that covers all but at most $C$ vertices of $G$. Denote the sizes of the two vertex classes of $F$ by $a$ and $b$, and write $t \coloneqq a+b$. Then two vertex-disjoint copies of $F$, with their two vertex classes placed in opposite orientations, partition the vertices of $K_{t,t}$. Hence, Theorem~\ref{mainthm:packingsubgraphs} follows from the following result.

	\begin{theorem}
	\label{mainthm:Ktt}
	Let $0<c<1$, and let $t \ge 2$ be an integer. Then there exists a constant $C = C(t,c)$ such that every $d$-regular graph $G$ of order $n$, where $d \ge cn$, has a $K_{t,t}$-packing that covers all but at most $C$ vertices of $G$.
\end{theorem}

\begin{proof}[ of Theorem~\ref{mainthm:Ktt}]
	Let $n_0 \in \mathbb{N}$ be chosen sufficiently large depending on $c$ and $t$, that is, $1/n_0 \ll c$. If $n<n_0$, then the empty $K_{t,t}$-packing leaves $n\le n_0-1$ vertices uncovered, so this case is covered by taking $C(t,c)\ge n_0-1$. We may therefore assume that $n\ge n_0$.
		By applying Lemma~\ref{lem:balancing} to $G$, we obtain positive numbers $\eta, \gamma, \zeta$, a positive integer $r\le R\coloneqq\lceil1/\min\{c,1/t\}\rceil$, a subgraph $G' \subseteq G$ with $V(G') = V(G)$, and pairwise disjoint sets
		$Z_1, \cdots, Z_{r}, L \subseteq V(G)$ satisfying properties \ref{lem:balance:leftover}--\ref{lem:balance:perfectpackL}
		such that $$1/n_0 \ll \eta \ll \zeta \ll \gamma  \ll c, 1/t.$$
			Fix $i\in[r]$. The graph $G'[Z_i]$ has maximum degree at most $d$ because $G'\subseteq G$, and it has average degree at least $d-\eta|Z_i|$ by \ref{lem:balance:highaveragedegree}. Since this average degree is at most $|Z_i|-1$, we have $|Z_i|\ge d/(1+\eta)\ge cn/2$, so $|Z_i|$ is sufficiently large. Property~\ref{lem:balance:nosparsecut} says that $G'[Z_i]$ has no $\zeta$-sparse cuts. By \ref{lem:balance:dichotomy}, it is either bipartite and balanced, or at least $\gamma n^2\ge\gamma|Z_i|^2$ edges must be removed to make it bipartite. Finally, $d\ge cn\ge c|Z_i|$. Thus Corollary~\ref{cor:Ktt-packing-expander}, with $G'[Z_i]$ and $|Z_i|$ in place of $G$ and $n$, gives a $K_{t,t}$-packing covering all but at most $2t-1$ vertices of $Z_i$. Applying this for every $i\in[r]$, using the perfect $K_{t,t}$-packing in $G[L]$ supplied by \ref{lem:balance:perfectpackL}, and using
			\[
			|V(G) \setminus (Z_1 \cup \cdots \cup Z_{r} \cup L)| \le 64r^2 \cdot \frac{t}{c} \cdot \left(\frac{e}{8\zeta}\right)^t
			\]
			from \ref{lem:balance:leftover}, we obtain a $K_{t,t}$-packing of $G$ covering all but at most the following number of vertices, where the parameter $\zeta$ depends only on $c$ and $t$ through \Cref{lem:balancing}.
		\begin{equation*}
			r \cdot (2t-1) + 64 \cdot r^2 \cdot \frac{t}{c} \cdot \left(\frac{e}{8\zeta}\right)^t
			\le 66 \cdot r^2 \cdot \frac{t}{c} \cdot \left(\frac{e}{8\zeta}\right)^t
			\le 66R^2 \cdot \frac{t}{c} \cdot \left(\frac{e}{8\zeta}\right)^t.
		\end{equation*}
		Taking $C(t,c)$ to be the maximum of $n_0-1$ and the final displayed quantity completes the proof of Theorem~\ref{mainthm:Ktt}.
	\end{proof}

\section{Packing subdivisions in dense regular graphs} \label{sec:Kt-subdivisions-dense}

	In this section, we prove \Cref{thm:Kt-subdivision-dense}. We start by giving a detailed sketch of its proof here. Let $F$ be a graph, and let $0<c<1$. Let $G$ be a $d$-regular graph of order $n$, where $d \ge cn$. We start, as before, by partitioning the vertex set of $G$ into sets $Z_1, \ldots, Z_r$ such that the subgraphs $G[Z_i]$, $i \in [r]$, induced by these sets are expanders with few edges between them (using Lemma~\ref{lem:expander-decomposition}). For each $i\in[r]$, let $\{X_i,Y_i\}$ be a partition of $Z_i$ maximising the number of edges in $G[X_i,Y_i]$. Set $G'[Z_i]:=G[X_i,Y_i]$ if $G[Z_i]$ is close to being bipartite and $G'[Z_i]:=G[Z_i]$ otherwise, and retain every edge of $G$ joining distinct sets $Z_i$ and $Z_j$. Thus, for every $i \in [r]$, $G'[Z_i]$ is either bipartite with bipartition $\{X_i,Y_i\}$ or is far from being bipartite.

	Next, we find short paths $P_1, \ldots, P_r$ (where each $P_i$ has both of its ends in $Z_i$) whose removal `balances' the expanders that are close to being bipartite.
	More precisely, for every $i \in [r]$, if $G[Z_i]$ is close to being bipartite, then, writing $Q \coloneqq P_1 \cup \ldots \cup P_r$, we have $|X_i \setminus V(Q)| = |Y_i \setminus V(Q)|$, and $P_i$ has one end in $X_i$ and the other end in $Y_i$. The key tool for building these paths is \Cref{lem:path-forest} due to Gruslys and Letzter~\cite{gruslys2021cycle}, which produces a small linear forest $H$ that balances the expanders that are close to being bipartite, while also ensuring that each expander contains zero or two leaves of $H$. To obtain the desired paths $P_1, \ldots, P_r$, we iteratively merge pairs of components of $H$ whose leaves lie in the same expander and, if necessary, we define some paths $P_i$ as arbitrary edges within $G'[Z_i]$.

	Our next step is to find two small, vertex-disjoint $F$-subdivisions within each expander $G'[Z_i]$, denoted $F_i$ and $F_i'$, which are disjoint from $Q$ and whose union is balanced; that is, if $G[Z_i]$ is close to being bipartite, then $|X_i \cap V(F_i \cup F_i')| = |Y_i \cap V(F_i \cup F_i')|$. We construct $F_i$ and $F_i'$ greedily: we begin by selecting $|V(F)|$ vertices from $X_i$ to serve as the branch vertices of $F_i$, and $|V(F)|$ vertices from $Y_i$ to serve as the branch vertices of $F_i'$. We then iteratively connect pairs of branch vertices (corresponding to the edges of $F$) one pair at a time, to complete each subdivision.

    Finally, for each $i \in [r]$, we absorb the path $P_i$ and the remaining uncovered vertices in $Z_i$ into the subdivision $F_i$. To do this, we replace an arbitrary edge $x_i y_i$ in $F_i$ with a path that starts at $x_i$, connects to one end of $P_i$, traverses $P_i$, and then continues through all the remaining uncovered vertices in $Z_i$ before returning to $y_i$.

	In the last three steps, we rely on the fact that the expanders $G'[Z_i]$ are \emph{robustly connected} via short paths; that is, any two vertices can be joined by a short path that avoids any given small set of forbidden vertices (see \Cref{claim:findshortpathinexpander}). In the final step, we also use the \emph{robust Hamiltonicity} property of the expanders $G'[Z_i]$, meaning that for any small set of forbidden vertices $W$ (which is balancing if $G'[Z_i]$ is bipartite), and any two vertices $x, y \in Z_i \setminus W$ (lying in different parts of $G'[Z_i]$ if it is bipartite), there exists a Hamilton path in $G'[Z_i] \setminus W$ with ends $x$ and $y$.

   The proof of \Cref{thm:Kt-subdivision-dense} is given in \Cref{subsec:formalproofthm1.7}, after establishing two preparatory lemmas in the following two subsections.

	\subsection{Expanders are robustly connected via short paths}

	In the following lemma, we show that our expanders are `robustly connected' via short paths, a property we will use several times subsequently.

	\begin{lemma}
		\label{claim:findshortpathinexpander}
		Let $n$ be a positive integer, and let $\beta, \xi, \zeta, \delta \in (0,1)$ satisfy $1/n \ll \xi, \beta \ll \zeta \ll \delta$. Let $H$ be a graph on at most $n$ vertices with minimum degree at least $\delta n$ and no $\zeta$-sparse cuts. If $H$ is $\beta$-almost-bipartite, let $H'$ be a spanning bipartite subgraph of $H$ with the maximum possible number of edges; otherwise, set $H':=H$. Then, $H'$ has no $\zeta/2$-sparse cuts. Moreover, if $W \subseteq V(H)$ is a subset of size at most $\xi n$, then $H'' \coloneqq H' \setminus W$ satisfies the following properties.

		\begin{enumerate}[label = \rm(\roman*)]
			\item \label{itm:S-Z-cuts}
				$H''$ has no $\zeta/2$-sparse cuts.
			\item \label{itm:S-Z-path}
				For every pair $p, p' \in V(H'')$, there exists a $(p, p')$-path in $H''$ of length at most $15/\delta$.
		\end{enumerate}
	\end{lemma}

	\begin{proof}
		If $H$ is $\beta$-almost-bipartite, then the maximality of $e(H')$ implies that $H'$ arises from a maximum cut of $H$. Thus every vertex has at least half of its $H$-neighbours across this cut, and hence $\delta(H')\ge \delta n/2$. If $H$ is not $\beta$-almost-bipartite, then $H'=H$, so the same lower bound is immediate. Consequently,
		\begin{equation}\label{eqn:Hprime-min-degree}
			\delta(H')\ge \delta n/2
			\qquad\text{and}\qquad
			\delta(H'')\ge \delta n/2-|W|\ge \delta n/4.
		\end{equation}


		For proving \ref{itm:S-Z-cuts}, consider an arbitrary partition $\{A,B\}$ of $V(H'')$ with $|A|\le |B|$. If $|A|\le \delta n/8$, then \eqref{eqn:Hprime-min-degree} gives
		\begin{equation*}
			e_{H''}(A,B)
			\ge \sum_{v\in A}\bigl(d_{H''}(v)-(|A|-1)\bigr)
			\ge \frac{\delta n}{8}|A|
			\ge \frac{\zeta}{2}|A||B|,
		\end{equation*}
		using $\zeta \ll \delta$.
		We may therefore assume that $|A|,|B|\ge\delta n/8$. 
		If $H$ is $\beta$-almost-bipartite, then there is a bipartite spanning subgraph of $H$ obtained by deleting at most $\beta n^2$ edges, and hence the maximality of $H'$ gives $e(H)-e(H')\le\beta n^2$; the same inequality is trivial when $H'=H$. 
		Since $A,B,W$ partition $V(H')$ and $e(H) - e(H') \le \beta n^2$, we have
		\begin{align*}
			e_{H''}(A,B)
			&=e_{H'}(A,B)\\
			&=e_{H'}(A,B\cup W)-e_{H'}(A,W)\\
			&\ge e_H(A,B\cup W)-\beta n^2-\xi n^2\\
			&\ge \zeta|A||B|-\beta n^2-\xi n^2\\
			&\ge \frac{\zeta}{2}|A||B|,
		\end{align*}
		where the last inequality uses $\xi,\beta\ll\delta,\zeta$. This proves \ref{itm:S-Z-cuts}. Taking $W=\emptyset$ in the same argument also proves the preceding assertion that $H'$ itself has no $\zeta/2$-sparse cuts.

		It remains to prove \ref{itm:S-Z-path}. Let $p,p'\in V(H'')$. By \ref{itm:S-Z-cuts}, $H''$ is connected, so it contains a $(p,p')$-path. Let $P=u_0\cdots u_t$ be a shortest such path, where $u_0=p$ and $u_t=p'$. Suppose for a contradiction that $t>15/\delta$. Set
		\begin{equation*}
			U\coloneqq\{u_0,u_3,\ldots,u_{3q}\},
			\qquad q\coloneqq\left\lceil\frac4\delta\right\rceil.
		\end{equation*}
		This is well defined because $t>15/\delta\ge3(4/\delta+1)\ge3q$. By \eqref{eqn:Hprime-min-degree}, the sum of the degrees in $H''$ of the vertices in $U$ is at least $(q+1)\delta n/4>n$. Hence some vertex $v\in V(H'')$ has at least two neighbours in $U$, say $u_{3i}$ and $u_{3j}$ with $i<j$. The walk
		\begin{equation*}
			u_0u_1\cdots u_{3i}vu_{3j}\cdots u_t
		\end{equation*}
		is shorter than $P$ and contains a $(p,p')$-path, contradicting the choice of $P$. This proves \ref{itm:S-Z-path} and completes the proof.
	\end{proof}

	\subsection{Balancing the expanders}

		As further preparation for the proof of \Cref{thm:Kt-subdivision-dense}, we state the following corollary of \Cref{lem:path-forest}, which will allow us to balance the expanders using a suitable collection of paths. This is somewhat similar to \Cref{lem:balancing}, where we also balance the expanders, but using a collection of vertex-disjoint $K_{t,t}$'s instead of paths.

		\begin{lemma} \label{cor:path-forest}
			Let $\eta,\beta,\xi,\gamma,\zeta,\delta\in(0,1)$, let $0<c<1$, and let $n\in\bN$ satisfy $1/n \ll \eta \ll \beta \ll \xi \ll \gamma \ll \zeta \ll \delta \ll c$.
			Let $G$ be a $d$-regular graph on $n$ vertices, where $d \ge cn$.
			Suppose that $\{Z_1, \ldots, Z_r\}$ is a partition of $V(G)$ satisfying properties \ref{itm:expander-decomp-1}--\ref{itm:expander-decomp-4} in Lemma~\ref{lem:expander-decomposition}, where $r \le \ceil{1/c}$. For $i \in [r]$ such that $G[Z_i]$ is $\beta$-almost-bipartite, let $\{X_i, Y_i\}$ be a partition of $Z_i$ maximising $e_G(X_i, Y_i)$. Then there are vertex-disjoint paths $P_1, \ldots, P_r$ in $G$ satisfying the following properties, where $Q \coloneqq \bigcup_{i \in [r]}P_i$.
			\begin{enumerate}[label = \rm(P'\arabic*)]
				\item \label{itm:Hforest-1-var}
					For every $i \in [r]$, we have $2 \le |V(P_i)| \le  \xi n$.
				\item \label{itm:Hforest-3-var}
					For every $i \in [r]$, both leaves of $P_i$ are in $Z_i$. Moreover, for each $i \in [r]$ such that $G[Z_i]$ is $\beta$-almost-bipartite, one of the leaves of $P_i$ is in $X_i$ and the other leaf is in $Y_i$.
				\item \label{itm:Hforest-4-var}
					For each $i \in [r]$ such that $G[Z_i]$ is $\beta$-almost-bipartite, $|X_i \setminus V(Q)| = |Y_i \setminus V(Q)|$.
			\end{enumerate}
		\end{lemma}

		\begin{proof}
			Apply \Cref{lem:path-forest} with $\xi/2$ in place of $\xi$, and let $H$ be the resulting linear forest, satisfying properties \ref{itm:Hforest-1}--\ref{itm:Hforest-4}. Let $G'$ be the subgraph of $G$ obtained by removing all edges with both endpoints in $X_i$ or both endpoints in $Y_i$, for each $i\in[r]$ such that $G[Z_i]$ is $\beta$-almost-bipartite.

			We construct subgraphs $H_0,H_1,\cdots,H_r$ iteratively, starting with $H_0\coloneqq H$, so that
			\begin{equation}\label{eqn:Hi-size}
				|V(H_i)|\le \frac{\xi n}{2}+\frac{15i}{\delta}
			\qquad\text{for every }i\in\{0,1,\ldots,r\}.
			\end{equation}
			Suppose that $H_{i-1}$ has already been defined. If it has two distinct components $P$ and $P'$ with leaves $p\in V(P)\cap Z_i$ and $p'\in V(P')\cap Z_i$, let $L_i$ be a shortest $(p,p')$-path in
			\begin{equation*}
				G'[Z_i]\setminus\bigl(V(H_{i-1})\setminus\{p,p'\}\bigr).
			\end{equation*}
			Applying \ref{itm:S-Z-path} of \Cref{claim:findshortpathinexpander} with $G[Z_i]$, $G'[Z_i]$, $V(H_{i-1})\setminus\{p,p'\}$, and $\xi$ playing the roles of $H$, $H'$, $W$, and $\xi$, respectively, shows that $L_i$ has length at most $15/\delta$. Indeed, \eqref{eqn:Hi-size} and $r\le\lceil1/c\rceil$ imply that $|V(H_{i-1})|\le\xi n$. We then set $H_i\coloneqq H_{i-1}\cup L_i$. If no such two components exist, we set $H_i\coloneqq H_{i-1}$. In either case, \eqref{eqn:Hi-size} holds.

			Whenever $L_i$ is defined and $G[Z_i]$ is $\beta$-almost-bipartite, property \ref{itm:Hforest-3} implies that one of $p,p'$ lies in $X_i$ and the other lies in $Y_i$. Since $L_i$ is a path in the bipartite graph $G'[Z_i]$, we have
			\begin{equation}\label{eqn:balancedpath}
				|V(L_i)\cap X_i|=|V(L_i)\cap Y_i|.
			\end{equation}

			After constructing $H_1,\cdots,H_r$, observe that at step $j$ the number of leaves in a fixed set $Z_i$ changes only when $j=i$, and then it either remains two or decreases from two to zero. If $H_i$ has two leaves in $Z_i$, those leaves belong to the same component of $H_i$; as this component already has both of its leaves in $Z_i$, it remains a component of $H_r$. Denote it by $P_i$. If no such component exists, we instead choose $P_i$ to be an edge of $G'[Z_i]$ with both endpoints outside $V(H_r)$. We now verify that such an edge exists. If $G[Z_i]$ is $\beta$-almost-bipartite, then $\{X_i,Y_i\}$ maximises $e_G(X_i,Y_i)$, so every vertex has at least half of its $G[Z_i]$-neighbours in the other part; hence $\delta(G'[Z_i])\ge\delta n/2$. Otherwise, $G'[Z_i]=G[Z_i]$, and \ref{itm:expander-decomp-2} gives $\delta(G'[Z_i])\ge\delta n$. Thus, in either case, $\delta(G'[Z_i])\ge\delta n/2$. On the other hand,
			\begin{equation}\label{eqn:Hforest-size}
				|V(H_r)|\le\frac{\xi n}{2}+\frac{15r}{\delta}\le\xi n\ll\delta n.
			\end{equation}
			Since $|Z_i|\ge\delta n+1>|V(H_r)|$, at least one vertex remains in $Z_i\setminus V(H_r)$. Every vertex $v\in Z_i\setminus V(H_r)$ has at least $\delta n/2-|V(H_r)|>0$ neighbours in $G'[Z_i]\setminus V(H_r)$. Hence this graph contains an edge, as required.

			The paths $P_1,\cdots,P_r$ are pairwise vertex-disjoint. We verify properties \ref{itm:Hforest-1-var}--\ref{itm:Hforest-4-var}, with the original value of $\xi$ in \ref{itm:Hforest-1-var}. By \eqref{eqn:Hforest-size},
			\begin{equation*}
				2\le |V(P_i)|\le\max\{2,|V(H_r)|\}\le\xi n,
			\end{equation*}
			which proves \ref{itm:Hforest-1-var}. Property \ref{itm:Hforest-3-var} follows from \ref{itm:Hforest-3} when $P_i$ is a component of $H_r$, and from the fact that the auxiliary edge $P_i$ lies in the bipartite graph $G'[Z_i]$ otherwise.

			Finally, let $Q=P_1\cup\cdots\cup P_r$ and fix $i\in[r]$ such that $G[Z_i]$ is $\beta$-almost-bipartite. If $L_i$ is defined, denote its endpoints by $p,p'$. The vertices of $Q$ lying in $Z_i$ are described by
			\begin{equation*}
				V(Q)\cap Z_i=
				\begin{cases}
					V(H)\cap Z_i,&\text{if }P_i\subseteq H_r,\\
					(V(H)\cap Z_i)\,\cupdot\,V(P_i),&\text{if }P_i\nsubseteq H_r\text{ and }L_i\text{ is not defined},\\
					(V(H)\cap Z_i)\,\cupdot\,V(P_i)\,\cupdot\,\bigl(V(L_i)\setminus\{p,p'\}\bigr),&\text{if }P_i\nsubseteq H_r\text{ and }L_i\text{ is defined}.
				\end{cases}
			\end{equation*}
			Here $A\cupdot B$ denotes a disjoint union. Property \ref{itm:Hforest-4-var} now follows from \ref{itm:Hforest-4}, \eqref{eqn:balancedpath}, and the fact that an auxiliary edge $P_i$ has one endpoint in each of $X_i,Y_i$. This completes the proof.
		\end{proof}

	\subsection{Proof of \Cref{thm:Kt-subdivision-dense}}
    \label{subsec:formalproofthm1.7}

		We are now ready to prove \Cref{thm:Kt-subdivision-dense}.

		\begin{proof}[ of \Cref{thm:Kt-subdivision-dense}]
			Let $F$ be a graph with at least one edge. Choose $n_0=n_0(F,c)$ sufficiently large, and let $G$ be a $d$-regular graph of order $n\ge n_0$ with $d\ge cn$. We show that $G$ has a perfect $\TF$-packing.

			Apply \Cref{lem:expander-decomposition} to obtain positive numbers $\eta_0,\beta,\gamma_0,\zeta_0,\delta_0$ satisfying
			\begin{equation}\label{eqn:subdivision-hierarchy}
				1/n_0\ll \eta_0\ll\beta\ll\gamma_0\ll\zeta_0\ll\delta_0\ll c,
			\end{equation}
			and a partition $\{Z_1,\cdots,Z_r\}$ of $V(G)$ satisfying \ref{itm:expander-decomp-1}--\ref{itm:expander-decomp-4}, where $r\le\lceil1/c\rceil$. Choose $\xi\in(0,1)$ so that
			\begin{equation}\label{eqn:subdivision-xi}
				\beta\ll\xi
			\qquad\text{and}\qquad
			\frac{2\sqrt\xi}{\delta_0}\ll\gamma_0,\zeta_0,c.
			\end{equation}
			For each $i\in[r]$, let $\{X_i,Y_i\}$ be a partition of $Z_i$ maximizing $e_G(X_i,Y_i)$. Let $G'$ be the subgraph of $G$ obtained by removing all edges with both endpoints in $X_i$ or both endpoints in $Y_i$, for each $i$ such that $G[Z_i]$ is $\beta$-almost-bipartite. By \Cref{cor:path-forest}, there are vertex-disjoint paths $P_1,\cdots,P_r$ satisfying \ref{itm:Hforest-1-var}--\ref{itm:Hforest-4-var}; set $Q\coloneqq P_1\cup\cdots\cup P_r$.

\subsubsection{Constructing a balanced pair of $F$-subdivisions in each expander}

			For each $i\in[r]$, we construct two vertex-disjoint subdivisions of $F$ in $G'[Z_i]\setminus V(Q)$.

			\begin{claim}\label{claim:findtwosubdivisionsbalanced}
				For every $i\in[r]$, there exist two vertex-disjoint $F$-subdivisions $F_i,F_i'$ in $G'[Z_i]\setminus V(Q)$ such that
				\begin{equation}\label{eqn:Fi-size}
					|V(F_i)|,|V(F_i')|\le |V(F)|+\frac{15e(F)}{\delta_0}.
				\end{equation}
				Moreover, if $G[Z_i]$ is $\beta$-almost-bipartite, then
				\begin{equation}\label{eqn:balanced-subdivision}
					|X_i\cap V(F_i\cup F_i')|=|Y_i\cap V(F_i\cup F_i')|.
				\end{equation}
			\end{claim}

			\begin{proof}
				Fix $i\in[r]$. Property~\ref{itm:expander-decomp-2} gives $|Z_i|\ge\delta_0n+1$, while \ref{itm:Hforest-1-var} gives $|V(Q)|\le r\xi n$. If $G[Z_i]$ is $\beta$-almost-bipartite, maximality of the cut $\{X_i,Y_i\}$ gives $|X_i|,|Y_i|\ge\delta_0n/2$; otherwise no restriction on the locations of the chosen vertices is needed. Thus, since $F$ is fixed and $n_0$ is sufficiently large, we can choose disjoint sets $S,S'\subseteq Z_i\setminus V(Q)$ of size $|V(F)|$, with $S\subseteq X_i$ and $S'\subseteq Y_i$ when $G[Z_i]$ is $\beta$-almost-bipartite. Let $\phi:S\to V(F)$ and $\phi':S'\to V(F)$ be arbitrary bijections. Let $T$ be the set of unordered pairs $\{s,s'\}\subseteq S$ for which $\phi(s)\phi(s')\in E(F)$. Let $T'$ be the set of unordered pairs $\{s,s'\}\subseteq S'$ for which $\phi'(s)\phi'(s')\in E(F)$.

				Let $\mathcal P$ be a maximal collection of paths in $G'[Z_i]\setminus V(Q)$ that are pairwise internally vertex-disjoint, whose internal vertices avoid $S\cup S'$, and such that each path joins the two vertices of a distinct pair in $T\cup T'$ and has length at most $15/\delta_0$. We claim that $|\mathcal P|=2e(F)$. Otherwise, choose an uncovered pair $\{s,s'\}\in T\cup T'$ and set
				\begin{equation*}
					W\coloneqq V(\mathcal P)\cup V(Q)\cup S\cup S'.
				\end{equation*}
				Since $F$ is fixed and $n_0$ is sufficiently large,
				\begin{equation*}
					|W|\le \frac{32e(F)}{\delta_0}+\xi n+2|V(F)|\le2\xi n.
				\end{equation*}
				Applying \ref{itm:S-Z-path} of \Cref{claim:findshortpathinexpander}, with $2\xi$ in the role of its parameter $\xi$, gives an $(s,s')$-path in $G'[Z_i]\setminus(W\setminus\{s,s'\})$ of length at most $15/\delta_0$. This contradicts the maximality of $\mathcal P$.

				Let $F_i$ be the union of the paths in $\mathcal P$ with endpoints in $S$, together with every vertex of $S$. Let $F_i'$ be the union of the paths in $\mathcal P$ with endpoints in $S'$, together with every vertex of $S'$. Thus isolated vertices of $F$, if any, are retained as branch vertices. The bijections $\phi,\phi'$ show that $F_i,F_i'$ are $F$-subdivisions. They are vertex-disjoint, and every path used in either subdivision has length at most $15/\delta_0$, so \eqref{eqn:Fi-size} holds.

				Suppose now that $G[Z_i]$ is $\beta$-almost-bipartite. Every path in $\mathcal P$ with both endpoints in $S\subseteq X_i$ has one more internal vertex in $Y_i$ than in $X_i$. Therefore
				\begin{align*}
					|V(F_i)\cap X_i|-|V(F_i)\cap Y_i|
				&=\sum_{\substack{P\in\mathcal P\\V(P)\cap S\ne\emptyset}}
				\bigl(|V(P^\circ)\cap X_i|-|V(P^\circ)\cap Y_i|\bigr)+|S|\\
				&=-e(F)+|V(F)|.
				\end{align*}
				Every path in $\mathcal P$ with both endpoints in $S'\subseteq Y_i$ has one more internal vertex in $X_i$ than in $Y_i$. Therefore, the same count for $F_i'$ gives
				\begin{equation*}
					|V(F_i')\cap Y_i|-|V(F_i')\cap X_i|=-e(F)+|V(F)|.
				\end{equation*}
				Subtracting the second identity from the first proves \eqref{eqn:balanced-subdivision}.
			\end{proof}

\subsubsection{Absorbing the paths $P_1,\ldots,P_r$ and all uncovered vertices}

			For every $i\in[r]$, let $F_i,F_i'$ be the subdivisions given by \Cref{claim:findtwosubdivisionsbalanced}. Denote the endpoints of $P_i$ by $u_i,v_i$, with $u_i\in X_i$ and $v_i\in Y_i$ when $G[Z_i]$ is $\beta$-almost-bipartite. Choose an edge $x_iy_i$ of $F_i$, labelled so that $x_i\in X_i$ and $y_i\in Y_i$ when $G[Z_i]$ is $\beta$-almost-bipartite.

			Define
			\[
				W_i^0\coloneqq\bigl(V(Q)\cup V(F_i)\cup V(F_i')\bigr)\setminus\{y_i,u_i\}.
			\]
			Then $y_i,u_i\notin W_i^0$, and
			\begin{equation}\label{QfiF'i}
				|W_i^0|
				\le |V(Q)|+|V(F_i)|+|V(F_i')|
				\le r\xi n+2\left(|V(F)|+\frac{15e(F)}{\delta_0}\right)
				\le\sqrt\xi\,n,
			\end{equation}
			where we used $r\le\lceil1/c\rceil$ and the choice of $n_0$. Apply \ref{itm:S-Z-path} of \Cref{claim:findshortpathinexpander} with $H=G[Z_i]$, $H'=G'[Z_i]$, $W=W_i^0$, and with $\sqrt\xi,\beta,\zeta_0,\delta_0$ playing the roles of $\xi,\beta,\zeta,\delta$, respectively. It gives a $(y_i,u_i)$-path in $G'[Z_i]\setminus W_i^0$ of length at most $15/\delta_0$. Let $Q_i$ be a shortest such path. By the definition of $W_i^0$, the interior of $Q_i$ avoids $V(Q)\cup V(F_i)\cup V(F_i')$.

			Set $H\coloneqq G[Z_i]$ and $H'\coloneqq G'[Z_i]$. By \Cref{claim:findshortpathinexpander}, $H'$ has no $\zeta_0/2$-sparse cuts. Define
			\begin{equation}\label{eqn:Wi-subdivision}
				W_i\coloneqq
				\Bigl((V(Q)\cap Z_i)\cup V(F_i)\cup V(F_i')\cup V(Q_i)\Bigr)
				\setminus\{x_i,v_i\}.
			\end{equation}
			The set $W_i$ contains all vertices of $Q$ that lie in $Z_i$, not only the vertices of $P_i$, because a path $P_j$ with $j\ne i$ may also contain vertices of $Z_i$. By \ref{itm:Hforest-1-var}, 
			\begin{equation*}
				|V(Q)\cap Z_i|\le |V(Q)| \le \sum_{j\in[r]}|V(P_j)|\le r\xi n.
			\end{equation*}
			By \eqref{eqn:Wi-subdivision}, \eqref{QfiF'i}, $|V(Q_i)|\le15/\delta_0+1$, and $|Z_i| \ge \delta_0 n$, we have
			\begin{equation}\label{eqn:Wi-size}
				|W_i|\le2\sqrt\xi\,n\le\frac{2\sqrt\xi}{\delta_0}|Z_i|.
			\end{equation}
			The graph $H'$ has maximum degree at most $d$. Properties \ref{itm:expander-decomp-1} and \ref{itm:expander-decomp-2}, together with $\eta_0\ll\beta$, give
			\begin{align*}
				2e(H')
			&\ge d|Z_i|-\eta_0n^2-2\beta n^2\\
			&\ge d|Z_i|-\frac{3\beta}{\delta_0^2}|Z_i|^2.
			\end{align*}
			Thus $H'$ has average degree at least $d-(3\beta/\delta_0^2)|Z_i|$. Moreover, $d\ge cn\ge c|Z_i|$, and $H'$ is either bipartite or $\gamma_0$-far-from-bipartite.

			If $H'$ is bipartite, then $\{X_i,Y_i\}$ is its bipartition, and $Q_i$ has endpoints $y_i\in Y_i$ and $u_i\in X_i$. Hence
			\begin{equation}\label{eqn:Qi-balanced}
				|V(Q_i)\cap X_i|=|V(Q_i)\cap Y_i|.
			\end{equation}
			Still assuming that $H'$ is bipartite, set $R_i\coloneqq (V(Q)\cap Z_i)\cup V(F_i) \cup V(F_i')\cup V(Q_i) = W_i \cup \{x_i,v_i\}$.
			The sets $V(Q)\cap Z_i$ and $V(F_i) \cup V(F_i')$ are disjoint, and $Q_i$ meets their union exactly in $u_i\in X_i\cap V(Q)$ and $y_i\in Y_i\cap V(F_i)$. Thus \eqref{eqn:balanced-subdivision} and \eqref{eqn:Qi-balanced} show that passing from $V(Q)\cap Z_i$ to $R_i$ adds the same number of vertices of $X_i$ and $Y_i$. Property \ref{itm:Hforest-4-var} therefore gives $|X_i\setminus R_i|=|Y_i\setminus R_i|$. By \eqref{eqn:Wi-subdivision}, $W_i=R_i\setminus\{x_i,v_i\}$, where $x_i\in X_i$ and $v_i\in Y_i$. Consequently,
			\begin{equation*}
				|X_i\setminus W_i|=|X_i\setminus R_i|+1
				=|Y_i\setminus R_i|+1=|Y_i\setminus W_i|.
			\end{equation*}

			Equation~\eqref{eqn:Wi-size} gives $|W_i|\le(2\sqrt\xi/\delta_0)|Z_i|$. Thus the hypothesis $|W|\le\xi n$ in Lemma~\ref{lem:hamilton-cycle} is satisfied, with $W_i$, $2\sqrt\xi/\delta_0$, and $|Z_i|$ playing the roles of $W$, $\xi$, and $n$, respectively. We may therefore apply Lemma~\ref{lem:hamilton-cycle}, with
			\begin{equation*}
				H',\ d,\ c,\ \frac{3\beta}{\delta_0^2},\ \frac{2\sqrt\xi}{\delta_0},\ \gamma_0,\ \frac{\zeta_0}{2},\ W_i,\ |Z_i|
			\end{equation*}
			playing the roles of $G,d,c,\eta,\xi,\gamma,\zeta,W,n$, respectively. The required hierarchy $1/|Z_i|\ll3\beta/\delta_0^2\ll2\sqrt\xi/\delta_0\ll\gamma_0,\zeta_0/2\ll c$ follows from \eqref{eqn:subdivision-hierarchy}--\eqref{eqn:subdivision-xi}. We obtain a Hamilton path $Q_i'$ in $H'\setminus W_i$ with endpoints $x_i$ and $v_i$.

			Replace the edge $x_iy_i$ of $F_i$ by the path
			\begin{equation*}
				x_iQ_i'v_iP_iu_iQ_iy_i
			\end{equation*}
			to obtain an $F$-subdivision $F_i''$. By construction,
			\begin{equation*}
				V(F_i'')\cup V(F_i')=(Z_i\setminus V(Q))\cup V(P_i).
			\end{equation*}
			Hence $\{F_i',F_i'':i\in[r]\}$ is a collection of pairwise vertex-disjoint $F$-subdivisions covering $V(G)$. This completes the proof.
		\end{proof}

	\section{Concluding remarks}

		Recall that our main result (\Cref{mainthm:packingsubgraphs}) states that for every bipartite graph $H$, every dense regular graph $G$ contains an $H$-packing covering all but $O(1)$ vertices. As noted in the introduction, this was already known for unbalanced bipartite graphs $H$ by K\"uhn and Osthus~\cite{kuhn2005packings} (see \Cref{KOdifferentvertexclasses}). In fact, they proved a slightly stronger statement in this case: the host graph $G$ need only be \emph{almost regular} rather than regular. More precisely, for every bipartite $H$ and every $0<c<1$, there exist $\eps>0$ and $C$ such that any $n$-vertex graph $G$ in which every vertex has degree between $(c-\eps)n$ and $(c+\eps)n$ contains an $H$-packing covering all but at most $C$ vertices. This fails, however, for balanced bipartite $H$ (for example, if $G$ is a slightly unbalanced bipartite graph), so our result cannot, in general, be extended to the almost regular setting.

		Keevash proved an analogous result for hypergraphs (see Theorem~5.14 in~\cite{keevash2011hypergraph}).\footnote{Keevash’s theorem is stated and proved for $3$-uniform hypergraphs, but the analogous statement for higher uniformities can be obtained by the same methods.} Specifically, he showed that for every $3$-partite $3$-graph $H$ whose parts are not all of equal size, and for every $0 < c_1, c_2 < 1$, there exist $\eps>0$ and $C>0$ such that if $G$ is a $3$-graph on $n$ vertices in which every vertex has degree between $(1 - \eps) c_1n^2$ and $(1+ \eps)c_1 n^2$, and every pair of vertices has degree more than $c_2 n$, then $G$ contains an $H$-packing covering all but at most $C$ vertices. Here, the \emph{degree of a pair} $\{u,v\}$ of vertices in $G$ is the number of vertices $w \in V(G) \setminus \{u,v\}$ such that $\{u,v,w\} \in E(G)$.

        It is plausible that, analogous to \Cref{mainthm:packingsubgraphs}, the conclusion holds for \emph{all} tripartite $3$-graphs $H$ when $G$ is assumed to be regular.

\begin{qn}
Is it true that for every tripartite $3$-graph $H$ and every $0 < c < 1$, there exist $\eps>0$ and $C>0$ such that the following holds? If $G$ is a $3$-graph on $n$ vertices in which all vertices have the same degree, which is at least $c n^2$, and every pair of vertices has degree at least $c n$, then $G$ contains an $H$-packing covering all but at most $C$ vertices.
\end{qn}

        The problem of partitioning the vertex set of edge-coloured complete graphs into a small number of monochromatic subgraphs has a very rich history; see \cite{gyarfas2016vertex} for a recent survey. An early example of a problem of this kind is Lehel's conjecture.
		An interesting problem in this direction, suggested by Matija Buci\'c, is to consider a variant of \Cref{mainthm:packingsubgraphs} where the edges of $G$ are coloured and we seek packings using monochromatic $H$-copies (where different copies of $H$ may receive different colours).

		\begin{qn}
			Is it true that for every bipartite graph $H$, integer $r \ge 2$, and $0 < c < 1$, there exists $C>0$ such that the following holds? If $G$ is an $n$-vertex $d$-regular graph with $d \ge c n$, whose edges are coloured with $r$ colours, then $G$ contains a collection of vertex-disjoint monochromatic copies of $H$ covering all but at most $C$ vertices.
		\end{qn}
The problem of finding a large monochromatic $H$-packing in a graph with a given minimum degree was studied by Balogh, Freschi, and Treglown~\cite{balogh2026ramsey}.

\paragraph{Acknowledgements.} We thank the anonymous referees for their careful reading of the manuscript and their many helpful suggestions, which significantly improved the exposition of the paper.

\paragraph{AI declaration.} 
We used AI tools solely to check the final versions of the paper for typos and mathematical errors. We reviewed all its suggestions and take responsibility for the final text.

	\bibliographystyle{abbrv}
	\bibliography{references}

	\appendix

	\section{Perfect fractional matching implies perfect $2$-matching}
    \label{fracmatchingimplies2matching}

		\begin{proof}[ of \Cref{fact:fractional-matching}]
				Suppose that $G$ is a graph with a perfect fractional matching. Let $w : E(G) \to [0,1]$ be a perfect fractional matching in $G$ that minimises the number of edges whose weight is neither $0$ nor $1$, and denote by $G'$ the subgraph of $G$ whose edges have non-zero weight in $w$. We claim that $G'$ is a $2$-matching; that is, every connected component of $G'$ is either a single edge or an odd cycle. Notice that $G'$ is spanning, because $w$ is a perfect fractional matching. Thus, if $G'$ is a $2$-matching, this would show that $G'$ is a perfect $2$-matching in $G$ as required.

			Therefore, it remains to show that $G'$ is a $2$-matching. Suppose for a contradiction that it is not, and let $F$ be a component of $G'$ which is neither an edge nor an odd cycle.

			First, notice that $F$ has minimum degree at least $2$. Indeed, otherwise, let $v$ be a vertex with degree $1$ in $F$, let $u$ be its unique neighbour, and let $x$ be a neighbour of $u$ which is not $v$. The unique positive-weight edge incident to $v$ is $vu$, so the fractional-matching equality at $v$ gives $w(vu)=1$. On the other hand, $w(ux)>0$, so the equality at $u$ gives $w(vu)<1$, a contradiction.

			Next, we claim that $F$ does not have even cycles. Indeed, suppose that $C = (v_1 \ldots v_{2s})$ is an even cycle in $F$. For $x > 0$, let $w_x : E(G) \to \mathbb R$ be the weighting of the edges of $G$ defined as follows (with the addition of indices taken modulo $2s$).
			\begin{equation*}
					w_x(e) \coloneqq \left\{
					\begin{array}{ll}
						w(e) & e \notin E(C) \\
						w(e) + x & e = v_{2i-1}v_{2i} \text{ for $i \in [s]$}\\
						w(e) - x & e = v_{2i}v_{2i+1} \text{ for $i \in [s]$}.
					\end{array}
				\right.
			\end{equation*}
			Note that the edges $v_iv_{i+1}$, with $i \in [2s]$, have weight strictly between $0$ and $1$, and so there exists $x > 0$ such that all weights in $w_x$ are in $[0,1]$. Let $x'$ be the maximum $x$ with this property. Then at least one of the edges in $C$ has weight either $0$ or $1$ in $w_{x'}$, showing that $w_{x'}$ is a perfect fractional matching in $G$ with fewer edges whose weight is neither $0$ nor $1$ compared to $w$, a contradiction to the choice of $w$. This shows that $F$ does not have even cycles, as desired. Moreover, this shows that $F$ cannot be a cycle since we assumed that $F$ is not an odd cycle.

			Since $F$ has minimum degree at least $2$ and is not a cycle, it contains two distinct cycles $C_1$ and $C_2$. Because $F$ has no even cycles, these cycles are odd and share at most one vertex (otherwise their union contains an even cycle). Let $P$ be a shortest path with one end in $C_1$ and the other end in $C_2$. (Note that $P$ is a singleton if the cycles share a vertex.) Write $C_1 = (v_1 \ldots v_{2s+1})$, $C_2 = (u_1 \ldots u_{2t+1})$, and suppose that $P$ is a path from $v_1$ to $u_1$. Now, for $x > 0$, let $w_x : E(G) \to \mathbb R$ be the weighting of the edges of $G$ obtained from $w$ by making the following modifications.
			\begin{itemize}
				\item
					Decrease the weight of edges $v_{2i-1}v_{2i}$, with $i \in [s+1]$, by $x$ (where addition in the indices is taken modulo $2s+1$).
				\item
					Increase the weight of edges $v_{2i}v_{2i+1}$, with $i \in [s]$, by $x$.
				\item
					Increase the weight of the first, third, fifth, etc.\ edges of $P$ by $2x$ (starting from $v_1$).
				\item
					Decrease the weight of the second, fourth, sixth, etc.\ edges of $P$ by $2x$ (starting from $v_1$).
				\item
					If $P$ has odd length, decrease the weight of edges $u_{2i-1}u_{2i}$ with $i \in [t+1]$ by $x$ (where addition of indices is taken modulo $2t+1$). Otherwise, increase the weight of these edges by $x$.
				\item
					If $P$ has odd length, increase the weight of edges $u_{2i}u_{2i+1}$, with $i \in [t]$, by $x$. Otherwise, decrease the weight of these edges by $x$.
			\end{itemize}
			Every edge whose weight is modified belongs to $G'$ and has both endpoints incident with at least two positive-weight edges, so its weight under $w$ lies strictly between $0$ and $1$. Hence all weights in $w_x$ lie in $[0,1]$ for every sufficiently small $x>0$. Let $x'$ be the maximum such $x$. Then $w_{x'}$ is a perfect fractional matching of $G$ with fewer edges whose weight is neither $0$ nor $1$ compared to $w$, a contradiction.

			In conclusion, the assumption that $F$ is a component of $G'$ that is neither an odd cycle nor an edge leads to a contradiction, so $G'$ is a $2$-matching (and thus $G'$ is a perfect $2$-matching in $G$), as required.
		\end{proof}

	\section{Hamiltonicity of clusters} \label{sec:ham}

		\def \RN{\mathrm{RN}}

In this section, we present the proof of Lemma~\ref{lem:hamilton-cycle}. Our proof closely follows that of Lemma 4 in \cite{gruslys2021cycle}, but we include it here for completeness and to address differences in the lemma's statement.

The proof is based on known results concerning the Hamiltonicity of \emph{robust out-expanders}, a concept introduced by K\"uhn, Osthus and Treglown \cite{kuhn2010hamiltonian}. Before stating the relevant result, we first introduce some necessary definitions. Throughout this section, paths and cycles in digraphs are directed.

Given a digraph $G$ on $n$ vertices, a set of vertices $S$ and a parameter $\nu\in(0,1)$, the \emph{robust $\nu$-out-neighbourhood} of $S$ in $G$, denoted $\RN^+_{\nu,G}(S)$, is the set of vertices in $G$ that have at least $\nu n$ in-neighbours in $S$. Given $0<\nu\le\tau<1$, we say that $G$ is a \emph{robust $(\nu,\tau)$-out-expander} if $|\RN^+_{\nu,G}(S)|\ge|S|+\nu n$ for every set of vertices $S$ with $\tau n \le |S| \le (1 - \tau) n$.  We shall also use the following undirected version of a robust out-neighbourhood. In a graph $G$ on $n$ vertices, the \emph{robust $\nu$-neighbourhood} of a set of vertices $S$, denoted $\RN_{\nu,G}(S)$, is the set of vertices in $G$ with at least $\nu n$ neighbours in $S$. We say that $G$ is a \emph{robust $(\nu,\tau)$-expander} if $|\RN_{\nu,G}(S)|\ge|S|+\nu n$ for every set of vertices $S$ with $\tau n \le |S| \le (1 - \tau) n$.

We will need the following theorem from \cite{kuhn2010hamiltonian}. The minimum semidegree of a digraph $G$, denoted $\delta^0(G)$, is defined as $\delta^0(G) = \min\{ \delta^+(G), \delta^-(G)\}$, where $\delta^+(G), \delta^-(G)$ are the minimum out-degree and in-degree of $G$, respectively.
		\begin{theorem}[K\"uhn, Osthus and Treglown \cite{kuhn2010hamiltonian}] \label{thm:robust-expanders-hamilton}
			Let $n \in \N$ and let $ \nu, \tau, \gamma$ be reals such that $1 / n \ll \nu \le \tau \ll \gamma < 1$. Let $G$ be a digraph on $n$ vertices with $\delta^0(G) \ge \gamma n$ which is a robust $(\nu, \tau)$-out-expander. Then $G$ contains a Hamilton cycle.
		\end{theorem}

		We shall use the following corollary of \Cref{thm:robust-expanders-hamilton}.

		\begin{corollary} \label{cor:robust-expanders-hamilton}
			Let $ \nu, \tau, \gamma$ be reals and let $n$ be an integer such that $1 / n \ll \nu \le \tau \ll \gamma < 1$. Let $G$ be a digraph on $n$ vertices with $\delta^0(G) \ge \gamma n$ which is a robust $(\nu, \tau)$-out-expander. Then for every ordered pair of distinct vertices $x,y$, there is a directed Hamilton path in $G$ from $x$ to $y$.
		\end{corollary}

		\begin{proof}
			Given an ordered pair $x,y$, form $G'$ by adding the edge $yx$ to $G$, removing the edge $xy$ (if it exists), and removing every other edge directed towards $x$ or away from $y$; in particular, the edge $yx$ remains in $G'$.
			Next, form $G''$ by contracting the edge $yx$. Apart from the contracted vertex, each vertex loses at most one in-edge and at most one out-edge when $G'$ is formed and maintains its in- and out-degrees when $G''$ is formed, so $\delta^0(G'')\ge\gamma n-1\ge(\gamma/2)|V(G'')|$. For any set $S''\subseteq V(G'')$, replace the contracted vertex, when present, by $\{x,y\}$ to obtain a set $S\subseteq V(G)$. This changes the size of the set and every relevant in-neighbour count by at most two. Hence the robust $(\nu,\tau)$-out-expansion of $G$ implies that $G''$ is a robust $(\nu/2,2\tau)$-out-expander when $n$ is sufficiently large. Thus, by \Cref{thm:robust-expanders-hamilton}, $G''$ contains a directed Hamilton cycle. This corresponds to a directed Hamilton cycle in $G'$ containing the edge $yx$; deleting that edge gives a directed Hamilton path in $G$ from $x$ to $y$.
		\end{proof}

		\begin{proof} [ of Lemma~\ref{lem:hamilton-cycle}]
			Let $G$ be a graph with $n$ vertices, maximum degree at most $d$, and average degree at least $d - \eta n$. Suppose $G$ has no $\zeta$-sparse cuts and is either $\gamma$-far-from-bipartite or bipartite with a bipartition $\{X, Y\}$.
			Let $W \subseteq V(G)$ be a set of size at most $\xi n$, which satisfies $|X \setminus W| = |Y \setminus W|$ if $G$ is bipartite.
			Define $H = G \setminus W$, $V = V(H)$, and, if $G$ is bipartite, then also define $X' = X \setminus W$ and $Y' = Y \setminus W$.
			The graph $H$ has maximum degree at most $d$. Also, $d\le(1+\eta)n$ because the average degree of $G$ is at most $n-1$, and therefore
			\begin{equation*}
				2e(H)\ge n(d-\eta n)-2|W|d
				\ge |V(H)|d - (\eta + 2\xi(1+\eta))n^2
				\ge |V(H)|(d - 3\xi n),
			\end{equation*}
			where the penultimate inequality uses $d \le (1 + \eta)n$, $|W|\le\xi n$ and $\eta\ll\xi$, and the final inequality uses $|V(H)| \ge (1-\xi)n$ and $\eta \ll \xi$. Thus $H$ has average degree at least $d-3\xi n$. We now verify that $H$ has no $\zeta/2$-sparse cuts. Consider a partition $\{A,B\}$ of $V(H)$ with $|A|\le|B|$. Applying the assumption that $G$ has no $\zeta$-sparse cuts to each partition $\big\{\{v\},V(G)\setminus\{v\}\big\}$, where $v\in V(G)$, gives $\delta(G)\ge\zeta(n-1)$, so $\delta(H)\ge3\zeta n/4$ for sufficiently large $n$. If $|A|\le\zeta n/4$, then every vertex of $A$ has at least $\zeta n/2$ neighbours in $B$, and hence $e_H(A,B)\ge(\zeta/2)|A||B|$. Otherwise, $\{A,B\cup W\}$ is a partition of $V(G)$, so the assumption that $G$ has no $\zeta$-sparse cuts gives $e_G(A,B\cup W)\ge\zeta|A|(|B|+|W|)$. Since $e_H(A,B)=e_G(A,B\cup W)-e_G(A,W)$ and $e_G(A,W)\le|A||W|$, we obtain $e_H(A,B)\ge\zeta|A|(|B|+|W|)-|A||W|\ge(\zeta/2)|A||B|$. Here the final inequality follows from $|B|\ge|A|>\zeta n/4$, $|W|\le\xi n$, and $\xi\ll\zeta$. Hence $H$ has no $\zeta/2$-sparse cuts.

			The following claim will allow us to use \Cref{cor:robust-expanders-hamilton} above.

			\begin{claim} \label{claim:robust-nbd}
				Let $S\subseteq V$ be a set satisfying $\xi^{1/7}n\le|S|\le(1-\xi^{1/7})n$ if $G$ is $\gamma$-far-from-bipartite, or $\xi^{1/7}n\le|S|\le(1/2-\xi^{1/7})n$ if $G$ is bipartite. When $G$ is bipartite, $S$ is arbitrary and may meet both vertex classes.
				Then $|\RN_{\xi,H}(S)|\ge|S|+\xi n$.
			\end{claim}

			\begin{proof}
				Define $S_1 = S \setminus \RN_{\xi,H}(S)$, $S_2 = S \cap \RN_{\xi,H}(S)$, $T_1 = \RN_{\xi,H}(S) \setminus S$, $T_2 = V \setminus (S \cup T_1)$. Assume, towards a contradiction, that $|\RN_{\xi,H}(S)| < |S| + \xi n$, which implies that $|T_1| < |S_1| + \xi n$.

				Given sets $A, B \subseteq V$, denote by $e(A, B)$ the number of ordered pairs $ab$ such that $ab$ is an edge of $H$ and $a \in A, b \in B$.
				Then
				\begin{equation} \label{eqn:robust-a}
					e(S_1, V \setminus T_1) \le e(S_1, S) + e(S_1, T_2) \le (|S_1| + |T_2|) \cdot \xi n \le \xi n^2,
				\end{equation}
				using that the vertices in $S_1 \cup T_2$ are not in $\RN_{\xi,H}(S)$.
				By the maximum- and average-degree assumptions on $H$, the total degree deficit from $d$ is at most $3\xi n^2$. Equivalently,
				\begin{equation*}
					\sum_{v\in S_1}d_H(v)
					\ge |V|(d-3\xi n)-d(|V|-|S_1|)
					\ge d|S_1|-3\xi n^2.
				\end{equation*}
				Together with \eqref{eqn:robust-a}, this gives
				\begin{equation*}
					e(S_1,T_1)
				=\sum_{v\in S_1}d_H(v)-e(S_1,V\setminus T_1)
				\ge d|S_1|-4\xi n^2.
				\end{equation*}
				As $|T_1| < |S_1| + \xi n$, we obtain the following bound on $e(T_1, V \setminus S_1)$.
				\begin{equation} \label{eqn:robust-b}
					e(T_1, V \setminus S_1) \le d|T_1|  - e(S_1, T_1)
					\le d(|T_1| - |S_1|) + 4 \xi n^2
					\le 5 \xi n^2.
				\end{equation}
				Consider the quantity $e(S_1 \cup T_1, V \setminus (S_1 \cup T_1))$.
				By \eqref{eqn:robust-a} and \eqref{eqn:robust-b}, this quantity is at most $6 \xi n^2$, and since $H$ has no $\zeta/2$-sparse cuts, it is at least $(\zeta/2) \cdot |S_1 \cup T_1| \cdot |V \setminus (S_1 \cup T_1)|$.
				As $\xi \ll \zeta$, we find that either $|S_1 \cup T_1| \le \xi^{1/3} n$ or $|V \setminus (S_1 \cup T_1)| \le \xi^{1/3} n$.

				First, suppose that $|S_1 \cup T_1| \le \xi^{1/3} n$. Then
				\begin{equation*}
					e(S_2, V \setminus S_2) \le e(S_1 \cup T_1, V) + e(S_2, T_2) \le \xi^{1/3} n^2 + |T_2| \xi n \le 2\xi^{1/3}n^2,
				\end{equation*}
				using that $|S_1\cup T_1|\le\xi^{1/3}n$ and $T_2\cap\RN_{\xi,H}(S_2)=\emptyset$.
			Since $S_1\subseteq S_1\cup T_1$, the assumption in this case also gives $|S_1|\le|S_1\cup T_1|\le\xi^{1/3}n$. Moreover, $|S_2|=|S|-|S_1|\ge\xi^{1/7}n-|S_1|$ and
			\begin{equation*}
				|V\setminus S_2|\ge |V|-|S|\ge\xi^{1/7}n-|W|\ge\frac{\xi^{1/7}n}{2}.
			\end{equation*}
			Since $H$ has no $\zeta/2$-sparse cuts and $|S_1|\le\xi^{1/3}n\le\xi^{1/7}n/2$, it follows that
			\begin{equation*}
					e(S_2,V\setminus S_2)\ge(\zeta/2)|S_2||V\setminus S_2|
					\ge \frac{\zeta}{2}\cdot(\xi^{1/7}n-|S_1|)\cdot\frac{\xi^{1/7}n}{2}
					\ge \frac{\zeta}{2} \cdot \frac{\xi^{2/7}n^2}{4}
					> 2\xi^{1/3}n^2,
				\end{equation*}
				where the last inequality follows from $\xi\ll\zeta$, yielding a contradiction.

				Next, suppose that $|V \setminus (S_1 \cup T_1)| \le \xi^{1/3} n$. If $G$ is bipartite, then using that $|T_1| < |S_1| + \xi n$ and $|S| \le (1/2 - \xi^{1/7}) n$, we have
				\begin{align*}
					|V| = |V \setminus (S_1 \cup T_1)| + |S_1 \cup T_1|
					& < \xi^{1/3}n + |S_1| + (|S_1| + \xi n)  \\
					& \le \xi^{1/3}n+ n - 2\xi^{1/7}n + \xi n
					< n - \xi n \le n - |W|,
				\end{align*}
				a contradiction. Hence, we may assume that $G$ is $\gamma$-far-from-bipartite. Note that $G$ can be made bipartite by removing all edges incident to $W\cup S_2\cup T_2$, all edges within $S_1$, and all edges within $T_1$. There are at most $(\xi + \xi^{1/3})n^2$ edges incident to $W\cup S_2\cup T_2$, because $|W| \le \xi n$ and $|V \setminus (S_1 \cup T_1)| = |S_2 \cup T_2| \le \xi^{1/3} n$. Moreover, \eqref{eqn:robust-a} and \eqref{eqn:robust-b} show that there are at most $6 \xi n^2$ edges within $S_1$ or within $T_1$. Thus fewer than $\gamma n^2$ edges are removed in total (using $\xi \ll \gamma$). This contradicts the assumption that $G$ is $\gamma$-far-from-bipartite, thus completing the proof of the claim.
			\end{proof}

	            Let $x,y\in V$, where $x\in X'$ and $y\in Y'$ if $G$ is bipartite.
			Our goal is to show that $H$ contains a Hamilton path with ends $x$ and $y$. First, consider the case where $G$ is $\gamma$-far-from-bipartite. Form a digraph $D$ by replacing each edge $uv$ of $H$ with the two directed edges $uv$ and $vu$. It follows from \Cref{claim:robust-nbd} that $D$ is a robust $(\xi, 2\xi^{1/7})$-out-expander. Then, \Cref{cor:robust-expanders-hamilton} implies the existence of a Hamilton path in $D$ with ends $x$ and $y$, which corresponds to a Hamilton path in $H$ with the same ends.

			Now, suppose that $G$ is bipartite. Note that in this case $H$ is a balanced bipartite graph by our choice of $W$. We first show that $H$ has a perfect matching by verifying Hall's condition; that is, we show that $|N_H(S)|\ge|S|$ for every $S\subseteq X'$. Applying the assumption that $G$ has no $\zeta$-sparse cuts to each partition $\big\{\{v\},V(G)\setminus\{v\}\big\}$, where $v\in V(G)$, gives $\delta(G)\ge\zeta(n-1)$, and hence $\delta(H)\ge\zeta(n-1)-\xi n\ge3\zeta n/4$ for sufficiently large $n$.
			Thus, if $|S| < \zeta n / 2$ (and $S \neq \emptyset$), then $|N_H(S)| \ge 3\zeta n/4 \ge |S|$.
			If $|S| > |V|/2 - \zeta n / 2$, then every vertex $y\in Y'$ has a neighbour in $S$: otherwise all neighbours of $y$ would lie in $X'\setminus S$, which has size less than $\zeta n/2$, contradicting the minimum-degree bound. Thus $N_H(S)=Y'$ and $|N_H(S)|=|V|/2\ge|S|$.
			Hence, we may assume that $\zeta n / 2 \le |S| \le |V|/2 - \zeta n / 2$, which implies that $\xi^{1/7}n \le |S| \le (1/2 - \xi^{1/7})n$, but then \Cref{claim:robust-nbd} implies that $|N_H(S)| \ge |S|$, showing that $H$ has a perfect matching.

			Let $q=|X'|$. We next verify Hall's condition in $H-xy$. Its minimum degree is at least $3\zeta n/4-1\ge\zeta n/2$, so the preceding argument applies when $|S|<\zeta n/2$. If $\zeta n/2\le |S|\le |V|/2-\zeta n/2$, then $\xi^{1/7}n\le |S|\le(1/2-\xi^{1/7})n$, so \Cref{claim:robust-nbd} gives $|N_H(S)|\ge|S|+\xi n$. Deleting the edge $xy$ can remove at most one vertex from $N_H(S)$; hence $|N_{H-xy}(S)|\ge|S|$. Finally, if $|S|>|V|/2-\zeta n/2$ and $S\ne X'$, then $N_H(S)=Y'$ by the preceding paragraph, so $|N_{H-xy}(S)|\ge q-1\ge|S|$. If $S=X'$, the positive minimum degree of $H-xy$ gives $N_{H-xy}(S)=Y'$. Thus $H-xy$ has a perfect matching. Write this matching as $\{a_1b_1,\ldots,a_qb_q\}$, where $a_i\in X'$ and $b_i\in Y'$. Relabel its edges so that $a_1=x$ and $b_q=y$.
			Form an auxiliary directed graph $D$ with vertex set $\{v_1, \ldots, v_q\}$, where, for distinct $i,j\in[q]$, $v_i v_j$ is a directed edge whenever $b_i a_j$ is an edge of $H$. We verify its robust expansion explicitly. Let $S\subseteq V(D)$ satisfy $3\xi^{1/7}q\le|S|\le(1-3\xi^{1/7})q$, and set $B_S=\{b_i:v_i\in S\}$. Since $q=|V|/2$ and $|W|\le\xi n$, the size of $B_S$ lies between $\xi^{1/7}n$ and $(1/2-\xi^{1/7})n$. By \Cref{claim:robust-nbd}, at least $|S|+\xi n$ vertices $a_j\in X'$ have at least $\xi n$ neighbours in $B_S$. For every corresponding $v_j$, the in-degree from $S$ in $D$ is at least $\xi n-1\ge\xi q$, where the possible loss of one accounts for the omitted loop $v_jv_j$. Hence $D$ is a robust $(\xi,3\xi^{1/7})$-out-expander. Moreover, $\delta^0(D)=\delta(H)-1\ge\zeta n/2\ge(\zeta/2)q$. Since the matching lies in $H-xy$, we have $v_1\ne v_q$. Therefore \Cref{cor:robust-expanders-hamilton} gives a Hamilton path in $D$ starting at $v_1$ and ending at $v_q$. Relabel the internal matched edges according to this path. The resulting sequence corresponds to the Hamilton path $x=a_1,b_1,\ldots,a_q,b_q=y$ in $H$, as desired.
		\end{proof}
\end{document}